\documentclass[a4paper]{amsart}

\usepackage{amsmath,amssymb,dsfont,mathtools,nicefrac}
\usepackage[pdftitle={An equivariant Lefschetz fixed-point formula for correspondences},
  pdfauthor={Ivo Dell'Ambrogio, Heath Emerson, Ralf Meyer},
  pdfsubject={Mathematics}
]{hyperref}
\usepackage[lite]{amsrefs}
\usepackage{lmodern}
\usepackage[T1]{fontenc}
\usepackage{microtype}

\usepackage{tikz}
\usetikzlibrary{intersections,positioning,matrix,decorations.pathmorphing}
\tikzset{cd/.style={matrix of math nodes,row sep=2em,column sep=2.5em,text height=1.5ex,text depth=0.5ex}}
\tikzset{cdar/.style={->,auto,font=\scriptsize,decoration={snake,segment length=5pt,amplitude=1pt}}}
\tikzset{corr/.style={->,sloped,anchor=south,font=\scriptsize}}
\tikzset{odd/.style={decorate}}
\tikzset{blw/.style={anchor=north}}

\usepackage{xcolor}

\newcommand*{\MRref}[2]{ \href{http://www.ams.org/mathscinet-getitem?mr=#1}{MR \textbf{#1}}}
\newcommand*{\arxiv}[1]{\href{http://www.arxiv.org/abs/#1}{arXiv: #1}}
\renewcommand{\PrintDOI}[1]{\href{http://dx.doi.org/#1}{DOI #1}%
  \IfEmptyBibField{pages}{, (to appear in print)}{}}

\newtheorem{theorem}{Theorem}[section]
\newtheorem{corollary}[theorem]{Corollary}
\newtheorem{lemma}[theorem]{Lemma}
\newtheorem{proposition}[theorem]{Proposition}
\newtheorem{assumption}[theorem]{Assumption}
\theoremstyle{definition}
\newtheorem{definition}[theorem]{Definition}
\theoremstyle{remark}
\newtheorem{remark}[theorem]{Remark}
\newtheorem{example}[theorem]{Example}
\numberwithin{equation}{section}

\DeclareMathOperator{\Rep}{R}
\DeclareMathOperator{\Ind}{Ind}
\DeclareMathOperator{\Hom}{Hom}
\DeclareMathOperator{\End}{End}
\DeclareMathOperator{\Ad}{Ad}
\DeclareMathOperator{\tr}{tr}
\DeclareMathOperator{\coker}{coker}
\DeclareMathOperator{\ind}{ind}
\DeclareMathOperator{\lefind}{L-ind}
\DeclareMathOperator{\Eul}{Eul}
\DeclareMathOperator{\sign}{sign}

\newcommand*{\prid}{\mathfrak{p}}

\newcommand*{\intr}[1]{\mathring{#1}}
\newcommand*{\pt}{\textup{pt}}
\newcommand*{\pr}{\textup{pr}}
\newcommand*{\GKK}{\widehat{\textsc{kk}}}
\newcommand*{\Tang}{\textup T}
\newcommand*{\Nor}{\textup N}

\newcommand*{\Unit}{\mathds 1}
\newcommand*{\Un}{\textup U}

\newcommand*{\mono}{\rightarrowtail}
\newcommand*{\epi}{\twoheadrightarrow}
\newcommand*{\opem}{\hookrightarrow}
\newcommand*{\xopem}{\xhookrightarrow}

\newcommand*{\Boot}{\mathcal B}

\newcommand*{\Tri}{\mathcal T}
\newcommand*{\Type}{\mathcal T_1}
\newcommand*{\Cat}{\mathcal C}
\newcommand*{\KK}{\textup{KK}}

\newcommand*{\braid}{\textsf{braid}}

\newcommand*{\N}{\mathbb N}
\newcommand*{\Z}{\mathbb Z}
\newcommand*{\Q}{\mathbb Q}
\newcommand*{\R}{\mathbb R}
\newcommand*{\T}{\mathbb T}
\newcommand*{\C}{\mathbb C}
\newcommand*{\RK}{\textup{RK}}
\newcommand*{\K}{\textup K}
\newcommand*{\KO}{\textup{KO}}
\newcommand*{\Mat}{\mathbb M}
\newcommand*{\Comp}{\mathcal K}

\newcommand*{\Lef}{\mathcal L}
\newcommand*{\Ccat}{\mathcal{C}}

\newcommand*{\Cst}{\textup{C}^*}

\newcommand*{\Cont}{\textup C}

\newcommand*{\id}{\textup{id}}
\newcommand*{\loc}{\textup{loc}}

\newcommand*{\nb}{\nobreakdash}
\newcommand*{\defeq}{\mathrel{\vcentcolon=}}
\newcommand*{\blank}{\textup{\textvisiblespace}}

\newcommand*{\norm}[1]{\lVert#1\rVert}

\newcommand*{\thick}[1]{\langle#1\rangle}

\begin{document}
\title[Equivariant Lefschetz fixed-point formula]{An equivariant Lefschetz fixed-point formula\\for correspondences}

\author{Ivo Dell'Ambrogio}
\address{Laboratoire Paul Painlev\'e, Universit\'e de Lille~1, Cit\'e Scientifique - B\^atiment M2, F-59655 Villeneuve d'Ascq Cedex, France}
\email{ivo.dellambrogio@math.univ-lille1.fr}

\author{Heath Emerson}
\address{Department of Mathematics and Statistics, University
  of Victoria, PO BOX 3045 STN CSC, Victoria, B.\,C., Canada V8W\,3P4}
\email{hemerson@math.uvic.ca}

\author{Ralf Meyer}
\address{Mathematisches Institut and Courant Centre ``Higher order structures,'' Georg-August Universit\"at G\"ottingen, Bunsenstra{\ss}e 3--5, 37073 G\"ottingen, Germany}
\email{rameyer@uni-math.gwdg.de}

\thanks{This research was supported by the Volkswagen Foundation (Georgian--German non-commutative partnership).  Heath Emerson was supported by a National Science and Engineering Research Council of Canada Discovery Grant. Ralf Meyer was supported by the German Research Foundation (Deutsche Forschungsgemeinschaft (DFG)) through the Institutional Strategy of the University of G\"ottingen.}
\dedicatory{We dedicate this article to Tamaz Kandelaki, who was a coauthor in an earlier version of this article, and passed away in 2012. We will remember him for his warm character and his perseverance in doing mathematics in difficult circumstances.}

\begin{abstract}
  We compute the trace of an endomorphism in equivariant bivariant
  K-theory for a compact group~\(G\) in several ways: geometrically
  using geometric correspondences, algebraically using localisation,
  and as a Hattori--Stallings trace.  This results in an equivariant
  version of the classical Lefschetz fixed-point theorem, which
  applies to arbitrary equivariant correspondences, not just maps.
\end{abstract}

\maketitle

\section{Introduction}
\label{sec:intro}

Here we continue a series of articles by the last two authors
about Euler characteristics and Lefschetz invariants in equivariant
bivariant \(\K\)\nb-theory. These invariants were introduced in \cites{Emerson-Meyer:Euler,
  Emerson-Meyer:Equi_Lefschetz, Emerson-Meyer:Normal_maps,
  Emerson-Meyer:Correspondences, Emerson-Meyer:Dualities}.  The goal
is to compute Lefschetz invariants explicitly in a way that
generalises the Lefschetz--Hopf fixed-point formula.

Let~\(X\) be a smooth compact manifold and \(f\colon X\to X\) a
self-map with simple isolated fixed points.  The Lefschetz--Hopf
fixed-point formula identifies
\begin{enumerate}
\item \label{intro:geom_Lef} the sum over the fixed points of~\(f\), where each fixed point
  contributes~\(\pm1\) depending on its index;
\item \label{intro:homol_Lef} the supertrace of the \(\Q\)\nb-linear,
  grading-preserving map on \(\K^*(X)\otimes\Q\) induced
  by~\(f\).
\end{enumerate}
It makes no difference in~\eqref{intro:homol_Lef} whether we use rational
cohomology or \(\K\)\nb-theory because the Chern character is
an isomorphism between them.

We will generalise this result in two ways.  First, we allow a compact
group~\(G\) to act on~\(X\) and get elements of the representation
ring~\(\Rep(G)\) instead of numbers.  Secondly, we replace self-maps
by self-correspondences in the sense
of~\cite{Emerson-Meyer:Correspondences}.  Sections
\ref{sec:traces_GKK} and \ref{sec:trace_localisation} generalise the
invariants \eqref{intro:geom_Lef} and~\eqref{intro:homol_Lef} respectively to this setting.  The invariant of
Section~\ref{sec:traces_GKK} is local and geometric and
generalises~\eqref{intro:geom_Lef} above; the formulas in Sections
\ref{sec:trace_localisation} and~\ref{sec:Hattori-Stallings} are
global and homological and generalise~\eqref{intro:homol_Lef} (in two
different ways.) The equality of the geometric and homological
invariants is our generalisation of the Lefschetz fixed-point theorem.

A first step is to interpret the invariants \eqref{intro:geom_Lef}
or~\eqref{intro:homol_Lef} in a category-theoretic way in terms of the
trace of an endomorphism of a dualisable object in a symmetric
monoidal category.

Let~\(\Ccat\) be a symmetric monoidal category with tensor
product~\(\otimes\) and tensor unit~\(\Unit\).  An object~\(A\)
of~\(\Ccat\) is called \emph{dualisable} if there is an
object~\(A^*\), called its \emph{dual}, and a natural isomorphism
\[
\Ccat(A\otimes B,C) \cong \Ccat(B,A^*\otimes C)
\]
for all objects \(B\) and~\(C\) of~\(\Ccat\).  Such duality
isomorphisms exist if and only if there are two morphisms \(\eta\colon
\Unit\to A\otimes A^*\) and \(\varepsilon\colon A^*\otimes A\to
\Unit\), called unit and counit of the duality, that satisfy two
appropriate conditions.  Let \(f\colon A\to A\) be an endomorphism
in~\(\Cat\).  Then the \emph{trace} of~\(f\) is the composite
endomorphism
\[
\Unit \xrightarrow{\eta} A\otimes A^* \xrightarrow{\braid}
A^*\otimes A \xrightarrow{\id_{A^*}\otimes f}
A^*\otimes A \xrightarrow{\varepsilon} \Unit,
\]
where~\(\braid\) denotes the braiding isomorphism. In this article we
also call the trace the \emph{Lefschetz index} of the morphism. This
is justified by the following example.

Let~\(\Ccat\) be the Kasparov category~\(\KK\) with its usual tensor
product structure, \(A=\Cont(X)\) for a smooth compact manifold~\(X\),
and \(\hat{f}\in\KK_0(A,A)\) for some morphism. 
We may construct a
dual~\(A^*\) from the tangent bundle or the stable normal bundle
of~\(X\).  In the case of a smooth self-map of~\(X\), and assuming a
certain transversality condition, the trace of the morphism \(\hat{f}\) 
induced by
the self-map equals the invariant~\eqref{intro:geom_Lef}, that is,
equals the number of fixed-points of the map, counted with appropriate
signs.  This is checked by direct computation in Kasparov theory,
see~\cite{Emerson-Meyer:Equi_Lefschetz} for more general results.

This paper springs in part from the reference
\cite{Emerson-Meyer:Equi_Lefschetz}.  A similar invariant to the
Lefschetz index was introduced there, called the \emph{Lefschetz
  class} (of the morphism).  The Lefschetz class for an equivariant
Kasparov endomorphism of~\(X\) was defined as an equivariant
\(\K\)\nb-homology class for~\(X\).  The Lefschetz \emph{index}, that
is, the categorical trace, discussed above, is the Atiyah--Singer
index of the Lefschetz class of~\cite{Emerson-Meyer:Equi_Lefschetz}.

The main goal of this article is to give a global, homological formula
for the Lefschetz index generalising the
invariant~\eqref{intro:homol_Lef} for a non-equivariant self-map.  The
formulation and proof of our homological formula works best for
Hodgkin Lie groups.  A more complicated form applies to all compact
groups.  The article~\cite{Emerson-Meyer:Equi_Lefschetz} also provides
two formulas for the equivariant Lefschetz class whose equality
generalises that of the invariants \eqref{intro:geom_Lef}
and~\eqref{intro:homol_Lef}, but the methods there are completely
different.

The other main contribution of this article is to compute the
geometric expression for the Lefschetz index in the
category~\(\GKK^G\) of geometric correspondences introduced
in~\cite{Emerson-Meyer:Correspondences}.  This simplifies the
computation in Kasparov's analytic theory
in~\cite{Emerson-Meyer:Equi_Lefschetz} and also gives a more general
result, since we can work with general smooth correspondences rather
than just maps.  Furthermore, using an idea of Baum and Block
in~\cite{Baum-Block:Excess}, we give a recipe for composing two smooth
equivariant correspondences under a weakening of the usual
transversality assumption (of~\cite{Connes-Skandalis:Longitudinal}).
This technique is important for computing the Lefschetz index in the
case of continuous group actions, where transversality is sometimes
difficult to achieve, and in particular, aids in describing
equivariant Euler characteristics in our framework.

Section~\ref{sec:traces_GKK} contains our geometric formula for the
Lefschetz index of an equivariant self-correspondence.  Why is there a
nice geometric formula for the Lefschetz index of a self-map in
Kasparov theory?  A good explanation is that Connes and
Skandalis~\cite{Connes-Skandalis:Longitudinal} describe \(\KK\)-theory
for commutative \(\Cst\)\nb-algebras geometrically, including the
Kasparov product; furthermore, the unit and counit of the
\(\KK\)\nb-duality for smooth manifolds have a simple form in this
geometric variant of~\(\KK\).  An equivariant version of the theory
in~\cite{Connes-Skandalis:Longitudinal} is developed
in~\cite{Emerson-Meyer:Correspondences}.  In
Section~\ref{sec:traces_GKK}, we also recall some basic results about
the geometric \(\KK\)\nb-theory introduced
in~\cite{Emerson-Meyer:Correspondences}.  If~\(X\) is a smooth compact
\(G\)\nb-manifold for a compact group~\(G\), then
\(\KK^G_*(\Cont(X),\Cont(X))\) is isomorphic to the geometrically
defined group \(\GKK^G_*(X,X)\).  Its elements are smooth
\emph{correspondences}
\begin{equation}
  \label{eq:intro_smooth_correspondence}
  X\xleftarrow{b} (M,\xi) \xrightarrow{f} X
\end{equation}
consisting of a smooth \(G\)\nb-map~\(b\), a \(\K_G\)\nb-oriented
smooth \(G\)\nb-map~\(f\), and \(\xi\in \K^*_G(M)\).
Theorem~\ref{the:geo_trace_transversal} computes the categorical
trace, or Lefschetz index, of such a correspondence under suitable
assumptions on \(b\) and~\(f\).

Assume first that~\(X\) has no boundary and that \(b\) and~\(f\) are
transverse; equivalently, for all \(m\in M\) with \(f(m)=b(m)\) the
linear map \(Db - Df\colon \Tang_m M\to \Tang_{f(m)} X\) is
surjective.  Then
\begin{equation}
  \label{eq:coincidence_manifold}
  Q \defeq \{m\in M\mid b(m)=f(m)\}
\end{equation}
is naturally a \(\K_G\)\nb-oriented smooth manifold.  We show that the
Lefschetz index is the \(G\)\nb-index of the Dirac operator on~\(Q\)
twisted by \(\xi|_Q\in\K^*_G(Q)\)
(Theorem~\ref{the:geo_trace_transversal}).  More generally, suppose
that the coincidence space~\(Q\) as defined above is merely assumed to
be a smooth submanifold of~\(M\), and that \(x\in \Tang X\) and
\(Df(\xi)=Db(\xi)\) implies that \(\xi\in\Tang Q\).  Then we say that
\(f\) and~\(b\) \emph{intersect smoothly}.  For example, the identity
correspondence, where \(f\) and~\(b\) are the identity maps on~\(X\),
does not satisfy the above transversality hypothesis, but \(f\)
and~\(b\) clearly intersect smoothly.  In the case of a smooth
intersection, the cokernels of the map \(Df-Db\) form a vector bundle
on~\(Q\) which we call the \emph{excess intersection bundle}~\(\eta\).
This bundle measures the failure of transversality of \(f\) and~\(b\).
Let~\(\eta\) be \(\K_G\)\nb-oriented.  Then~\(\Tang Q\) also inherits
a canonical \(\K_G\)\nb-orientation.  The restriction of the Thom
class of~\(\eta\) to the zero section gives a class
\(e(\eta)\in\K^*_G(Q)\).

Then Theorem~\ref{the:geo_trace_transversal} asserts that the
Lefschetz index of the
correspondence~\eqref{eq:intro_smooth_correspondence} with smoothly
intersecting \(f\) and~\(b\) is the index of the Dirac operator on the
coincidence manifold~\(Q\) twisted by \(\xi\otimes e(\eta)\).  This is
the main result of Section~\ref{sec:traces_GKK}.

In Section~\ref{sec:trace_localisation} we generalise the global
homological formula involved in the classical Lefschetz fixed-point
theorem, to the equivalent situation.  This involves completely
different ideas.  The basic idea to use K\"unneth and Universal
Coefficient theorems for such a formula already appears
in~\cite{Emerson:Lefschetz}.  In the equivariant case, these theorems
become much more complicated, however.  The new idea that we need here
is to first localise~\(\KK^G\) and compute the Lefschetz index in the
localisation.

In the introduction, we only state our result in the simpler case of a
Hodgkin Lie group~\(G\).  Then~\(\Rep(G)\) is an integral domain and
thus embeds into its field of fractions~\(F\).  For any
\(G\)\nb-\(\Cst\)-algebra~\(A\), \(\K^G_*(A)\) is a \(\Z/2\)-graded
\(\Rep(G)\)-module.  Thus \(\K^G_*(A;F) \defeq
\K^G_*(A)\otimes_{\Rep(G)} F\) becomes a \(\Z/2\)-graded
\(F\)\nb-vector space.  Assume that~\(A\) is dualisable and belongs to
the bootstrap class in~\(\KK^G\).  Then \(\K^G_*(A;F)\) is
finite-dimensional, so that the map on \(\K^G_*(A;F)\) induced by an
endomorphism \(\varphi\in\KK^G_0(A,A)\) has a well-defined
(super)trace in~\(F\).  Theorem~\ref{the:trace_on_genUnit} asserts
that this supertrace belongs to \(\Rep(G)\subseteq F\) and is equal to
the Lefschetz index of~\(\varphi\).  In particular, this applies if
\(A=\Cont(X)\) for a compact \(G\)\nb-manifold.

The results of Sections \ref{sec:traces_GKK}
and~\ref{sec:trace_localisation} together thus prove the following:

\begin{theorem}
  \label{the:intro_Lefschetz-Hopf}
  Let~\(G\) be a Hodgkin Lie group, let~\(F\) be the field of
  fractions of~\(\Rep(G)\).  Let~\(X\) be a closed
  \(G\)\nb-manifold.  Let \(X \xleftarrow{b} (M, \xi)
  \xrightarrow{f} X\) be a smooth \(G\)\nb-equivariant
  correspondence from~\(X\) to~\(X\) with \(\xi\in\K^{\dim
    M-\dim X}_G(X)\); it represents a class \(\varphi\in
  \GKK^G_0(X,X)\).  Assume that \(b\) and~\(f\) intersect
  smoothly with \(\K_G\)\nb-oriented excess intersection
  bundle~\(\eta\).  Equip \(Q\defeq \{m \in M \mid b(m) =
  f(m)\}\) with its induced \(\K_G\)\nb-orientation.

  Then the \(\Rep(G)\)-valued index of the Dirac operator
  on~\(Q\) twisted by \(\xi|_Q\otimes e(\eta)\) is equal to the
  supertrace of the \(F\)\nb-linear map on
  \(\K^*_G(X)\otimes_{\Rep(G)} F\) induced by~\(\varphi\).
\end{theorem}

If~\(G\) is a connected Lie group, then there is a finite covering
\(\hat{G}\epi G\) that is a Hodgkin Lie group.  We may turn
\(G\)\nb-actions into \(\hat{G}\)\nb-actions using the projection map,
and get a symmetric monoidal functor \(\KK^G\to\KK^{\hat{G}}\).  Since
the map \(\Rep(G)\to\Rep(\hat{G})\) is clearly injective, we may
compute the Lefschetz index of \(\varphi\in\KK^G_0(A,A)\) by computing
instead the Lefschetz index of the image of~\(\varphi\)
in~\(\KK^{\hat{G}}_0(A,A)\).  By the result mentioned above, this uses
the induced map on \(\K^{\hat{G}}_*(A)\otimes_{\Rep(\hat{G})}
\hat{F}\), where~\(\hat{F}\) is the field of fractions of
\(\Rep(\hat{G})\).  Thus we get a satisfactory trace formula for all
connected Lie groups.  But the result may be quite different from the
trace of the induced map on \(\K^G_*(A)\otimes_{\Rep(G)} F\).

If~\(G\) is not connected, then the total ring of fractions
of~\(G\) is a product of finitely many fields.  Its factors
correspond to conjugacy classes of Cartan subgroups in~\(G\).
Each Cartan subgroup \(H\subseteq G\) corresponds to a minimal
prime ideal~\(\prid_H\) in~\(\Rep(G)\).  The
quotient~\(\Rep(G)/\prid_H\) is an integral domain and embeds
into a field of fractions~\(F_H\).  We show that the map
\(\Rep(G)\to F_H\) maps the Lefschetz index of~\(\varphi\) to the
supertrace of \(\K^H_*(\varphi;F_H)\)
(Theorem~\ref{the:trace_compact_Lie_group}).  It is crucial to
use \(H\)\nb-equivariant \(\K\)\nb-theory here.  The very
simple counterexample~\ref{exa:pt_does_not_generate_for_Ztwo}
shows that there may be two elements
\(\varphi_1,\varphi_2\in\KK^G_0(A,A)\) with different Lefschetz index
but inducing the same map on~\(\K^G_*(A)\).

Thus the generalisation of
Theorem~\ref{the:intro_Lefschetz-Hopf} to disconnected~\(G\)
identifies the image of the index of the Dirac operator
on~\(Q\) twisted by \(\xi|_Q\otimes e(\eta)\) under the
canonical map \(\Rep(G)\to F_H\) with the supertrace of the
\(F_H\)\nb-linear map on \(\K^*_G(X)\otimes_{\Rep(G)} F_H\)
induced by~\(\varphi\), for each Cartan subgroup~\(H\).

The trace formulas in Section~\ref{sec:trace_localisation} require the
algebra~\(A\) on which we compute the trace to be dualisable and to
belong to an appropriate bootstrap class, namely, the class of all
\(G\)\nb-\(\Cst\)-algebras that are \(\KK^G\)-equivalent to a type~I
\(G\)\nb-\(\Cst\)-algebra.  This is strictly larger than the class of
\(G\)\nb-\(\Cst\)-algebras that are \(\KK^G\)-equivalent to a
commutative one, already if~\(G\) is the circle group
(see~\cite{Emerson:Localization_circle}).  We describe the bootstrap
class as being generated by so-called elementary
\(G\)\nb-\(\Cst\)-algebras in Section~\ref{sec:equiv_bootstrap}.  This
list of generators is rather long, but for the purpose of the trace
computations, we may localise~\(\KK^G\) at the multiplicatively closed
subset of non-zero divisors in~\(\Rep(G)\).  The image of the
bootstrap class in this localisation has a very simple structure,
which is described in Section~\ref{sec:localise_bootstrap}.  The
homological formula for the Lefschetz index follows easily from this
description of the localised bootstrap category.

In Section~\ref{sec:Hattori-Stallings}, we give a variant of the
global homological formula for the trace for a Hodgkin Lie
group~\(G\).  Given a commutative ring~\(R\) and an
\(R\)\nb-module~\(M\) with a projective resolution of finite type, we
may define a Hattori--Stallings trace for endomorphisms of~\(M\) by
lifting the endomorphism to a finite type projective resolution and
using the standard trace for endomorphisms of finitely generated
projective resolutions.  This defines the trace of the
\(\Rep(G)\)-module homomorphism \(\K^G_*(\varphi)\colon
\K^G_*(A)\to\K^G_*(A)\) in~\(\Rep(G)\) without passing through a field
of fractions.

\section{Lefschetz indices in geometric bivariant K-theory}
\label{sec:traces_GKK}

The category~\(\GKK^G\) introduced
in~\cite{Emerson-Meyer:Correspondences} provides a geometric
analogue of Kasparov theory.  We first recall some basic facts
about this category and duality in bivariant \(\K\)\nb-theory
from \cites{Emerson-Meyer:Dualities, Emerson-Meyer:Normal_maps,
  Emerson-Meyer:Correspondences} and then compute Lefschetz
indices in it as intersection products.  Later we are going to
compare this with other formulas for Lefschetz indices.  We
also prove an excess intersection formula to compute the
composition of geometric correspondences under a weaker
assumption than transversality.  This formula goes back to Baum
and Block~\cite{Baum-Block:Excess}.

All results in this section extend to the case where~\(G\) is
a proper Lie groupoid with enough \(G\)\nb-vector bundles in the sense
of \cite{Emerson-Meyer:Normal_maps}*{Definition 3.1} because the
theory in \cites{Emerson-Meyer:Dualities, Emerson-Meyer:Normal_maps,
  Emerson-Meyer:Correspondences} is already developed in this
generality.  For the sake of concreteness, we limit our treatment here
to compact Lie groups acting on smooth manifolds.

The results in this section work both for real and complex
\(\K\)\nb-theory.  For concreteness, we assume in our notation that we
are dealing with the complex case.  In the real case, \(\K\) must be
replaced by \(\KO\) throughout.  In particular,
\(\K_G\)\nb-orientations (that is, \(G\)\nb-equivariant Spin\(^\textup
c\)-structures) must be replaced by \(\KO^G\)\nb-orientations (that
is, \(G\)\nb-equivariant Spin structures).  In some examples, we use
the isomorphisms \(\GKK^G_{2n}(\pt,\pt)=\Rep(G)\) and
\(\GKK^G_{2n+1}(\pt,\pt)=0\) for all \(n\in\Z\).  Here \(\Rep(G)\)
denotes the representation ring of~\(G\).  Of course, this is true
only in complex \(\K\)\nb-theory.

\subsection{Geometric bivariant K-theory}
\label{sec:geom_KK}

Like Kasparov theory, geometric bivariant K\nb-theory yields
a category \(\GKK^G\).  Its objects are (Hausdorff) locally compact
\(G\)\nb-spaces; arrows from~\(X\) to~\(Y\) are \emph{geometric
  correspondences} from~\(X\) to~\(Y\) in the sense of
\cite{Emerson-Meyer:Correspondences}*{Definition 2.3}.  These
consist of
\begin{description}
\item[\(M\)] a \(G\)\nb-space;
\item[\(b\)] a \(G\)\nb-map (that is, a continuous
  \(G\)\nb-equivariant map) \(b\colon M\to X\);
\item[\(\xi\)] a \(G\)\nb-equivariant \(\K\)\nb-theory class
  on~\(M\) with \(X\)\nb-compact support (where we view~\(M\)
  as a space over~\(X\) via the map~\(b\)); we write
  \(\xi\in\RK^*_{G,X}(M)\);
\item[\(f\)] a \(\K_G\)\nb-oriented normally non-singular
  \(G\)\nb-map \(f\colon M\to Y\).
\end{description}
Equivariant \(\K\)\nb-theory with \(X\)\nb-compact support and
equivariant vector bundles are defined in
\cite{Emerson-Meyer:Equivariant_K}*{Definitions 2.5 and~2.6}.
If~\(b\) is a proper map, in particular if~\(M\) is compact, then
\(\RK^*_{G,X}(M)\) is the ordinary \(G\)\nb-equivariant (compactly
supported) \(\K\)-theory \(\K^*_G(M)\) of~\(M\).

A \emph{\(\K_G\)\nb-oriented normally non-singular map}
from~\(M\) to~\(Y\) consists of
\begin{description}
\item[\(V\)] a \(\K_G\)\nb-oriented \(G\)\nb-vector bundle on~\(M\),
\item[\(E\)] a \(\K_G\)\nb-oriented finite-dimensional linear
  \(G\)\nb-representation, giving rise to a trivial
  \(\K_G\)\nb-oriented \(G\)\nb-vector bundle \(Y\times E\)
  on~\(Y\),
\item[\(\hat{f}\)] a \(G\)\nb-equivariant homeomorphism from
  the total space of~\(V\) to an open subset in the total space
  of \(Y\times E\), \(\hat{f}\colon V\opem Y\times E\).
\end{description}
We will not distinguish between a vector bundle and its total
space in our notation.

A normally non-singular map \(f=(V,E,\hat{f})\) has an
\emph{underlying map}
\[
M\mono V\xopem{\hat{f}} Y\times E\epi Y,
\]
where the first map is the zero section of the vector
bundle~\(V\) and the third map is the coordinate projection.
This map is called its ``trace''
in~\cite{Emerson-Meyer:Normal_maps}, but we avoid this name
here because we use ``trace'' in a different sense.
The \emph{degree} of~\(f\) is \(d = \dim V-\dim E\).
A wrong-way element \(f_!\in\KK^G_d(\Cont_0(M),\Cont_0(Y))\)
induced by~\(f\) is defined in
\cite{Emerson-Meyer:Normal_maps}*{Section 5.3}).

Our geometric correspondences are variants of those introduced
by Alain Connes and Georges Skandalis
in~\cite{Connes-Skandalis:Longitudinal}.  The changes in the
definition avoid technical problems with the usual definition
in the equivariant case.

The (\(\Z/2\)-graded) geometric \(\KK\)-group \(\GKK^G_*(X,Y)\) is
defined as the quotient of the set of geometric correspondences
from~\(X\) to~\(Y\) by an appropriate equivalence relation, generated
by bordism, Thom modification, and equivalence of normally
non-singular maps.  Bordism includes homotopies for the maps \(b\)
and~\(f\) by \cite{Emerson-Meyer:Correspondences}*{Lemma~2.12}.  We
will use this several times below.  The Thom modification allows to
replace the space~\(M\) by the total space of a \(\K_G\)\nb-oriented
vector bundle on~\(M\).  In particular, we could take the
\(\K_G\)\nb-oriented vector bundle from the normally non-singular
map~\(f\).  This results in an equivalent normally non-singular map
where \(f\colon M\to Y\) is a \emph{special submersion}, that is, an
open embedding followed by a coordinate projection \(Y\times E\epi Y\)
for some linear \(G\)\nb-representation~\(E\).  Correspondences with
this property are called \emph{special}.

The composition in~\(\GKK^G\) is defined as an intersection
product (see Section~\ref{subsec:trans}) if the map \(f\colon
M\to Y\) is such a special submersion.  This turns~\(\GKK^G\)
into a category; the identity map on~\(X\) is the
correspondence with \(f=b=\id_X\) and \(\xi=1\).  The product
of \(G\)\nb-spaces provides a symmetric monoidal structure
in~\(\GKK^G\) (see
\cite{Emerson-Meyer:Correspondences}*{Theorem 2.27}).

There is an additive, grading-preserving, symmetric monoidal
functor
\[
\GKK^G_*(X,Y) \to \KK^G_*(\Cont_0(X),\Cont_0(Y)).
\]
This is an isomorphism if~\(X\) is \emph{normally non-singular}
by \cite{Emerson-Meyer:Correspondences}*{Corollary 4.3}, that
is, if there is a normally non-singular map \(X\to\pt\).  This
means that there is a \(G\)\nb-vector bundle~\(V\) over~\(X\)
whose total space is \(G\)\nb-equivariantly homeomorphic to an
open \(G\)\nb-invariant subset of some linear \(G\)\nb-space.
In particular, by Mostow's Embedding Theorem smooth
\(G\)\nb-manifolds of finite orbit type are normally
non-singular (see~\cite{Emerson-Meyer:Normal_maps}*{Theorem
  3.22}).

Stable \(\K_G\)\nb-orientations play an important technical role in
our trace formulas and should therefore be treated with care.  A
\emph{\(\K_G\)\nb-orientation} on a \(G\)\nb-vector bundle~\(V\) is,
by definition, a \(G\)\nb-equivariant complex spinor bundle for~\(V\).
(This is equivalent to a reduction of the structure group to
\(\textup{Spin}^\textup c\).)  Given such \(\K_G\)\nb-orientations on
\(V_1\) and~\(V_2\), we get an induced \(\K_G\)\nb-orientation on
\(V_1\oplus V_2\); conversely, \(\K_G\)\nb-orientations on \(V_1\oplus
V_2\) and~\(V_1\) induce one on~\(V_2\).

Let \(\xi\in\RK^0_G(M)\) be represented by the formal difference
\([V_1]-[V_2]\) of two \(G\)\nb-vector bundles.  A \emph{stable
  \(\K_G\)\nb-orientation} on~\(\xi\) means that we are given another
\(G\)\nb-vector bundle~\(V_3\) and \(\K_G\)\nb-orientations on both
\(V_1\oplus V_3\) and \(V_2\oplus V_3\).  Since \(\xi=[V_1\oplus
V_3]-[V_2\oplus V_3]\), this implies that~\(\xi\) is a formal
difference of two \(\K_G\)\nb-oriented \(G\)\nb-vector bundles.
Conversely, assume that \(\xi=[W_1]-[W_2]\) with two
\(\K_G\)\nb-oriented \(G\)\nb-vector bundles; then there are
\(G\)\nb-vector bundles \(V_3\) and~\(W_3\) such that \(V_i\oplus
V_3\cong W_i\oplus W_3\) for \(i=1,2\); since~\(W_3\) is a direct summand
in a \(\K_G\)\nb-oriented \(G\)\nb-vector bundle, we may enlarge
\(V_3\) and~\(W_3\) so that~\(W_3\) itself is \(\K_G\)\nb-oriented.
Then \(V_i\oplus V_3\cong W_i\oplus W_3\) for \(i=1,2\) inherit
\(\K_G\)\nb-orientations.  Roughly speaking, stably
\(\K_G\)\nb-oriented \(\K\)\nb-theory classes are equivalent to formal
differences of \(\K_G\)\nb-oriented \(G\)\nb-vector bundles.

A \(\K_G\)\nb-orientation on a normally non-singular map \(f = (V, E,
\hat{f})\) from~\(M\) to~\(Y\) means that both \(V\) and~\(E\) are
\(\K_G\)\nb-oriented.  Since ``lifting'' allows us to replace~\(E\) by
\(E\oplus E'\) and \(V\) by \(V\oplus (M\times E')\), we may assume
without loss of generality that~\(E\) is already \(\K_G\)\nb-oriented.
Thus a \(\K_G\)\nb-orientation on~\(f\) becomes equivalent to one
on~\(V\).  But the chosen \(\K_G\)\nb-orientation on~\(E\) remains
part of the data: changing it changes the \(\K_G\)\nb-orientation
on~\(f\).  By \cite{Emerson-Meyer:Normal_maps}*{Lemma 5.13}, all
essential information is contained in a \(\K_G\)\nb-orientation on the
formal difference \([V]-[M\times E]\in\RK^0_G(M)\), which we call the
\emph{stable normal bundle} of the normally non-singular map~\(f\).
If \([V]-[M\times E]\) is \(\K_G\)\nb-oriented, then we may find a
\(G\)\nb-vector bundle~\(V_3\) such that \(V\oplus V_3\) and
\((M\times E)\oplus V_3\) are \(\K_G\)\nb-oriented.  Since \((M\times
E)\oplus V_3\) is a direct summand in a \(\K_G\)\nb-oriented trivial
\(G\)\nb-vector bundle, we may assume without loss of generality
that~\(V_3\) itself is trivial, \(V_3=M\times E'\), and that already
\(E\oplus E'\) is \(\K_G\)\nb-oriented.  Lifting~\(f\) along~\(E'\)
then gives a normally non-singular map \((V\oplus (M\times E'),
E\oplus E',\hat{f}\times\id_{E'})\), where both \(V\oplus (M\times
E')\) and \(E\oplus E'\) are \(\K_G\)\nb-oriented.  Thus a
\(\K_G\)\nb-orientation on~\(f\) is equivalent to a stable
\(\K_G\)\nb-orientation on the stable normal bundle of~\(f\).

\begin{lemma}
  \label{lem:normal_bundle_smooth}
  If \(f=(V,E,\hat{f})\) is a smooth normally non-singular map with
  underlying map \(\bar{f}\colon M\to Y\), then its stable normal
  bundle is equal to \(\bar{f}^*[\Tang Y] -[\Tang M] \in
  \RK^0_G(M)\).
\end{lemma}

\begin{proof}
  The tangent bundles of the total spaces of~\(V\) and \(Y\times E\)
  are \(\Tang M\oplus V\) and \(\Tang Y\oplus E\), respectively.
  Since~\(\hat{f}\) is an open embedding, \(\hat{f}^*(\Tang Y\oplus
  E)\cong \Tang M\oplus V\).  This implies \(\bar{f}^*(\Tang Y)\oplus
  (M\times E) \cong \Tang M\oplus V\).  Thus \([V]-[M\times E] =
  \bar{f}^*[\Tang Y]-[\Tang M]\).
\end{proof}

This lemma also shows that the stable normal bundle of~\(f\) and hence
the orientability assumption depend only on the equivalence class
of~\(f\).

Another equivalent way to describe stable \(\K_G\)\nb-orientations is
the following.  Suppose we are already given a \(G\)\nb-vector
bundle~\(W\) on~\(Y\) such that \(\Tang Y\oplus V\) is
\(\K_G\)\nb-oriented.  Then a stable \(\K_G\)\nb-orientation on~\(f\)
is equivalent to one on
\[
[\bar{f}^*V\oplus \Tang M]
= \bar{f}^*[\Tang Y\oplus V] - (\bar{f}^*[\Tang Y]-[\Tang M]),
\]
which is equivalent to a \(\K_G\)\nb-orientation on \(\bar{f}^*V\oplus
\Tang M\) in the usual sense.

If \(X\) and~\(Y\) are smooth \(G\)\nb-manifolds (without
boundary), we may require the maps \(b\) and~\(\hat{f}\) and
the vector bundles \(V\) and~\(E\) to be smooth.  This leads to
a smooth variant of~\(\GKK^G\).  This variant is isomorphic to
the one defined above by
\cite{Emerson-Meyer:Correspondences}*{Theorem 4.8}
provided~\(X\) is of finite orbit type and hence normally
non-singular.

Working in the smooth setting has two advantages.

First, assuming~\(M\) to be of finite orbit type,
\cite{Emerson-Meyer:Normal_maps}*{Theorem 3.22} shows that any
smooth \(G\)\nb-map \(f\colon M\to Y\) lifts to a
smooth normally non-singular map that is unique up to
equivalence.  Thus we may replace normally non-singular maps by
smooth maps in the usual sense in the definition of a geometric
correspondence.  Moreover, \(Nf = f^*[\Tang Y] - [\Tang M]\), so \(f\) is
 \(\K_G\)-oriented if and only if there are \(\K_G\)\nb-oriented
\(G\)\nb-vector
bundles \(V_1\) and~\(V_2\) over~\(M\) with \(f^*[\Tang
Y]\oplus V_1 \cong \Tang M \oplus V_2\) (compare
\cite{Emerson-Meyer:Normal_maps}*{Corollary 5.15}).

Secondly, in the smooth setting there is a particularly elegant
way of composing correspondences when they satisfy a suitable
transversality condition, see
\cite{Emerson-Meyer:Correspondences}*{Corollary 2.39}.  This
description of the composition is due to Connes and
Skandalis~\cite{Connes-Skandalis:Longitudinal}.

\subsection{Composition of geometric correspondences}
\label{subsec:trans}

By \cite{Emerson-Meyer:Correspondences}*{Theorem 2.38}, a smooth
normally non-singular map lifting \(f\colon M_1\to Y\) and a smooth
map \(b\colon M_2\to Y\) are \emph{transverse} if
\[
D_{m_1} f(\Tang_{m_1}M_1) + D_{m_2} b(\Tang_{m_2} M_2)
= \Tang_y Y
\]
for all \(m_1\in M_1\), \(m_2\in M_2\) with \(y\defeq
f(m_1)=b(m_2)\).  Equivalently, the map
\[
Df - Db\colon \pr_1^*(\Tang M_1)\oplus \pr_2^*(\Tang M_2) \to
(f\circ \pr_1)^*(\Tang Y)
\]
is surjective; this is a bundle map of vector bundles over
\[
M_1\times_Y M_2 \defeq \{ (m_1, m_2) \mid f(m_1) = b(m_2)\},
\]
where \(\pr_1\colon M_1 \times_Y M_2\to M_1\) and \(\pr_2\colon
M_1\times_YM_2\to M_2\) denote the restrictions to
\(M_1\times_Y M_2\) of the coordinate projections.  (We shall
always use this notation for restrictions of coordinate
projections.)

A commuting square diagram of smooth manifolds is called
\emph{Cartesian} if it is isomorphic (as a diagram) to a square
\[
\begin{tikzpicture}
  \matrix (m) [cd]{
    M_1\times_Y M_2 &M_2\\M_1&Y\\};
  \begin{scope}[cdar]
    \draw (m-1-1)-- node {$\pr_2$} (m-1-2);
    \draw (m-1-1)-- node[swap] {$\pr_1$} (m-2-1);
    \draw (m-1-2)-- node {$f$} (m-2-2);
    \draw (m-2-1)-- node {$b$} (m-2-2);
  \end{scope}
\end{tikzpicture}
\]
where \(f\) and~\(b\) are transverse smooth maps in the sense above;
then \(M_1\times_Y M_2\) is again a smooth manifold and \(\pr_1\)
and~\(\pr_2\) are smooth maps.

The tangent bundles of these four manifolds are related by an exact
sequence
\begin{equation}
  \label{es1}
  0 \rightarrow\Tang (M_1\times_Y M_2) \xrightarrow{(D\pr_1,D\pr_2)}
  \pr_1^*(\Tang M_1) \oplus \pr_2^*(\Tang M_2)
  \xrightarrow{Df -Db}
  (f\circ \pr_1)^*\Tang Y \rightarrow 0.
\end{equation}
That is, \(\Tang (M_1\times_Y M_2)\) is the sub-bundle of
\(\pr_1^*(\Tang M_1) \oplus \pr_2^*(\Tang M_2)\) consisting of those
vectors \((m_1,\xi, m_2, \eta) \in \Tang M_1\oplus \Tang M_2\) (where
\(f(m_1) = b(m_2)\)) with \(D_{m_1}f(\xi)= D_{m_2} b(\eta)\).  We may
denote this bundle briefly by \(\Tang M_1\oplus_{\Tang Y} \Tang M_2\).

Furthermore, from~\eqref{es1},
\begin{equation}
  \label{es200}
  \Tang (M_1\times_Y M_2) -
  \pr_2^*(\Tang M_2) = \pr_1^*(\Tang M_1 - f^*(\Tang Y))
\end{equation}
as stable \(G\)\nb-vector bundles.  Thus a stable
\(\K_G\)\nb-orientation for \(\Tang M_1 - f^*(\Tang Y)\) may be pulled
back to one for \(\Tang (M_1\times_Y M_2) - \pr_2^*(\Tang M_2)\).
More succinctly, a \(\K_G\)\nb-orientation for the map~\(f\) induces
one for~\(\pr_2\).

Now consider two composable smooth correspondences
\begin{equation}
  \label{csm}
  \begin{tikzpicture}[baseline=(current bounding box.west),column sep=3em]
    \matrix (m) [cd]{
      & M_1 && M_2\\
      X&&Y&&Z,\\};
    \begin{scope}[corr]
      \draw (m-1-2)-- node[swap] {$b_1$} (m-2-1);
      \draw (m-1-2)-- node {$f_1$} (m-2-3);
      \draw (m-1-4)-- node[swap] {$b_2$} (m-2-3);
      \draw (m-1-4)-- node {$f_2$} (m-2-5);
    \end{scope}
  \end{tikzpicture}
\end{equation}
with \(\K\)\nb-theory classes \(\xi_1\in \RK^G_{*,X}(M_1)\) and
\(\xi_2\in \RK^G_{*,Y}(M_2)\).  We assume that the pair of smooth maps
\((f_1,b_2)\) is transverse.  Then there is an essentially unique
commuting diagram
\begin{equation}
  \label{eq:intersection_diagram}
  \begin{tikzpicture}[baseline=(current bounding box.west),column sep=3em]
    \matrix (m) [cd]{
      &&M_1\times_Y M_2\\
      & M_1 && M_2\\
      X&&Y&&Z,\\};
    \begin{scope}[corr]
      \draw (m-1-3)-- node[swap] {$\pr_1$} (m-2-2);
      \draw (m-1-3)-- node {$\pr_2$}       (m-2-4);
      \draw (m-2-2)-- node[swap] {$b_1$} (m-3-1);
      \draw (m-2-2)-- node {$f_1$}       (m-3-3);
      \draw (m-2-4)-- node[swap] {$b_2$} (m-3-3);
      \draw (m-2-4)-- node {$f_2$}       (m-3-5);
    \end{scope}
  \end{tikzpicture}
\end{equation}
where the square is Cartesian.  We briefly call such a diagram an
\emph{intersection diagram} for the two given correspondences.

By the discussion above, the map~\(\pr_2\) inherits a
\(\K_G\)\nb-orientation from~\(f_1\), so
that the map \(f\defeq f_2\circ \pr_2\) is also
\(\K_G\)\nb-oriented.  Let \(M\defeq
M_1\times_Y M_2\) and \(b\defeq
b_1\circ\pr_1\).  The product \(\xi\defeq \pr_1^*(\xi_1)\otimes
\pr_2^*(\xi_2)\) belongs to~\(\RK^G_{*,X}(M)\), that is, it has
\(X\)\nb-compact support with respect to the map \(b\colon M\to
X\).  Thus we get a \(G\)\nb-equivariant correspondence
\((M,b,f,\xi)\) from~\(X\) to~\(Y\).  The assertion of
\cite{Emerson-Meyer:Correspondences}*{Corollary 2.39} --
following~\cite{Connes-Skandalis:Longitudinal} -- is that this
represents the composition of the two given correspondences.
It is called their \emph{intersection product}.

\begin{example}
  \label{exa:compose_diagonal}
  Consider the diagonal embedding \(\delta\colon X\to X\times X\) and
  the graph embedding \(\bar{f}\colon X\to X\times X\), \(x\mapsto (x,
  f(x))\), for a smooth map \(f\colon X\to X\).  These two maps are
  transverse if and only if~\(f\) has simple fixed points.  If this is
  the case, then the intersection space is the set of fixed points
  of~\(f\).  If, say, \(f=\id_X\), then \(\delta\) and~\(\bar{f}\) are
  not transverse.
\end{example}

To define the composition also in the non-transverse case, a Thom
modification is used in~\cite{Emerson-Meyer:Correspondences} to
achieve transversality (see
\cite{Emerson-Meyer:Correspondences}*{Theorem 2.32}).  Take two
composable (smooth) correspondences as in~\eqref{csm}, and let \(f_1=
(V_1, E_1, \hat{f}_1)\) as a normally non-singular map.  By a Thom
modification, the geometric correspondence \(X\xleftarrow{b_1}
(M_1,\xi) \xrightarrow{f_1} Y\) is equivalent to
\begin{equation}
  \label{c1}
  X\xleftarrow{b_1\circ \pi_{V_1}} (V_1, \tau_{V_1}\otimes\pi_{V_1}^*\xi)
  \xrightarrow{\pi_{E_1} \circ \hat{f}_1}Y,
\end{equation}
where \(\pi_{V_1}\colon V_1\to M_1\) and \(\pi_{E_1}\colon Y\times
E\to Y\) are the bundle projections, and \(\tau_{V_1}\in
\RK^*_{G,M_1}(V_1)\) is the Thom class of~\(V_1\).  We
write~\(\otimes\) for the multiplication of \(\K\)\nb-theory classes.
The support of such a product is the intersection of the supports of
the factors.  Hence the support of \(\tau_{V_1}\otimes\pi_{V_1}^*\xi\)
is an \(X\)\nb-compact subset of~\(V_1\).

The forward map \(V_1\to Y\) in~\eqref{c1} is a special
submersion and, in particular, a submersion.  As such it is
transverse to any other map \(b_2\colon M_2 \to Y\).  Hence
after the Thom modification we may compute the composition of
correspondences as an intersection product of the
correspondence~\eqref{c1} with the correspondence \(Y
\xleftarrow{b_2}M_2 \xrightarrow{f_2} Y\).  This yields
\begin{equation}
  \label{c2}
  X\xleftarrow{b_1\circ\pi_{V_1} \circ \pr_1}
  \bigl(V_1\times_Y M_2,
  \pr_{V_1}^*(\tau_{V_1}\otimes\pi_{V_1}^*(\xi))\bigr)
  \xrightarrow{f_2\circ\pr_2} Z,
\end{equation}
where
\[
V_1\times_Y M_2\defeq \{ (x,v, m_2)\in V_1\times M_2 \mid
(\pi_{E_1}\circ \hat{f}_1)(x,v) = b_2(m_2)\}
\]
and \(\pr_1\colon V_1\times_Y M_2\to V_1\) and \(\pr_2\colon
V_1\times_Y M_2\to M_2\) are the coordinate projections.  The
intersection space \(V_1\times_Y M_2\) is a smooth manifold with
tangent bundle
\[
\Tang V_1\oplus_{\Tang Y} \Tang M_2\defeq \pr_1^*(\Tang V_1)
\oplus_{(\pi_{E_1}\circ \hat{f}_1)^*(\Tang Y)} \pr_2^*(\Tang M_2),
\]
and the map~\(\pr_2\) is a submersion with fibres tangent to~\(E_1\).
Thus it is \(\K_G\)\nb-oriented.

This recipe to define the composition product for all geometric
correspondences is introduced in~\cite{Emerson-Meyer:Correspondences}.
It is shown there that it is equivalent to the intersection product if
\(f_1\) and~\(b_2\) are transverse.  But the space \(V_1\times_Y M_2\)
has high dimension, making it inefficient to compute with this
formula.  And we are usually given only the underlying map \(f_1\colon
M_1\to Y\), not its factorisation as a normally non-singular map --
and the latter is difficult to compute.  We will weaken the
transversality requirement in Section~\ref{subsec:ptrans}.  The more
general condition still applies, say, if \(f_1=b_2\).  This is
particularly useful for computing Euler characteristics.

\subsection{Duality and the Lefschetz index}
\label{sec:Lefind}

Duality plays a crucial role
in~\cite{Emerson-Meyer:Correspondences} in order to compare
the geometric and analytic models of equivariant Kasparov theory.
Duality is also used in \cite{Emerson-Meyer:Dualities}*{Definition
  4.26} to construct a Lefschetz map
\begin{equation}
  \label{eq:Lefschetz_map}
  \Lef\colon \KK^G_*\bigl(\Cont(X), \Cont(X)\bigr)
  \to \KK^G_*(\Cont(X), \C),
\end{equation}
for a compact smooth \(G\)\nb-manifold~\(X\).  We may compose~\(\Lef\)
with the index map \(\KK^G_*(\Cont(X), \C) \to \KK^G_*(\C, \C) \cong
\Rep(G)\) to get a Lefschetz index \(\lefind(f) \in \Rep(G)\) for any
\(f\in \KK_*^G\bigl(\Cont(X), \Cont(X)\bigr)\).  This is the invariant
we will be studying in this paper.

This Lefschetz map~\(\Lef\) is a special case of a very general
construction.  Let~\(\Ccat\) be a symmetric monoidal category.
Let~\(A\) be a dualisable object of~\(\Ccat\) with a dual~\(A^*\).
Let \(\eta\colon \Unit\to A\otimes A^*\) and \(\varepsilon\colon
A^*\otimes A\to \Unit\) be the unit and counit of the duality.  Being
unit and counit of a duality means that they satisfy the zigzag
equations: the composition
\begin{equation}
  \label{zigzag}
  A\xrightarrow{\eta \otimes \id_A} A\otimes A^*\otimes A
  \xrightarrow{\id_A\otimes\varepsilon} A
\end{equation}
is equal to the identity \(\id_A\colon A \to A\), and similarly for
the composition
\begin{equation}
  \label{zigzag2}
  A^*\xrightarrow{\id_{A^*}\otimes\eta} A^*\otimes A\otimes A^*
  \xrightarrow{\varepsilon\otimes\id_{A^*}} A^*.
\end{equation}
If~\(\Ccat\) is \(\Z\)\nb-graded, then we may allow dualities to shift
degrees.  Then some signs are necessary in the zigzag equations, see
\cite{Emerson-Meyer:Dualities}*{Theorem 5.5}.

Given a multiplication map \(m\colon A\otimes A\to A\), we define the
\emph{Lefschetz map}
\[
\Lef\colon \Ccat(A,A)\to\Ccat(A,\Unit)
\]
by sending an endomorphism \(f\colon A\to A\) to the composite
morphism
\[
A \cong A \otimes\Unit \xrightarrow{\id_A\otimes\eta}
A\otimes A\otimes A^* \xrightarrow{m\otimes\id_{A^*}}
A\otimes A^* \xrightarrow{f\otimes\id_{A^*}}
A\otimes A^* \xrightarrow{\braid}
A^* \otimes A \xrightarrow{\varepsilon} \Unit.
\]
This depends only on \(m\) and~\(f\), not on the choices of the dual,
unit and counit.  For \(f=\id_A\) we get the \emph{higher Euler
  characteristic} of~\(A\) in \(\Ccat(A,\Unit)\).

While the geometric computations below give the Lefschetz map as
defined above, the global homological computations in Sections
\ref{sec:trace_localisation} and~\ref{sec:Hattori-Stallings} only
apply to the following coarser invariant:

\begin{definition}
  \label{maindef:lefindex}
  The \emph{Lefschetz index}~\(\lefind(f)\) (or \emph{trace}
  \(\tr(f)\)) of an endomorphism \(f\colon A \to A\) is the composite
  \begin{equation}
    \label{eq:lefindex}
    \Unit \xrightarrow{\eta} A\otimes A^* \xrightarrow{\braid}
    A^*\otimes A \xrightarrow{\id_{A^*}\otimes f}
    A^*\otimes A \xrightarrow{\varepsilon} \Unit,
  \end{equation}
  where \(\braid\) denotes the braiding.  The Lefschetz index
  of~\(\id_A\) is called the \emph{Euler characteristic} of~\(A\).
\end{definition}

If~\(A\) is a unital algebra object in~\(\Ccat\) with multiplication
\(m\colon A\otimes A\to A\) and unit \(u\colon \Unit\to A\), then
\(\lefind(f) = \Lef(f)\circ u\).  In particular, the Euler
characteristic is the composite of the higher Euler characteristic
with~\(u\).

In this section, we work in \(\Ccat=\GKK^G\) for a compact group~\(G\)
with \(\Unit=\pt\) and \(\otimes=\times\).  In
Section~\ref{sec:trace_localisation}, we work in the related analytic
category \(\Ccat=\KK^G\) with \(\Unit=\C\) and the usual tensor
product.

We will show below that any compact smooth \(G\)\nb-manifold~\(X\) is
dualisable in \(\GKK^G\).  The multiplication \(m\colon X\times X\to
X\) and unit \(u\colon \pt\to X\) are given by the geometric
correspondences
\[
X\times X \xleftarrow{\Delta} X \xrightarrow[=]{\id_X} X,\qquad
\pt \leftarrow X \xrightarrow[=]{\id_X} X
\]
with \(\Delta(x)=(x,x)\); these induce the multiplication
\(^*\)\nb-homomorphism \(m\colon \Cont(X\times X)\cong
\Cont(X)\otimes\Cont(X)\to\Cont(X)\) and the embedding \(\C\to
\Cont(X)\) of constant functions.  Composing with~\(u\) corresponds to
taking the \emph{index} of a \(\K\)\nb-homology class.

\begin{remark}
  In \cites{Emerson-Meyer:Euler, Emerson-Meyer:Dualities,
    Emerson-Meyer:Equi_Lefschetz} Lefschetz maps are also studied for
  non-compact spaces~\(X\), equipped with group actions of possibly
  non-compact groups.  A non-compact \(G\)\nb-manifold~\(X\) is
  usually not dualisable in~\(\GKK^G\), and even if it were, the
  Lefschetz map that we would get from this duality would not be the
  one studied in \cites{Emerson-Meyer:Euler, Emerson-Meyer:Dualities,
    Emerson-Meyer:Equi_Lefschetz}.
\end{remark}

\subsection{Duality for smooth compact manifolds}
\label{sec:duality}

We are going to show that compact smooth \(G\)\nb-manifolds are
dualisable in the equivariant correspondence theory~\(\GKK^G\).  This
was already proved in~\cite{Emerson-Meyer:Correspondences}, but since
we need to know the unit and counit to compute Lefschetz indices, we
recall the proof in detail.  It is of some interest to treat duality
for smooth manifolds with boundary because any finite CW-complex is
homotopy equivalent to a manifold with boundary.

In case~\(X\) has a boundary~\(\partial X\), let \(\intr X
\defeq X\setminus \partial X\) denote its interior and let
\(\iota\colon \intr X\to X\) denote
the inclusion map.  The boundary \(\partial X\subseteq X\)
admits a \(G\)\nb-equivariant collar, that is, the embedding
\(\partial X\to X\) extends to a \(G\)\nb-equivariant
diffeomorphism from \(\partial X\times [0,1)\) onto an open
neighbourhood of~\(\partial X\) in~\(X\) (see
also~\cite{Emerson-Meyer:Dualities}*{Lemma 7.6} for this
standard result).  This collar neighbourhood together with a
smooth map \([0,1)\to (0,1)\) that is the identity near~\(1\)
provides a smooth \(G\)\nb-equivariant map \(\rho\colon X\to
\intr X\) that is inverse to~\(\iota\) up to smooth
\(G\)\nb-homotopy.  Furthermore, we may assume that~\(\rho\) is
a diffeomorphism onto its image.

If~\(X\) has no boundary, then \(\intr X=X\), \(\iota=\id\),
and \(\rho=\id\).

The results about smooth normally non-singular maps
in~\cite{Emerson-Meyer:Normal_maps} extend to smooth manifolds
with boundary if we add suitable assumptions about the
behaviour near the boundary.  We mention one result of this
type and a counterexample.

\begin{proposition}
  \label{pro:non-singular_smooth_with_boundary}
  Let \(X\) and~\(Y\) be smooth \(G\)\nb-manifolds with~\(X\)
  of finite orbit type and let \(f\colon X\to Y\) be a smooth
  map with \(f(\partial X)\subseteq \partial Y\) and~\(f\)
  transverse to~\(\partial Y\).  Then~\(f\) lifts to a
  normally non-singular map, and any two such normally
  non-singular liftings of~\(f\) are equivalent.
\end{proposition}

\begin{proof}
  Since~\(X\) has finite orbit type, we may smoothly
  embed~\(X\) into a finite-dimensional linear
  \(G\)\nb-representation~\(E\).  Our assumptions ensure that
  the resulting map \(X\to Y\times E\) is a smooth embedding
  between \(G\)\nb-manifolds with boundary in the sense of
  \cite{Emerson-Meyer:Normal_maps}*{Definition 3.17} and hence
  has a tubular neighbourhood by
  \cite{Emerson-Meyer:Normal_maps}*{Theorem 3.18}.  This
  provides a normally non-singular map \(X\to Y\)
  lifting~\(f\).  The uniqueness up to equivalence is proved as
  in the proof of \cite{Emerson-Meyer:Normal_maps}*{Theorem
    4.36}.
\end{proof}

\begin{example}
  \label{exa:smooth_with_boundary_not_nns}
  The inclusion map \(\{0\}\to [0,1)\) is a smooth map between
  manifolds with boundary, but it does not lift to a smooth
  normally non-singular map.
\end{example}

Let~\(X\) be a smooth compact \(G\)\nb-manifold.  Since~\(X\) has
finite orbit type, it embeds into some linear
\(G\)\nb-representation~\(E\).  We may choose this
\(G\)\nb-representation to be \(\K_G\)\nb-oriented and even-dimensional
by a further stabilisation.  Let \(\Nor X\epi X\) be the \emph{normal
  bundle} for such an embedding \(X\to E\).  Thus \(\Tang X\oplus \Nor
X \cong X\times E\) is \(G\)\nb-equivariantly isomorphic to a
\(\K_G\)\nb-oriented trivial \(G\)\nb-vector bundle.

\begin{theorem}
  \label{the:compact_smooth_dualisable}
  Let~\(X\) be a smooth compact \(G\)\nb-manifold, possibly
  with boundary.  Then~\(X\) is dualisable in \(\GKK^G_*\) with
  dual~\(\Nor \intr X\), and the unit and counit for the
  duality are the geometric correspondences
  \[
  \pt \leftarrow
  X \xrightarrow{(\id,\zeta\rho)}
  X\times \Nor \intr X,\qquad
  \Nor \intr X\times X \xleftarrow{(\id,\iota\pi)}
  \Nor\intr X \rightarrow \pt,
  \]
  where \(\zeta\colon \intr X\to\Nor\intr X\) is the zero
  section, \(\rho\colon X\to\intr X\) is some
  \(G\)\nb-equivariant collar retraction, \(\pi\colon \Nor\intr
  X\to\intr X\) is the bundle projection, and \(\iota\colon
  \intr X\to X\) the identical inclusion.  The \(\K\)\nb-theory
  classes on the space in the middle are the trivial rank-one
  vector bundles for both correspondences.
\end{theorem}

\begin{proof}
  First we must check that the purported unit and counit above are
  indeed geometric correspondences; this contains describing the
  \(\K_G\)\nb-orientations on the forward maps, which is part of the
  data of the geometric correspondences.

  The maps \(X\to\pt\) and \(\Nor\intr X\to \Nor\intr X\times X\)
  above are proper.  Hence there is no support restriction for the
  \(\K\)\nb-theory class on the middle space, and the trivial rank-one
  vector bundle is allowed.

  By the Tubular Neighbourhood Theorem, the normal bundle~\(\Nor \intr
  X\) of the embedding \(\intr X\to E\) is diffeomorphic to an open
  subset of~\(E\).  This gives a canonical isomorphism between the
  tangent bundle of~\(\Nor \intr X\) and~\(E\).  We choose this
  isomorphism and the given \(\K_G\)\nb-orientation on the linear
  \(G\)\nb-representation~\(E\) to \(\K_G\)\nb-orient \(\Nor\intr X\)
  and thus the projection \(\Nor\intr X\to\pt\).  With this
  \(\K_G\)\nb-orientation, the counit \(\Nor \intr X\times X
  \xleftarrow{(\id,\iota\pi)} \Nor\intr X \rightarrow \pt\) is a
  \(G\)\nb-equivariant geometric correspondence -- even a special one
  in the sense of~\cite{Emerson-Meyer:Correspondences}.

  We identify the tangent bundle of \(X\times \Nor\intr X\) with
  \(\Tang X\times \Tang\intr X\oplus \Nor\intr X\) in the obvious way.
  The normal bundle of the embedding \((\id,\zeta\rho)\colon X\to
  X\times \Nor\intr X\) is isomorphic to the quotient of \(\Tang
  X\oplus\rho^*(\Tang \intr X)\oplus \rho^*(\Nor\intr X)\) by the
  relation \((\xi,D\rho(\xi),0)\sim0\) for \(\xi\in\Tang X\).  We
  identify this with \(\Tang X\oplus \Nor X\cong X\times E\) by
  \((\xi_1,\xi_2,\eta)\mapsto
  (D\rho^{-1}(\xi_2)-\xi_1,D\rho^{-1}(\eta))\) for \(\xi_1\in\Tang_x
  X\), \(\xi_2\in\Tang_{\rho(x)} X\), \(\eta\in\rho^*(\Nor\intr X)_x =
  \Nor_{\rho(x)}X\).  With this \(\K_G\)\nb-orientation on
  \((\id,\zeta\rho)\), the unit above is a \(G\)\nb-equivariant
  geometric correspondence.  A boundary of~\(X\), if present, causes
  no problems here.  The same goes for the computations below:
  although the results in~\cite{Emerson-Meyer:Correspondences} are
  formulated for smooth manifolds without boundary, they continue to
  hold in the cases we need.

  We establish the duality isomorphism by checking the zigzag
  equations as in \cite{Emerson-Meyer:Dualities}*{Theorem 5.5}.  This
  amounts to composing geometric correspondences.  In the case at
  hand, the correspondences we want to compose are transverse, so that
  they may be composed by intersections as in
  Section~\ref{subsec:trans}.  Actually, we are dealing with manifolds
  with boundary, but the argument goes through nevertheless.  We write
  down the diagrams together with the relevant Cartesian square.

  The intersection diagram for the first zigzag equation is
  \begin{equation}
    \label{eq:duality_first_zigzag}
    \begin{tikzpicture}[baseline=(current bounding box.west),column sep=3em]
      \matrix (m) [cd]{
        && X\\
        & X\times X && X\times \Nor\intr X\\
        X&&X\times \Nor \intr X\times X&&X.\\};
      \begin{scope}[corr]
        \draw (m-1-3)-- node[swap] {$(\id,\iota\rho)$} (m-2-2);
        \draw (m-1-3)-- node       {$(\id,\zeta\rho)$} (m-2-4);
        \draw (m-2-2)-- node[swap] {$\pr_2$} (m-3-1);
        \draw (m-2-2)-- node       {$(\id,\zeta\rho)\times\id$} (m-3-3);
        \draw (m-2-4)-- node[swap] {$\id\times(\id,\iota\pi)$} (m-3-3);
        \draw (m-2-4)-- node       {$\pr_1$} (m-3-5);
      \end{scope}
    \end{tikzpicture}
  \end{equation}
  The square is Cartesian because \((x,y,z,(w,\nu))\in X^3\times
  \Nor\intr X\) satisfies \((x,(\rho(x),0),y) = (z,(w,\nu),w)\) if and
  only if \(y=\rho(x)\), \(z=x\), \(w=\rho(x)\), and \(\nu=0\) for
  some \(x\in X\).  The \(\K_G\)\nb-orientation on the map
  \((\id,\zeta\rho)\) described above is chosen such that the
  composite map \(f\defeq
  \pr_1\circ (\id,\zeta\rho) = \id\) carries the standard
  \(\K_G\)\nb-orientation.  The map \(b\defeq \pr_2\circ (\id,\iota\rho)
  = \iota\rho\) is properly homotopic to the identity map.  Hence
  the composition above gives the identity map on~\(X\) as required.

  The intersection diagram for the second zigzag equation is
  \begin{equation}
    \label{eq:duality_second_zigzag}
    \begin{tikzpicture}[baseline=(current bounding box.west)]
      \matrix (m) [cd]{
        && \Nor\intr X\\
        & \Nor \intr X\times X && \Nor\intr X\times \Nor \intr X\\
        \Nor \intr X&&\Nor \intr X\times X\times \Nor \intr X&&\Nor \intr X\\};
      \begin{scope}[corr]
        \draw (m-1-3)-- node[swap] {$(\id,\iota\pi)$} (m-2-2);
        \draw (m-1-3)-- node       {$(\id,\zeta\rho\pi)$} (m-2-4);
        \draw (m-2-2)-- node[swap] {$\pr_1$} (m-3-1);
        \draw (m-2-2)-- node {$\id\times(\id,\zeta\rho)$} (m-3-3);
        \draw (m-2-4)-- node[swap] {$(\id,\iota\pi)\times\id$} (m-3-3);
        \draw (m-2-4)-- node {$\pr_2$} (m-3-5);
      \end{scope}
    \end{tikzpicture}
  \end{equation}
  because \(((x,\nu),y,(w,\mu),(z,\kappa)) \in \Nor \intr X\times
  X\times (\Nor \intr X)^2\) satisfy
  \[
  ((x,\nu),y,(\rho(y),0)) = ((w,\mu),w,(z,\kappa))
  \]
  if and only if \((w,\mu)=(x,\nu)\), \(y=x\), \(z=\rho(x)\),
  \(\kappa=0\) for some \((x,\nu)\in\Nor\intr X\).

  The map \((\id,\zeta\rho\pi)\) is smoothly homotopic to the diagonal
  embedding \(\delta\colon \Nor\intr X\to\Nor\intr X\times \Nor\intr
  X\).  Replacing \((\id,\zeta\rho\pi)\) by~\(\delta\) gives an
  equivalent geometric correspondence.  The \(\K_G\)\nb-orientation on
  the normal bundle of \((\id,\zeta\rho\pi)\) that comes with the
  composition product is transformed by this homotopy to the
  \(\K_G\)\nb-orientation on the normal bundle of the diagonal embedding
  that we get by identifying the latter with the pull-back of~\(E\) by
  mapping
  \[
  (\xi_1,\eta_1,\xi_2,\eta_2)\in
  \Tang_{(x,\zeta,x,\zeta)}(\Nor\intr X\times\Nor\intr X)
  \cong \Tang_{x}\intr X\oplus\Nor_{x}\intr X\times
  \Tang_{x}\intr X\times\Nor_{x}\intr X \cong E_{x}\times E_{x}
  \]
  to \((\xi_2-\xi_1,\eta_2-\eta_1)\in E_x\).  Since~\(E\) has even
  dimension, changing this to \((\xi_1-\xi_2,\eta_1-\eta_2)\) does not
  change the \(\K_G\)\nb-orientation.  Hence the induced
  \(\K_G\)\nb-orientation on the fibres of~\(D\pr_2\) is the same one
  that we used to \(\K_G\)\nb-orient \(\pr_2\).  The induced
  \(\K_G\)\nb-orientation on \(\pr_2\circ\delta=\id\) is the standard
  one.  Thus the composition in~\eqref{eq:duality_second_zigzag} is
  the identity on~\(\Nor\intr X\).
\end{proof}

\begin{corollary}
  \label{cor:KK_by_duality}
  Let~\(X\) be a compact smooth \(G\)\nb-manifold and let~\(Y\)
  be any locally compact \(G\)\nb-space.  Then every
  element of \(\GKK^G_*(X,Y)\) is represented by a geometric
  correspondence of the form
  \[
  X\xleftarrow{\iota\circ\pi\circ\pr_1} \Nor\intr X \times Y
  \xrightarrow{\pr_2} Y,\qquad
  \xi\in\K^*_G(\Nor\intr X\times Y),
  \]
  and two such correspondences for
  \(\xi_1,\xi_2\in\K^*_G(\Nor\intr X\times Y)\) give the same
  element of \(\GKK^G_*(X,Y)\) if and only if \(\xi_1=\xi_2\).
  Here \(\pr_1\colon \Nor\intr X\times Y\to \Nor\intr X\) and
  \(\pr_2\colon \Nor\intr X\times Y\to Y\) are the coordinate
  projections and \(\iota\circ\pi\colon \Nor\intr X\to \intr
  X\subseteq X\) is as above.
\end{corollary}

\begin{proof}
  Duality provides a canonical isomorphism
  \[
  \K^*_G(\Nor\intr X\times Y)
  \cong \GKK^G_*(\pt,\Nor\intr X\times Y)
  \cong \GKK^G_*(X,Y).
  \]
  It maps \(\xi\in\K^*_G(\Nor\intr X\times Y)\) to the composition of
  correspondences described by the following intersection diagram:
  \[
  \begin{tikzpicture}[baseline=(current bounding box.west)]
    \matrix (m) [cd]{
      && \Nor\intr X\times Y\\
      & X\times\Nor\intr X\times Y &&\Nor\intr X\times Y\\
      X&&X\times\Nor\intr X\times Y&&Y,\\};
    \begin{scope}[corr]
      \draw (m-1-3)-- node[swap] {$(\iota\pi,\id)\times\id$} (m-2-2);
      \draw (m-1-3)-- node       {$\id$} (m-2-4);
      \draw (m-2-2)-- node[swap] {$\pr_1$} (m-3-1);
      \draw (m-2-2)-- node {$\id$} (m-3-3);
      \draw (m-2-4)-- node[swap] {$(\iota\pi,\id)\times\id$} (m-3-3);
      \draw (m-2-4)-- node {$\pr_2$} (m-3-5);
    \end{scope}
  \end{tikzpicture}
  \]
  with the \(\K\)\nb-theory class~\(\xi\) on \(\Nor\intr X\times Y\).
  Hence it involves the maps \(\iota\pi\colon \Nor\intr X\times Y\to
  X\) and \(\pr_2\colon \Nor\intr X\times Y \to Y\).
\end{proof}

If~\(X\) is, in addition, \(\K_G\)\nb-oriented, then the Thom
isomorphism provides an isomorphism \(\Nor\intr X\cong\intr X\) in
\(\GKK^G_*\) (which has odd parity if the dimension of~\(X\) is odd).
A variant of Corollary~\ref{cor:KK_by_duality} yields a duality
isomorphism
\[
\K^{*+\dim(X)}_G(\intr X\times Y) \cong \GKK^G_*(X,Y),
\]
which maps \(\xi \in \K^*_G(\intr X\times Y)\) to the geometric
correspondence
\[
X\xleftarrow{\iota\circ\pr_1} \intr X\times Y\xrightarrow{\pr_2}
Y,\qquad
\xi \in \K^*_G(\intr X\times Y).
\]
Hence any element of \(\GKK^G_*(X,Y)\) is represented by a
correspondence of this form.

If~\(X\) is \(\K_G\)\nb-oriented and has no boundary, this becomes
\[
X\xleftarrow{\pr_1} X\times Y\xrightarrow{\pr_2}
Y,\qquad
\xi \in \K^*_G(X\times Y).
\]

These standard forms for correspondences are less useful than
one may hope at first because their intersection products are
no longer in this standard form.

\subsection{More on composition of geometric correspondences}
\label{subsec:ptrans}

With our geometric formulas for the unit and counit of the duality, we
could now compute Lefschetz indices geometrically, assuming the
necessary intersections are transverse.  While this works well, say,
for self-maps with regular non-degenerate fixed points, it fails badly
for the identity correspondence, whose Lefschetz index is the Euler
characteristic.  Building on work of Baum and
Block~\cite{Baum-Block:Excess}, we now describe the composition as a
modified intersection product under a much weaker assumption than
transversality that still covers the computation of Euler
characteristics.

\begin{definition}
  \label{def:smooth_intersection}
  We say that the smooth maps \(f_1\colon M_1\to Y\) and \(b_2\colon
  M_2\to Y\) \emph{intersect smoothly} if
  \[
  M \defeq M_1\times_Y M_2
  \]
  is a smooth submanifold of \(M_1\times M_2\) and any
  \((\xi_1,\xi_2)\in \Tang M_1\times\Tang M_2\) with
  \(Df_1(\xi_1)=Db_2(\xi_2)\in\Tang Y\) is tangent to~\(M\).

  If \(f_1\) and~\(b_2\) intersect smoothly, then we define the
  \emph{excess intersection bundle}~\(\eta(f_1,b_2)\) on~\(M\) as the
  cokernel of the vector bundle map
  \begin{equation}
    \label{testmap2}
    (Df_1,-Db_2)\colon \pr_1^*(\Tang M_1)\oplus \pr_2^*(\Tang M_2)
    \to f^*(\Tang Y),
  \end{equation}
  where \(f\defeq f_1\circ \pr_1 = b_2\circ\pr_2\colon M\to Y\).

  If the maps \(f_1\) and~\(b_2\) are \(G\)\nb-equivariant with
  respect to a compact group~\(G\), then
  the excess intersection bundle is a \(G\)\nb-vector bundle.

  We call the square
  \[
  \begin{tikzpicture}
    \matrix (m) [cd]{
      M &M_2\\M_1&Y\\};
    \begin{scope}[cdar]
      \draw (m-1-1)-- node {$\pr_2$} (m-1-2);
      \draw (m-1-1)-- node[swap] {$\pr_1$} (m-2-1);
      \draw (m-1-2)-- node {$f_1$} (m-2-2);
      \draw (m-2-1)-- node {$b_2$} (m-2-2);
    \end{scope}
  \end{tikzpicture}
  \]
  \emph{\(\eta\)\nb-Cartesian} if \(f_1\) and~\(b_2\) intersect
  smoothly with excess intersection bundle~\(\eta\).
\end{definition}

If~\(M\) is a smooth submanifold of \(M_1\times M_2\), then \(\Tang M
\subseteq \Tang(M_1\times M_2)\); and if \((\xi_1,\xi_2)\in
\Tang(M_1\times M_2)\) is tangent to~\(M\), then
\(Df_1(\xi_1)=Db_2(\xi_2)\) in~\(\Tang Y\).  These pairs
\((\xi_1,\xi_2)\) form a subspace of \(\Tang(M_1\times M_2)|_M =
\pr_1^*\Tang M_1 \oplus \pr_2^*\Tang M_2\), which in general need not
be a vector bundle, that is, its rank need not be locally constant.
The smooth intersection assumption forces it to be a subbundle: the
kernel of the map in~\eqref{testmap2}.  Hence
the excess intersection bundle is a vector bundle over~\(M\), and
there is the following exact sequence of vector bundles over~\(M\):
\begin{equation}
  \label{eq:excess_bundle_exact_sequence}
  0 \to \Tang M
  \to \pr_1^*(\Tang M_1)\oplus \pr_2^*(\Tang M_2)
  \xrightarrow{(Df_1,-Db_2)} (f_1\circ \pr_1)^*(\Tang Y)
  \to \eta
  \to 0.
\end{equation}

\begin{example}
  \label{exa:smooth_intersect_diagonal}
  Let \(M_1=M_2=X\) and let \(f_1=b_2=i\colon X\to Y\) be an injective
  immersion.  Then \(M_1\times_Y M_2\cong X\) is the diagonal in
  \(M_1\times M_2=X^2\), which is a smooth submanifold.  Furthermore,
  if \((\xi_1,\xi_2)\in \Tang M_1\times\Tang M_2\) satisfy
  \(Df_1(\xi_1)=Db_2(\xi_2)\), then \(\xi_1=\xi_2\) because \(Di\colon
  \Tang M\to\Tang Y\) is assumed injective.  Hence \(M_1\) and~\(M_2\)
  intersect smoothly, and the excess intersection bundle is the normal
  bundle of the immersion~\(i\).
\end{example}

\begin{example}
  \label{exa:touching_circles}
  \begin{figure}[htbp]
    \noindent\begin{tikzpicture}
      \draw [name path=cone] (0.1\linewidth,0) circle (0.08\linewidth);
      \draw [name path=ctwo] (0.14\linewidth,0) circle (0.09\linewidth);
      \fill [name intersections={of=cone and ctwo}]
      (intersection-1) circle (2pt) node[above] {$p_1$}
      (intersection-2) circle (2pt) node[below] {$p_2$};

      \draw (0.375\linewidth,0) circle (0.1\linewidth);
      \draw (0.35\linewidth,0) circle (0.05\linewidth);

      \draw (0.625\linewidth,0) circle (0.1\linewidth) node {\(M_1=M_2\)};

      \draw [name path=cthree] (0.875\linewidth,0) circle (0.1\linewidth);
      \draw [name path=cfour] (0.825\linewidth,0) circle (0.05\linewidth);
      \fill [name intersections={of=cthree and cfour}]
      (intersection-1) circle (2pt) node[right] {$p$};
    \end{tikzpicture}
    \caption{Four possible configurations of two circles in the plane}
    \label{fig:circles_in_plane}
  \end{figure}
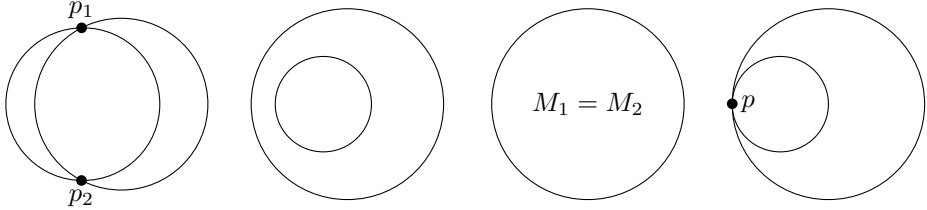
  Let \(M_1\) and~\(M_2\) be two circles embedded in \(Y=\R^2\).
  The four possible configurations are illustrated in
  Figure~\ref{fig:circles_in_plane}.
  \begin{enumerate}
  \item The circles meet in two points.  Then \(M=\{p_1,p_2\}\) and
    the intersection is transverse.
  \item The two circles are disjoint.  Then \(M=\emptyset\) and the
    intersection is transverse.
  \item The two circles are identical.  Then \(M=M_1=M_2\).  The
    intersection is not transverse, but smooth by
    Example~\ref{exa:smooth_intersect_diagonal}; the excess
    intersection bundle is the normal bundle of the circle, which is
    trivial.
  \item The two circles touch in one point.  Then \(M\defeq
    M_1\times_Y M_2 = \{p\}\), so that the tangent bundle of~\(M\) is
    zero-dimensional.  But \(\Tang_p M_1\cap\Tang_p M_2\) is
    one-dimensional because \(\Tang_p M_1=\Tang_p M_2\).  Hence the
    embeddings do \emph{not} intersect smoothly.
  \end{enumerate}
\end{example}

\begin{remark}
  \label{rem:intersect_smooth_normal_form}
  The maps \(f\colon M_1\to Y\) and \(b\colon M_2\to Y\)
  intersect smoothly if and only if \(f\times b\colon M_1\times
  M_2 \to Y\times Y\) and the diagonal embedding \(Y\to Y\times
  Y\) intersect smoothly; both pairs of maps have the same
  excess intersection bundle.  Thus we may always normalise
  intersections to the case where one map is a diagonal
  embedding and thus an embedding.
\end{remark}

\begin{example}
  \label{exa:defect_explained}
  Let~\(\eta\) be a \(\K_G\)\nb-oriented vector bundle over~\(X\).  Let
  \(M_1=M_2=X\), \(Y=\eta\), and let \(f_1=b_2=\zeta\colon Y\to\eta\)
  be the zero section of~\(\eta\).  This is a special case of
  Example~\ref{exa:smooth_intersect_diagonal}.  The maps \(f_1\)
  and~\(b_2\) intersect smoothly with excess intersection
  bundle~\(\eta\).

  In this example it is easy to compose the geometric correspondences
  \(X=X\to\eta\) and \(\eta\leftarrow X=X\).  A Thom modification of
  the first one along the \(\K_G\)\nb-oriented vector bundle~\(\eta\)
  gives the special correspondence
  \[
  X\leftarrow(\eta,\tau_\eta) =\eta,
  \]
  where \(\tau_\eta\in\RK^*_{G,X}(\eta)\) is the Thom class
  of~\(\eta\).  The intersection product of this with \(\eta\leftarrow
  X=X\) is \(X=(X,\zeta^*(\tau_\eta))=X\), that is, it is the class in
  \(\GKK^G_*(X,X)\) of \(\zeta^*(\tau_\eta)\in\RK^*_G(X)\).  This
  \(\K\)\nb-theory class is the restriction of~\(\tau_\eta\) to the
  zero section of~\(\eta\).  By the construction of the Thom class, it
  is the \(\K\)\nb-theory class of the spinor bundle of~\(\eta\).
\end{example}

\begin{definition}
  \label{def:Euler_defect}
  Let~\(\eta\) be a \(\K_G\)\nb-oriented \(G\)\nb-vector bundle over a
  \(G\)\nb-space~\(X\).  Let \(\zeta\colon X\to\eta\) be the zero
  section and let \(\tau_\eta\in\RK^*_{G,X}(\eta)\) be the Thom class.
  The \emph{Euler class} of~\(\eta\) is \(\zeta^*(\tau_\eta)\), the
  restriction of~\(\tau_\eta\) to the zero section.
\end{definition}

By definition, the Euler class is the composition of the
correspondences \(\pt \leftarrow X\to\eta\) and \(\eta\leftarrow X=X\)
involving the zero section \(\zeta\colon X\to\eta\) in both cases.

\begin{example}
  Assume that there is a \(G\)\nb-equivariant section \(s\colon X \to
  \eta\) of~\(\eta\) with isolated simple zeros; that is, \(s\)
  and~\(\zeta\) are transverse.  The linear homotopy connects~\(s\) to
  the zero section and hence gives an equivalent correspondence
  \(\eta\xleftarrow{s} X=X\).  Since \(s\) and~\(\zeta\) are
  transverse by assumption, the composition is \(X\leftarrow Z\to X\),
  where~\(Z\) is the zero set of~\(s\) and the maps \(Z\to X\) are the
  inclusion map, suitably \(\K_G\)\nb-oriented.
\end{example}

\begin{example}
  \label{ptex}
  Let \(M_1 = S^1\), \(M_2 = S^2\), \(Y = \R^3\), \(b_2\colon M_2 \to
  \R^3\) be the standard embedding of the \(2\)\nb-sphere in~\(\R^3\),
  and let \(f_1\colon M_1 \to M_2 \to \R^3\) be the embedding
  corresponding to the equator of the circle.  Then \(M_1\times_Y M_2
  = M_1\times_{M_2} M_2 = M_1\), embedded diagonally into \(M_1\times
  M_1\subset M_1\times M_2\).  This is a case of smooth intersection.
  The excess intersection bundle is the restriction to the equator of
  the normal bundle of the embedding~\(b_2\).  This is isomorphic to
  the rank-one trivial bundle on~\(S^2\).  Hence the Euler class
  \(e(\eta)\) is zero in this case.
\end{example}



\begin{theorem}
  \label{theorem:composition}
  Let
  \begin{equation}
    \label{eq:composition_smooth_intersection_data}
    X\xleftarrow{b_1} (M_1,\xi_1)
    \xrightarrow{f_1} Y
    \xleftarrow{b_2} (M_2, \xi_2)
    \xrightarrow{f_2} Z
  \end{equation}
  be a pair of \(G\)\nb-equivariant correspondences as in~\eqref{csm}.
  Assume that \(b_2\) and \(f_1\) intersect smoothly and with a
  \(\K_G\)\nb-oriented excess intersection bundle~\(\eta\).  Then the
  composition of~\eqref{eq:composition_smooth_intersection_data} is
  represented by the \(G\)\nb-equivariant correspondence
  \begin{equation}
    \label{eq:composition_smooth_intersection}
    X \xleftarrow{b_1\circ\pr_1}
    \bigl(M_1\times_Y M_2, e(\eta) \otimes
    \pr_1^*(\xi_1)\otimes\pr_2^*(\xi_2)\bigr)
    \xrightarrow{f_2\circ\pr_2} Z,
  \end{equation}
  where \(e(\eta)\) is the Euler class and the projection
  \(\pr_2\colon M_1\times_Y M_2\to M_2\) carries the
  \(\K_G\)\nb-orientation induced by the \(\K_G\)\nb-orientations on
  \(f_1\) and~\(\eta\) \textup(explained below\textup).
\end{theorem}

In the above situation of smooth intersection, we call the
diagram~\eqref{eq:intersection_diagram} an
\emph{\(\eta\)\nb-intersection diagram}.  It still computes the
composition, but we need the Euler class of the excess intersection
bundle~\(\eta\) to compensate the lack of transversality.

We describe the canonical \(\K_G\)\nb-orientation of \(\pr_2\colon
M_1\times_Y M_2\to M_2\).  The excess intersection bundle~\(\eta\) is
defined so as to give an exact sequence of vector
bundles~\eqref{eq:excess_bundle_exact_sequence}. From this it follows that
\[
[\eta] = (f_1\circ \pr_1)^*[\Tang Y] + \Tang M - \pr_1^*[\Tang M_1] - \pr_2^*[\Tang M_2].
\]
On the other hand, the stable normal bundle~\(N\pr_2 \) of~\(\pr_2\)
is equal to \(\pr_2^*[\Tang M_2]- [\Tang M]\).  Hence
\[
[\eta ] = \pr_1^*\bigl( f_1^*[\Tang Y] - [\Tang M_1]\bigr)- N\pr_2.
\]
A \(\K_G\)-orientation on \(f_1\) means a stable \(\K_G\)-orientation
on \(Nf_1 = f_1^*[\Tang Y] - [\Tang M_1]\).  If such an orientation is
given, it pulls back to one on \(\pr_1^*\bigl( f_1^*[\Tang Y] - [\Tang
M_1]\bigr)\), and then (stable) \(\K_G\)-orientations on~\([\eta]\)
and on~\(N\pr_2\) are in \(1\)-to-\(1\)-correspondence.  In
particular, a \(\K_G\)\nb-orientation on the bundle~\(\eta\) induces
one on the normal bundle of~\(\pr_2\).  This induced
\(\K_G\)\nb-orientation on~\(\pr_2\) is used
in~\eqref{eq:composition_smooth_intersection}.
(\cite{Emerson-Meyer:Normal_maps}*{Lemma 5.13} justifies working with
\(\K_G\)\nb-orientations on stable normal bundles.)

\begin{proof}[Proof of Theorem~\textup{\ref{theorem:composition}}]
  Lift~\(f_1\) to a \(G\)\nb-equivariant smooth normally non-singular
  map \((V_1, E_1, \hat{f}_1)\).  The composition
  of~\eqref{eq:composition_smooth_intersection_data} is defined
  in~\cite{Emerson-Meyer:Correspondences}*{Section 2.5} as the
  intersection product
  \begin{equation}
    \label{ipe1}
    X \xleftarrow{b_1\circ\pi_{V_1}\circ \pr_{V_1}} V_1\times_Y M_2
    \xrightarrow{f_2\circ\pr_2} Z
  \end{equation}
  with \(\K\)\nb-theory datum \(\pr_{V_1}^*(\tau_{V_1}) \otimes
  \pi_{V_1}^*(\xi_1)\otimes \pr_2^*(\xi_2) \in \RK^*_{G,X}(V_1\times_Y
  M_2)\).  We define the manifold \(V_1\times_Y M_2\) using the
  (transverse) maps \(\pi_{E_1}\circ \hat{f}_1\colon V_1 \to Y\) and
  \(b_2\colon M_2 \to Y\).  We must compare this with the
  correspondence in the statement of the theorem.

  We have a commuting square of embeddings of smooth manifolds
  \begin{equation}
    \label{eq:excess_intersection_diagram}
    \begin{tikzpicture}[baseline=(current bounding box.west)]
      \matrix (m) [cd]{
        M_1\times_Y M_2& M_1\times M_2\\
        V_1\times_Y M_2& V_1\times M_2\\};
      \begin{scope}[->,auto]
        \draw (m-1-1)-- node {$\iota_0$} (m-1-2);
        \draw (m-1-1)-- node {$\zeta_0$} (m-2-1);
        \draw (m-2-1)-- node {$\iota_1$} (m-2-2);
        \draw (m-1-2)-- node {$\zeta_1$} (m-2-2);
      \end{scope}
    \end{tikzpicture}
  \end{equation}
  where the vertical maps are induced by the zero section \(M_1\to
  V_1\) and the horizontal ones are the obvious inclusion maps.  The
  map~\(\zeta_0\) is a smooth embedding because the other three maps
  in the square are so.

  Let \(N\iota_0\) and \(\nu \defeq N\zeta_0\) denote the normal
  bundles of the maps \(\iota_0\) and~\(\zeta_0\)
  in~\eqref{eq:excess_intersection_diagram}.  The normal
  bundle of~\(\iota_1\) is isomorphic to the pull-back of~\(\Tang Y\)
  because
  \(V_1\to Y\) is submersive.  Since \(M_1\times M_2\to V_1\times
  M_2\) is the zero section of the pull back of the vector
  bundle~\(V_1\) to \(M_1\times M_2\), the normal bundle
  of~\(\zeta_1\) is isomorphic to~\(\pr_1^*(V_1)\).  Recall that
  \(M\defeq M_1\times_Y M_2\).  We get a diagram of vector bundles
  over~\(M\):
  \[
  \begin{tikzpicture}[baseline=(current bounding box.west),column sep=4em]
    \matrix (m) [cd]{
      &0&0&0&\\
      0& \Tang M& \Tang(M_1\times M_2)|_M& N\iota_0 &0\\
      0& \Tang(V_1\times_Y M_2)|_M& \Tang(V_1\times M_2)|_M& f^*(\Tang Y)&0\\
      0& \nu& \pr_1^*(V_1)& \eta&0\\
      &0&0&0&\\};
    \begin{scope}[->,auto]
      \draw (m-1-2)-- (m-2-2);
      \draw (m-1-3)-- (m-2-3);
      \draw (m-1-4)-- (m-2-4);

      \draw (m-2-2)-- node {$D\zeta_0$} (m-3-2);
      \draw (m-2-3)-- node {$D\zeta_1$} (m-3-3);
      \draw (m-2-4)-- (m-3-4);

      \draw (m-3-2)-- (m-4-2);
      \draw (m-3-3)-- (m-4-3);
      \draw (m-3-4)-- (m-4-4);

      \draw (m-4-2)-- (m-5-2);
      \draw (m-4-3)-- (m-5-3);
      \draw (m-4-4)-- (m-5-4);

      \draw (m-2-1)-- (m-2-2);
      \draw (m-3-1)-- (m-3-2);
      \draw (m-4-1)-- (m-4-2);

      \draw (m-2-2)-- node {$D\iota_0$} (m-2-3);
      \draw (m-3-2)-- node {$D\iota_1$} (m-3-3);
      \draw (m-4-2)-- (m-4-3);

      \draw (m-2-3)-- (m-2-4);
      \draw (m-3-3)-- (m-3-4);
      \draw[dotted] (m-4-3)-- (m-4-4);

      \draw (m-2-4)-- (m-2-5);
      \draw (m-3-4)-- (m-3-5);
      \draw (m-4-4)-- (m-4-5);
    \end{scope}
  \end{tikzpicture}
  \]
  The first two rows and the first two columns are exact by definition
  or by our description of the normal bundles of \(\zeta_1\)
  and~\(\iota_1\).  The third row is exact with the excess
  intersection bundle~\(\eta\)
  by~\eqref{eq:excess_bundle_exact_sequence}.  Hence the dotted arrow
  exists and makes the third row exact.  Since extensions of
  \(G\)\nb-vector bundles always split, we get
  \[
  \nu\oplus\eta \cong \pr_1^*(V_1).
  \]
  Since~\(\eta\) and~\(V_1\) are \(\K_G\)\nb-oriented, the
  bundle~\(\nu\) inherits a \(\K_G\)\nb-orientation.

  We apply Thom modification with the \(\K_G\)\nb-oriented
  \(G\)\nb-vector bundle~\(\nu\) to the correspondence
  in~\eqref{eq:composition_smooth_intersection}.  This gives the
  geometric correspondence
  \begin{equation}
    \label{ipe2}
    X \xleftarrow{b_1\circ\pr_1\circ\pi_\nu} \nu
    \xrightarrow{f_2\circ\pr_2\circ\pi_\nu} Z
  \end{equation}
  with \(\K\)\nb-theory datum
  \[
  \xi
  \defeq \tau_\nu\otimes \pi_\nu^*\bigl(e(\eta)\otimes \pr_1^*(\xi_1)
  \otimes \pr_2^*(\xi_2)\bigr)\in \RK^*_{G,X}(\nu).
  \]

  The Tubular Neighbourhood Theorem gives a \(G\)\nb-equivariant open
  embedding \(\hat{\zeta}_0\colon \nu \to V_1\times_Y M_2\) onto some
  \(G\)\nb-invariant open neighbourhood of \(M\) (see
  \cite{Emerson-Meyer:Normal_maps}*{Theorem 3.18}).

  We may find an open \(G\)\nb-invariant neighbourhood~\(U\) of the
  zero section in~\(V_1\) such that \(U\times_Y M_2\subseteq
  V_1\times_Y M_2\) is contained in the image of~\(\hat\zeta_0\) and
  relatively \(M\)\nb-compact.  We may choose the Thom class
  \(\tau_{V_1}\in\K^{\dim V_1}_G(V_1)\) to be supported in~\(U\).
  Hence we may assume that \(\pr_1^*(\tau_{V_1})\), the pull-back
  of~\(\tau_{V_1}\) along the coordinate projection \(\pr_1^*\colon
  V_1\times_Y M_2\to V_1\), is supported inside a relatively
  \(M\)\nb-compact subset of~\(\hat{\zeta}_0(\nu)\).

  Then \cite{Emerson-Meyer:Correspondences}*{Example 2.14} provides a
  bordism between the cycle in~\eqref{ipe1} and
  \begin{equation}
    \label{eq:ipe1_modified}
    X \xleftarrow{b_1\circ\pi_{V_1}\circ\pr_1\circ\hat{\zeta}_0} \nu
    \xrightarrow{f_2\circ\pr_2\circ\hat{\zeta}_0} Z,
  \end{equation}
  with \(\K\)\nb-theory class \(\pr_1^*(\tau_{V_1}) \otimes
  \hat{\zeta}_0^*\pr_1^*\pi_{V_1}^*(\xi_1)\otimes
  \hat{\zeta}_0^*\pr_2^*(\xi_2)\).

  Let \(s_t\colon \nu\to\nu\) be the scalar multiplication by
  \(t\in[0,1]\).  Composition with~\(s_t\) is a \(G\)\nb-equivariant
  homotopy
  \[
  \pi_{V_1}\pr_1\hat{\zeta}_0\sim\pr_1\pi_\nu\colon \nu\to M_1,\qquad
  \pr_2\hat{\zeta}_0\sim\pr_2\pi_\nu\colon \nu\to M_2.
  \]
  Hence \(s_t^*(\hat{\zeta}_0^*\pr_1^*\pi_{V_1}^*(\xi_1)\otimes
  \hat{\zeta}_0^*\pr_2^*(\xi_2))\) is a \(G\)\nb-equivariant homotopy
  \[
  \hat{\zeta}_0^*\pr_1^*\pi_{V_1}^*(\xi_1)\otimes
  \hat{\zeta}_0^*\pr_2^*(\xi_2) \sim
  \pi_\nu^*\bigl(\pr_1^*(\xi_1) \otimes \pr_2^*(\xi_2)\bigr).
  \]
  When we tensor with \(\pr_1^*(\tau_{V_1})\), this homotopy has
  \(X\)\nb-compact support because the support
  of~\(\pr_1^*(\tau_{V_1})\) is relatively \(M\)\nb-compact.

  This gives a homotopy of geometric correspondences
  between~\eqref{ipe1} and the variant of~\eqref{ipe2} with
  \(\K\)\nb-theory datum
  \[
  \pr_1^*(\tau_{V_1}) \otimes
  \pi_\nu^*\pr_1^*(\xi_1)\otimes
  \pi_\nu^*\pr_2^*(\xi_2);
  \]
  the relative \(M\)\nb-compactness of the support of
  \(\pr_{V_1}^*(\tau_{V_1})\) ensures that the homotopy of
  \(\K_G\)\nb-cycles implicit here has \(X\)\nb-compact support.  (We
  use \cite{Emerson-Meyer:Correspondences}*{Lemma 2.12} here, but the
  statement of the lemma is unclear about the necessary compatibility
  between the homotopy and the support of~\(\xi\).)

  The \(\K\)\nb-theory class \(\pr_1^*(\tau_{V_1})\) in this
  formula is the restriction of the
  Thom class for the vector bundle \(\pr_1^*(V_1)\) over~\(M\)
  to~\(\nu\).  Since \(\pr_1^*(V_1)\cong \nu\oplus\eta\) and the Thom
  isomorphism for a direct sum bundle is the composition of the Thom
  isomorphisms for the factors, the Thom class of~\(\pr_1^*(V_1)\) is
  \(\pr_1^*(\tau_{V_1}) = \tau_\nu\otimes\tau_\eta\).  Restricting
  this to the subbundle~\(\nu\) gives \(\tau_\nu\otimes
  \pi_\nu^*(e(\eta))\).  Hence the \(\K\)\nb-theory classes that come
  from \eqref{ipe1} and~\eqref{ipe2} are equal.  This finishes the
  proof.
\end{proof}



\subsection{The geometric Lefschetz index formula}
\label{sec:geom_trace}

In this section we compute Lefschetz indices in the symmetric monoidal
category~\(\GKK^G\) for smooth \(G\)\nb-manifolds with boundary.  Our
computation is geometric and uses the intersection theory of
equivariant correspondences discussed in Sections \ref{subsec:trans}
and~\ref{subsec:ptrans}.

Let~\(X\) be a smooth compact \(G\)\nb-manifold, possibly with
boundary.  Let~\(\intr X\) be its interior.  Let
\begin{equation}
  \label{eq:correspondence_to_compute_trace}
  X\xleftarrow{b} M\xrightarrow{f} X,\qquad
  \xi\in\RK^*_{G,X}(M)
\end{equation}
be a \(\K_G\)\nb-oriented smooth geometric correspondence from~\(X\)
to itself, with~\(M\) of finite orbit type to ensure that \(f\colon
M\to X\) lifts to an essentially unique normally non-singular map.
Since~\(X\) is compact, \(\RK^*_{G,X}(M) = \K^*_G(M)\) is the usual
\(\K\)\nb-theory with compact support.  The \(\K_G\)\nb-orientation
for~\eqref{eq:correspondence_to_compute_trace} means a
\(\K_G\)\nb-orientation on the stable normal bundle of~\(f\).  This is
equivalent to giving a \(G\)\nb-vector bundle~\(V\) over~\(X\) and
\(\K_G\)\nb-orientations on \(\Tang M\oplus f^*(V)\) and \(\Tang X
\oplus V\).

If~\(X\) has a boundary, then the requirements for a smooth
correspondence are that~\(M\) be a smooth manifold with boundary of
finite orbit type, such that \(f(\partial M)\subseteq \partial X\)
and~\(f\) is transverse to~\(\partial X\).  This ensures that~\(f\)
has an essentially unique lift to a normally non-singular map
from~\(M\) to~\(X\) by
Proposition~\ref{pro:non-singular_smooth_with_boundary}.  Recall the
map \(\rho\colon X\to\intr X\), which is shrinking the collar
around~\(\partial X\).

\begin{theorem}
  \label{the:geo_trace_transversal}
  Let \(\alpha\in\GKK^G_i(X,X)\) be represented by a
  \(\K_G\)\nb-oriented smooth geometric correspondence as
  in~\eqref{eq:correspondence_to_compute_trace}.  Assume that \((\rho
  b,f)\colon M\to X\times X\) and the diagonal embedding \(X\to
  X\times X\) intersect smoothly with a
  \(\K_G\)\nb-oriented excess intersection bundle~\(\eta\).  Then
  \(Q_{\rho b,f} \defeq \{m\in M \mid \rho b(m)= f(m)\}\) is a smooth
  manifold without boundary.  For a certain canonical
  \(\K_G\)\nb-orientation on~\(Q_{\rho b,f}\),
  \(\Lef(\alpha)\in \GKK^G_i(X,\pt)\) is represented by the
  geometric correspondence \(X \leftarrow Q_{\rho b,f} \to
  \pt\) with \(\K\)\nb-theory class \(\xi|_{Q_{\rho
      b,f}}\otimes e(\eta)\) on~\(Q_{\rho b,f}\); here the map
  \(Q_{\rho b,f}\to X\) is given by \(m\mapsto \rho
  b(m)=f(m)\).

  The Lefschetz index of~\(\alpha\) in \(\GKK^G_i(\pt,\pt)\) is
  represented by the geometric correspondence \(\pt \leftarrow
  Q_{\rho b,f} \to \pt\) with \(\K_G\)\nb-theory class
  \(\xi|_{Q_{\rho b,f}}\otimes e(\eta)\) on~\(Q_{\rho b,f}\).

  The Lefschetz index of~\(\alpha\) is the index of the
  Dirac operator on~\(Q_{\rho b,f}\) with coefficients
  in~\(\xi|_{Q_{\rho b,f}}\otimes e(\eta)\).
\end{theorem}

\begin{proof}
  We abbreviate \(Q\defeq Q_{\rho b,f}\) throughout the proof.  We
  have \(Q\subseteq \intr M\) because \(\rho b(M)\subseteq
  \rho(X)\subseteq \intr X\) and \(f(\partial M) \subseteq \partial
  X\).  The intersection \(\intr M\times_{\intr X\times \intr X} \intr
  X\) is~\(Q\) and hence a smooth submanifold of~\(\intr M\).

  We compute \(\Lef(\alpha)\) using the dual of~\(X\)
  constructed in Theorem~\ref{the:compact_smooth_dualisable}.
  This involves a \(G\)\nb-vector bundle~\(\Nor X\) such that \(\Tang
  X\oplus \Nor X \cong X\times E\) for a \(\K_G\)\nb-oriented
  \(G\)\nb-vector space~\(E\).

  With the unit and counit from
  Theorem~\ref{the:compact_smooth_dualisable}, \(\Lef(\alpha)\)
  becomes the composition of the three geometric
  correspondences in the bottom zigzag in
  Figure~\ref{fig:Lefschetz_intersection}; here we already
  composed~\(\alpha\) with the multiplication correspondence,
  which simply composes~\(b\) with~\(\Delta\).
  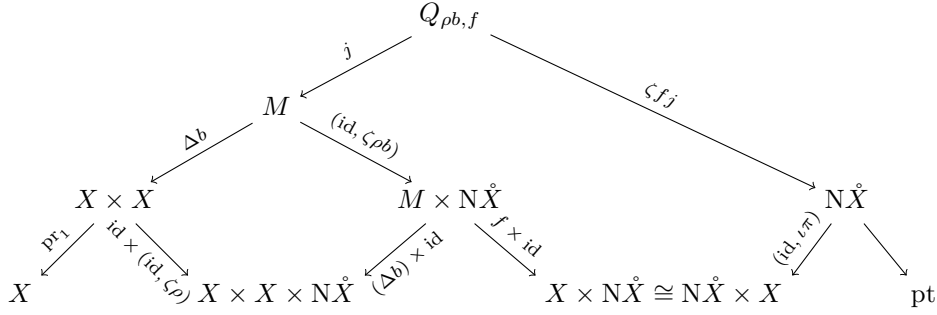
\begin{figure}
    \centering
    \begin{tikzpicture}[baseline=(current bounding box.west)]
      \matrix (m) [cd,column sep=.9em]{
        &&&Q_{\rho b,f}\\
        &&M\\
        &X\times X&&M\times \Nor\intr X&&\Nor\intr X\\
        X&&X\times X\times \Nor\intr X&&X\times \Nor\intr X \cong
        \Nor\intr X\times X&&\pt\\};
      \begin{scope}[corr]
        \draw (m-1-4)-- node      {$j$} (m-2-3);
        \draw (m-1-4)-- node      {$\zeta f j$} (m-3-6);
        \draw (m-2-3)-- node      {$\Delta b$} (m-3-2);
        \draw (m-2-3)-- node      {$(\id,\zeta\rho b)$} (m-3-4);
        \draw (m-3-2)-- node      {$\pr_1$} (m-4-1);
        \draw (m-3-2)-- node[blw] {$\id\times (\id,\zeta\rho)$} (m-4-3.north west);
        \draw (m-3-4)-- node[blw] {$(\Delta b)\times\id$} (m-4-3.north east);
        \draw (m-3-4)-- node      {$f\times\id$} (m-4-5.north west);
        \draw (m-3-6)-- node      {$(\id,\iota\pi)$} (m-4-5.north east);
        \draw (m-3-6)-- (m-4-7);
      \end{scope}
    \end{tikzpicture}
    \caption{The intersection diagram for the computation of
      \(\Lef(\alpha)\) in the proof of
      Theorem~\ref{the:geo_trace_transversal}.  Here \(j\colon
      Q_{\rho b,f} \to M\) denotes the inclusion map;
      \(\zeta\)~the zero section \(X\to\Nor X\) or \(\intr
      X\to\Nor\intr X\); \(\pi\colon \Nor\intr X\to \intr X\)
      the bundle projection; \(\iota\colon \intr X\to X\) the
      inclusion; \(\Delta\colon X\to X\times X\) the diagonal
      embedding; \(\pr_1\colon X\times X\to X\) the projection
      onto the first factor.}
    \label{fig:Lefschetz_intersection}
  \end{figure}

  We first consider the small left square.  Computing its intersection
  space naively gives~\(M\), which is a manifold with boundary.  We
  would hope that this square is Cartesian.  But \(X\times X\) is only
  a manifold with corners if~\(X\) has a boundary, and we we did not
  discuss smooth correspondences in this generality.  Hence we check
  directly that the composition of the correspondences from \(X\) to
  \(X\times X\times \Nor\intr X\) and on to \(M\times \Nor \intr X\)
  is represented by \(X\leftarrow M\rightarrow M\times \Nor \intr X\).

  The manifold~\(\Nor\intr X\) is an open subset of~\(E\) by
  construction.  Hence the map
  \[
  \id\times (\id,\zeta\rho b)\colon X\times X\to
  X\times X\times\Nor\intr X
  \]
  extends to an open embedding
  \[
  \psi\colon X\times X\times E\to X\times X\times \Nor\intr
  X,\qquad
  (x_1,x_2,e)\mapsto
  \bigl(x_1,x_2,\zeta\rho(x_2)+h_{x_2}(\norm{e}^2)\cdot e\bigr),
  \]
  where \(h_{x_2}\colon \R_+\to\R_+\) is a diffeomorphism onto
  a bounded interval~\([0,t)\) depending smoothly and
  \(G\)\nb-invariantly on~\(x_2\), such that the \(t\)\nb-ball
  in~\(E\) around \(\zeta\rho(x_2)\in\Nor\intr X\) is contained
  in~\(\Nor\intr X\).

  The map~\(\psi\) gives a special
  correspondence
  \[
  X\xleftarrow{\pr_1\circ\pi_E} X\times X\times E
  \xrightarrow{\psi} X\times X\times\Nor\intr X
  \]
  with \(\K\)\nb-theory class the pull-back of the Thom class
  of~\(E\).  This is equivalent to the given correspondence
  from~\(X\) to \(X\times X\times\Nor\intr X\) because of a
  Thom modification for the trivial vector bundle~\(E\) and a
  homotopy.  In particular, the \(\K_G\)\nb-orientation of
  \(\id\times(\id,\zeta\rho)\) that is implicit here is the one
  that we get from the \(\K_G\)\nb-orientation in the proof of
  Theorem~\ref{the:compact_smooth_dualisable}.

  For a special correspondence, the intersection always gives
  the composition product.  Here we get the space
  \begin{multline*}
    \bigl\{(x_1,x_2,e,m,y,\mu)\in X\times X\times E\times M\times
    \Nor\intr X \bigm|
    \\(x_1,x_2,\zeta\rho(x_2)+h_{x_2}(\norm{e}^2)\cdot e)) =
    (b(m),b(m),y,\mu)\bigr\}.
  \end{multline*}
  That is, \(x_1=x_2=b(m)\), \((y,\mu)=\rho
  b(m)+h_{b(m)}(\norm{e}^2)\cdot e)\).  Since \(m\in M\) and
  \(e\in E\) may be arbitrary and determine the other
  variables, we may identify this space with \(M\times E\).

  In the same way, we may replace
  \begin{equation}
    \label{eq:left_square_in_Lef_formula}
    X \xleftarrow{b} M \xrightarrow{(\id,\zeta\rho b)}
    M\times\Nor\intr X
  \end{equation}
  by an equivalent special correspondence with space \(M\times
  E\) in the middle.  This gives exactly the composition
  computed above.  Hence~\eqref{eq:left_square_in_Lef_formula}
  also represents the composition of the correspondences
  from~\(X\) to~\(M\times\Nor\intr X\) in
  Figure~\ref{fig:Lefschetz_intersection}.

  Composing further with \(f\times\id\) simply composes
  \(\K_G\)\nb-oriented normally non-singular maps.  Since we
  are now in the world of manifolds with boundary, we may
  identify smooth maps and smooth normally non-singular maps.
  The large right square contains the \(G\)\nb-maps
  \begin{align*}
    (f,\zeta\rho b)= (f\times\id)\circ (\id,\zeta\rho b)\colon
    M&\to X\times\Nor\intr X,\\
    (\iota\pi,\id)\colon \Nor\intr X&\to X\times\Nor\intr X.
  \end{align*}
  The pull-back contains those \((m,x,\mu) \in M\times \Nor\intr X\)
  with \((f(m),\rho b(x),0) = (x,x,\mu)\) in \(X\times\Nor\intr X\).
  This is equivalent to \(x=f(m) = \rho b(m)\) and \(\mu=0\), so that
  the pull-back is~\(Q\).  Since all vectors tangent to
  the fibres of~\(\Nor\intr X\) are in the image of
  \(D(\iota\pi,\id)\), the intersection is smooth and the excess
  intersection bundle is the same bundle~\(\eta\) as for \((f,\rho
  b)\colon \intr M\to \intr X\times \intr X\) and \(\delta\colon \intr
  X\to \intr X\times \intr X\).  Hence the right square is
  \(\eta\)\nb-Cartesian.

  Theorem~\ref{theorem:composition} shows that~\(\Lef(\alpha)\)
  is represented by a correspondence of the form \(X
  \xleftarrow{bj} Q \to \pt\), with a suitable class
  in~\(\K^*_G(Q)\) and a suitable \(\K_G\)\nb-orientation on
  the map \(Q\to\pt\) or, equivalently, the manifold~\(Q\).
  Here we may replace~\(bj\) by the properly homotopic map
  \(\rho bj=fj\).  It remains to describe the \(\K\)\nb-theory
  and orientation data.

  First, the given \(\K\)\nb-theory class~\(\xi\) on~\(M\) is pulled
  back to \(\xi\otimes1\) on \(M\times \Nor\intr X\) when we take the
  exterior product with~\(\Nor\intr X\).  In the intersection product,
  this is pulled back to~\(M\) along \((\id,\zeta\rho b)\),
  giving~\(\xi\) again, and then to~\(Q\) along~\(j\), giving the
  restriction of~\(\xi\) to \(Q\subseteq M\).  The unit and counit
  have~\(1\) as its \(\K\)\nb-theory datum.  Thus the Lefschetz
  index has \(\xi|_Q\otimes e(\eta)\in \K^*_G(Q)\) as its
  \(\K\)\nb-theory datum by Theorem~\ref{theorem:composition}.

  The given \(\K_G\)\nb-orientations on \(E\), \(f\) and~\(\eta\)
  induce \(\K_G\)\nb-orientations on all maps in
  Figure~\ref{fig:Lefschetz_intersection} that point to the right.  This
  is the \(\K_G\)\nb-orientation on the map \(Q\to\pt\) that we need.
  We describe it in greater detail after the proof of the theorem.

  The \(\K_G\)\nb-orientation on the map \(Q\to\pt\) is equivalent to
  a \(G\)\nb-equivariant Spin\(^\textup{c}\)-structure on~\(Q\).  The
  isomorphism
  \[
  \GKK^G_*(\pt,\pt) \to \GKK^G_*(\Cont(\pt),\Cont(\pt))
  \]
  described in \cite{Emerson-Meyer:Correspondences}*{Theorem 4.2} maps
  the geometric correspondence just described to the index of the
  Dirac operator on~\(Q\) for the chosen Spin\(^\textup{c}\)-structure
  twisted by \(\xi|_Q\otimes e(\eta)\).  This gives the last assertion
  of the theorem.
\end{proof}

Since the \(\K_G\)\nb-orientation on~\(Q_{\rho b,f}\) is necessary for
computations, we describe it more explicitly now.  We still use the
notation from the previous proof.

We are given \(\K_G\)\nb-orientations on \(E\), \(f\) and~\(\eta\).
The \(\K_G\)\nb-orientation on~\(f\) is equivalent to one on the
\(G\)\nb-vector bundle \(\Tang M \oplus f^*(\Nor X)\) over~\(M\)
because
\[
\Tang X\oplus\Nor X\cong X\times E
\]
is a \(\K_G\)\nb-oriented \(G\)\nb-vector bundle on~\(X\).

We already discussed during the proof of the theorem that
\(\id\times (\id,\zeta\rho)\) and \((\id,\zeta\rho)\) are
normally non-singular embeddings with normal bundle~\(E\); this
gives the correct \(\K_G\)\nb-orientation for these maps as
well.

A \(\K_G\)\nb-orientation on the map \((f,\zeta\rho b)\colon M\to
X\times \Nor\intr X\) is equivalent to one for \(\Tang M \oplus
f^*(\Nor X)\) because the bundle \(\Tang(X\times\Nor\intr
X)\oplus\pr_1^*(\Nor X)\) over \(X\times\Nor\intr X\) is isomorphic to
the trivial bundle with fibre~\(E\oplus E\) and \((f,\zeta\rho
b)^*\pr_1^*(\Nor X)=f^*(\Nor X)\).  We are already given such a
\(\K_G\)\nb-orientation from the \(\K_G\)\nb-orientation of~\(f\).

\begin{lemma}
  \label{lem:KG_orient_in_intersection}
  The given \(\K_G\)\nb-orientation on \(\Tang M \oplus f^*(\Nor X)\)
  is also the one that we get by inducing \(\K_G\)\nb-orientations on
  \((\id,\zeta\rho b)\) from \((\id,\zeta\rho)\) and on \(f\times\id\)
  from~\(f\) and then composing.
\end{lemma}

\begin{proof}
  The \(\K_G\)\nb-orientation of~\(f\) induces one for
  \(f\times\id\), which is equivalent to a
  \(\K_G\)\nb-orientation for
  \[
  \Tang(M\times\Nor\intr X) \oplus (f\pr_1)^*(\Nor X)
  \cong \bigl(\Tang M\oplus f^*(\Nor X)\bigr)\times (\Nor\intr X\times E).
  \]
  This \(\K_G\)\nb-orientation is exactly the direct sum
  orientation from \(\Tang M\oplus f^*(\Nor X)\) and~\(E\); no
  sign appears in changing the order because~\(E\) has even
  dimension.

  The map \(h=(\id,\zeta\rho)\) is a smooth embedding with normal
  bundle~\(E\).  Hence we get an extension of vector bundles
  \[
  \Tang M \oplus f^*(\Nor X) \mono
  h^*\bigl(\Tang(M\times\Nor\intr X) \oplus (f\pr_1)^*(\Nor
  X)\bigr)
  \epi E.
  \]
  The given \(\K_G\)\nb-orientations on \(\Tang M \oplus
  f^*(\Nor X)\) and~\(E\) induce one on the vector bundle in
  the middle.  This is the same one as the pull-back of the one
  constructed above.  This means that the
  \(\K_G\)\nb-orientation on \(\Tang M \oplus f^*(\Nor X)\)
  induced by~\(h\) is the given one.
\end{proof}

Equation~\eqref{eq:excess_bundle_exact_sequence} provides the
following exact sequence of vector bundles over~\(Q\):
\begin{multline*}
  0\to \Tang Q \xrightarrow{Dj,D(\zeta fj)}
  j^*(\Tang M)\oplus (\zeta fj)^*\Tang(\Nor\intr X)
  \\\xrightarrow{D(f,\zeta\rho b),-D(\iota\pi,\id)}
  (f,\zeta\rho b)^*\Tang(X\times\Nor\intr X) \to \eta \to 0.
\end{multline*}
Since \(-D(\iota\pi,\id)\) is injective, we may divide out
\(\Tang(\Nor\intr X)\) and its image to get the simpler short exact
sequence
\[
0\to \Tang Q \xrightarrow{Dj}
j^*\Tang M \xrightarrow{Df-D(\rho b)}
f^*\Tang X \to \eta \to 0.
\]
Then we add the identity map on \(j^*f^*(\Nor X)\) to get
\begin{equation}
  \label{eq:Tang_Q_K-orient}
  0\to \Tang Q \xrightarrow{(Dj,0)}
  j^*(\Tang M\oplus f^*\Nor X) \xrightarrow{(Df-D(\rho b),\id)}
  f^*(\Tang X\oplus\Nor X) \to \eta \to 0.
\end{equation}
In the last long exact sequence, the vector bundles \(j^*(\Tang
M\oplus f^*\Nor X)\), \(f^*(\Tang X\oplus\Nor X)\cong Q\times E\)
and~\(\eta\) carry \(\K_G\)\nb-orientations.  These together induce
one on~\(\Tang Q\).  This is the \(\K_G\)\nb-orientation that appears
in Theorem~\ref{the:geo_trace_transversal}.

Of course, the resulting geometric cycle should not depend on the
auxiliary choice of a \(\K_G\)\nb-orientation on~\(\eta\).  Indeed, if
we change it, then we change both \(e(\eta)\) and the
\(\K_G\)\nb-orientation on~\(\Tang Q\), and these changes cancel each
other.

We now consider some examples of Theorem~\ref{the:geo_trace_transversal}.

\subsubsection{Self-maps transverse to the identity map}
\label{sec:self-map_transverse}

Let~\(X\) be a compact \(G\)\nb-manifold with boundary and let
\(b\colon X\to X\) be a smooth \(G\)\nb-map that is transverse to the
identity map.  Thus~\(b\) has only finitely many isolated fixed points
and \(1-D_x b\colon \Tang_x X\to \Tang_x X\) is invertible for all
fixed points~\(x\) of~\(b\).  We turn~\(b\) into a geometric
correspondence~\(\alpha\) from~\(X\) to itself by taking
\(M=X\), \(f=\id\) (with
standard \(\K_G\)\nb-orientation) and \(\xi=1\).

Since~\(b\) has only finitely many fixed points, we may choose the
collar neighbourhood so small that all fixed points that do not lie
on~\(\partial X\) lie outside the collar neighbourhood, and such that
the fixed points of~\(\rho b\) are precisely the fixed points of~\(b\)
not on the boundary of~\(X\).  Hence \(\rho b=b\) near all fixed
points.

Then~\(\rho b\) is also transverse to the diagonal map and
Theorem~\ref{the:geo_trace_transversal} applies.  The intersection
space in Theorem~\ref{the:geo_trace_transversal} is
\[
Q = Q_{\rho b,\id} = \{x\in X\mid \rho b(x)=x\}
= \{x\in \intr X\mid b(x)=x\},
\]
the set of fixed points of~\(b\) in~\(\intr X\).  The \(\K\)\nb-theory
class on~\(Q\) is~\(1\) because \(\xi=1\) and the intersection is
transverse.  More precisely, the bundle~\(\eta\) is zero-dimensional,
and we may give it a trivial \(\K_G\)\nb-orientation for which
\(e(\eta)=1\).

Although~\(Q\) is discrete, the \(\K_G\)\nb-orientation of the
map \(Q\to\pt\) is important extra information: it provides the
signs that appear in the familiar Lefschetz fixed-point formula.
Equation~\eqref{eq:Tang_Q_K-orient} simplifies to
\[
0 \to \Tang Q \to (\Tang X\oplus\Nor X)|_Q
\xrightarrow{(\id-Db,\id)} (\Tang X\oplus\Nor X)|_Q \to 0.
\]
We left out~\(\eta\) because it is zero-dimensional and carries the
trivial \(\K_G\)\nb-orientation to ensure that \(e(\eta)=1\).  The
bundle~\(\Tang Q\) is also zero-dimensional.  But a zero-dimensional
bundle has non-trivial \(\K_G\)\nb-orientations.  The Clifford algebra
bundle of a zero-dimensional bundle is the trivial, trivially graded
one-dimensional bundle spanned by the unit section.  Thus an
irreducible Clifford module (spinor bundle) for it is the same as a
\(\Z/2\)\nb-graded \(G\)\nb-equivariant complex line bundle.

Let~\(S\) be the spinor bundle associated to the given
\(\K_G\)\nb-orientation on \(\Tang X\oplus\Nor X\cong E\).  The exact
sequence~\eqref{eq:Tang_Q_K-orient} says that the
\(\K_G\)\nb-orientation of~\(Q\) is the \(\Z/2\)\nb-graded
\(G\)\nb-equivariant complex line bundle~\(\ell\) such that
\((\id-Db)^*(S|_Q)\otimes\ell \cong S|_Q\) as Clifford modules.  This
uniquely determines~\(\ell\).  Thus~\(\ell\) measures whether~\(Db\)
changes orientation or not.  This is exactly the \emph{sign} of the
\(G\)\nb-equivariant vector bundle automorphism \(1-Db\) on~\(\Tang
X|_Q\), which is studied in detail
in~\cite{Emerson-Meyer:Equi_Lefschetz}.  In particular, it is shown
in~\cite{Emerson-Meyer:Equi_Lefschetz} that~\(\ell\) is the
complexification of a \(\Z/2\)\nb-graded \(G\)\nb-equivariant
\emph{real} line bundle.  The \(\Z/2\)\nb-grading gives one sign for
each \(G\)\nb-orbit in~\(Q\), namely, the index of \(\id-Db_x\).  In
addition, the sign gives a real character \(G_x\to\{-1,+1\}\) for each
orbit, where~\(G_x\) denotes the stabiliser of a point in the orbit.

Twisting the \(\K_G\)\nb-orientation by a line bundle over~\(Q\) has
the same effect as taking the trivial \(\K_G\)\nb-orientation and
putting this line bundle on~\(Q\).  Thus \(\Lef(\alpha)\) is
represented by the geometric correspondence
\[
X\leftarrow (Q,\sign(1-Db|_Q))\rightarrow \pt
\]
with the trivial \(\K_G\)\nb-orientation on the map
\(Q\to\pt\).

The Lefschetz index of~\(\alpha\) is the index of the Dirac
operator on~\(Q\) with coefficients in the line bundle
\(\sign(1-Db)|_Q\); this is simply the \(\Z/2\)\nb-graded
\(G\)\nb-representation on the space of sections of
\(\sign(1-Db)|_Q\), which is a certain finite-dimensional
\(\Z/2\)\nb-graded, real \(G\)\nb-representation.

If the group~\(G\) is trivial, then the Lefschetz index is a
number and \(\sign(1-Db)\) is the family of \(\sign(1-D_xb)\in
\{\pm1\}\) for \(x\in Q\).  If~\(X\) is connected, then all
maps \(X\leftarrow\pt\) give the same element in \(\GKK\).
Thus \(\Lef(\alpha)\) is \(\lefind(\alpha)\) times the point
evaluation class \([X\leftarrow \pt=\pt]\), and
\(\lefind(\alpha)\) is the sum of the indices of all fixed
points of~\(b\) in~\(\intr X\).

\subsubsection{Euler characteristics}
\label{sec:Euler_characteristics}

Now let \(\xi\in\K^*_G(X)\) and consider the correspondence
with \(M=X\), \(b=f=\id\), and the above class~\(\xi\).  We
want to compute the Lefschetz index of the geometric
correspondence~\(\alpha\)
associated to~\(\xi\).  In particular, for \(\xi=1\) we get the
Lefschetz index of the identity element in \(\GKK^G_0(X,X)\),
which is the Euler characteristic of~\(X\).

We only compute the Lefschetz index of \(\xi\in\K^*_G(X)\)
for~\(X\) with trivial boundary.  Then the map~\(\rho\) in
Theorem~\ref{the:geo_trace_transversal} is the identity map,
and~\(\id_X\) intersects itself smoothly.  The intersection
space is \(Q=X\), embedded diagonally into~\(X\times X\).  The
excess intersection bundle~\(\eta\) is~\(\Tang X\).  To apply
Theorem~\ref{the:geo_trace_transversal}, we also assume
that~\(X\) is \(\K_G\)\nb-oriented.  Then \(\Lef(\alpha)\) is
represented by the geometric correspondence
\[
X \xleftarrow{\id_X} (X,\xi \otimes e(\Tang X)) \rightarrow \pt.
\]
Here~\(e(\Tang X)\) and the map \(X\to\pt\) both use the same
\(\K_G\)\nb-orientation on~\(X\).  The Lefschetz index
of~\(\alpha\) is represented by
\[
\pt \leftarrow (X,\xi \otimes e(\Tang X)) \rightarrow \pt.
\]
By Theorem~\ref{the:geo_trace_transversal}, this is the index
of the Dirac operator of~\(X\) with coefficients in
\(\xi\otimes e(\Tang X)\).

Twisting the Dirac operator by~\(e(\Tang X)\) gives the de Rham
operator: this is the operator \(d+d^*\) on differential forms
with usual \(\Z/2\)\nb-grading, so that its index is the Euler
characteristic of~\(X\).  Thus (the analytic version of)
\(\Lef(\alpha)\) is the class in \(\KK^G_0(\Cont(X),\C)\) of
the de Rham operator with coefficients in~\(\xi\).  This was
proved already in~\cite{Emerson-Meyer:Euler} by computations in
Kasparov's analytic \(\KK\)-theory.  Now we have a purely
geometric proof of this fact, at least if~\(X\) is
\(\K_G\)\nb-oriented.

Theorem~\ref{the:geo_trace_transversal} no longer works for~\(X\)
without \(\K_G\)\nb-orientation because there is
no \(\K_G\)\nb-orientation on the excess intersection bundle.
A way around this restriction would be to use twisted
\(\K\)\nb-theory throughout.  We shall not pursue this here,
however.

We can now clarify the relationship between the Euler class
\(e(\Tang X)\in\K^{\dim(X)}_G(X)\) and the higher Euler
characteristic \(\Eul_X\in\KK^G_0(\Cont(X),\C)\) introduced
already in~\cite{Emerson-Meyer:Euler}.  Since we assume~\(X\)
\(\K_G\)\nb-oriented and without boundary, there is a duality
isomorphism \(\K^{\dim(X)}_G(X)\cong\K_0^G(X) =
\KK^G_0(\Cont(X),\C)\).  This duality isomorphism
maps~\(e(\Tang X)\) to~\(\Eul_X\).

\subsubsection{Self-maps without transversality}
\label{sec:self-map_not_transverse}

Let~\(X\) be a compact \(G\)\nb-manifold and let \(b\colon X\to
X\) be a smooth \(G\)\nb-map.  We want to compute the Lefschetz
map on the geometric correspondence
\[
X\xleftarrow{b} X \xrightarrow{\id_X} X
\]
with \(\K_G\)\nb-theory class~\(1\) on~\(X\).

If~\(b\) is transverse to the identity map, then this is done
already in Section~\ref{sec:self-map_transverse}.  The case
\(b=\id_X\) is done already in
Section~\ref{sec:Euler_characteristics}.  Now we assume that
\(b\) and~\(\id_X\) intersect smoothly.  We also assume
that~\(b\) has no fixed points on the boundary; then we may
choose the collar neighbourhood of~\(\partial X\) to contain no
fixed points of~\(b\), so that \(\rho(x)=x\) in a neighbourhood
of the fixed point subset of~\(b\).  Furthermore, all fixed
points of~\(\rho b\) are already fixed points of~\(b\).

That \(b\) and~\(\id_X\) intersect smoothly and away
from~\(\partial X\) means that
\[
Q\defeq \{x\in X\mid b(x)=x\} = \{x\in X\mid \rho b(x)=x\}
\]
is a smooth submanifold of~\(\intr X\) and that there is an
exact sequence of \(G\)\nb-vector bundles over~\(Q\):
\[
0 \to \Tang Q \to \Tang X|_Q \xrightarrow{1-D(\rho b)}
\Tang X|_Q \to \eta \to 0,
\]
where~\(\eta\) is the excess intersection bundle.

\begin{remark}
  The maps \(b\) and~\(\id_X\) always intersect smoothly if
  \(b\colon X \to X\) is isometric with respect to a Riemannian
  metric on~\(X\); the reason is that if~\(Db\) fixes a vector
  \((x,\xi)\) at a fixed point of~\(b\), then~\(b\) fixes the
  entire geodesic through~\(x\) in direction~\(\xi\).
\end{remark}

The vector bundles \(\Tang Q\) and~\(\eta\) are the kernel and
cokernel of the vector bundle endomorphism \(1-D(\rho b)\)
on~\(\Tang X|_Q\).  Since both are vector bundles, \(1-D(\rho
b)\) has locally constant rank.  We may split
\begin{align*}
  \Tang X|_Q &\cong
  \ker(\id-D(\rho b)) \oplus \operatorname{im}(\id-D(\rho b))
  = \Tang Q \oplus \operatorname{im}(\id-D(\rho b)),\\
  \Tang X|_Q &\cong
  \coker(\id-D(\rho b)) \oplus \operatorname{coim}(\id-D(\rho b))
  = \eta \oplus \operatorname{coim}(\id-D(\rho b)).
\end{align*}
Since \(\operatorname{im}(\varphi) \cong
\operatorname{coim}(\varphi)\) for any vector bundle
homomorphism, it follows that \(\eta\) and~\(\Tang Q\) are
stably isomorphic as \(G\)\nb-vector bundles.  Thus
\(\K_G\)\nb-orientations on one of them translate to
\(\K_G\)\nb-orientations on the other.

\begin{remark}
  \label{rem:stably_isomorphic_as_ker_coker}
  Given two stably isomorphic vector bundles, there is always a
  vector bundle endomorphism with these two as kernel and
  cokernel.  Hence we cannot expect \(\eta\) and~\(\Tang Q\) to
  be isomorphic.
\end{remark}


\begin{corollary}
  Let~\(X\) be a compact \(G\)\nb-manifold.  Let \(b\colon X
  \to X\) be a smooth \(G\)\nb-map without fixed points
  on~\(\partial X\), such that \(b\) and~\(\id_X\) intersect
  smoothly.  Let the fixed point submanifold~\(Q\) of~\(b\) be
  \(\K_G\)\nb-oriented, and equip the excess intersection
  bundle with the induced \(\K_G\)\nb-orientation.  Then the
  Lefschetz index of the geometric correspondence
  \[
  X\xleftarrow{b} X \xrightarrow{\id_X} X
  \]
  with \(\K_G\)\nb-theory class~\(1\) on~\(X\) is the index of
  the Dirac operator on~\(Q\) twisted by~\(e(\eta)\).
\end{corollary}

The Lefschetz map sends the correspondence above to
\[
X\xleftarrow{bj} Q \to \pt
\]
with \(\K\)\nb-theory class~\(e(\eta)\) on~\(Q\).

\subsubsection{Trace computation for standard correspondences}
\label{sec:trace_standard_correspondence}

By Corollary~\ref{cor:KK_by_duality}, any element of
\(\GKK^G_*(X,X)\) is represented by a correspondence of the
form
\[
X\xleftarrow{\iota\circ\pi\circ\pr_1} \Nor\intr X \times X
\xrightarrow{\pr_2} X
\]
for a unique \(\xi\in\K^*_G(\Nor\intr X\times X)\).  We may
view this as a standard form for an element in
\(\GKK^G_*(X,X)\).

The map \((\rho\circ\iota\circ\pi\circ\pr_1,\pr_2)=
(\rho\circ\pi)\times\id\colon \Nor\intr X\times X\to X\times
X\) is a submersion and hence transverse to the diagonal.  Thus
Theorem~\ref{the:geo_trace_transversal} applies.  The
space~\(Q_{\rho b,f}\) is the graph of \(\rho\pi\colon
\Nor\intr X\to X\).  Thus the Lefschetz map gives the geometric
correspondence
\[
X \leftarrow \Nor\intr X \rightarrow \pt,\qquad
\xi|_{\Nor\intr X}\in\K^*_G(\Nor\intr X),
\]
where we embed \(\Nor\intr X\to \Nor\intr X\times X\)
via~\((\id,\rho\pi)\) and use the canonical
\(\K_G\)\nb-orientation on~\(\Nor\intr X\).  The Lefschetz
index in \(\GKK^G_*(\pt,\pt) \cong \K^*_G(\pt)\) is computed
analytically as the \(G\)\nb-equivariant index of the Dirac
operator on~\(\Nor\intr X\) twisted by~\(\xi|_{\Nor\intr X}\).

\subsubsection{Trace computation for another standard form}
\label{sec:trace_another_standard_correspondence}

Assume now that~\(X\) has no boundary and is \(\K_G\)\nb-oriented.  As
we remarked at the end of Section~\ref{sec:duality}, any element of
\(\GKK^G_*(X,X)\) is represented by a correspondence
\[
X\xleftarrow{\pr_1} X \times X
\xrightarrow{\pr_2} X,\qquad
\xi\in\K^*_G(X\times X).
\]
The same computation as in
Section~\ref{sec:trace_standard_correspondence} shows that the
Lefschetz map sends this to
\[
X = X \rightarrow \pt,\qquad
\xi|_X\in\K^*_G(X),
\]
where~\(\xi|_X\) is for the diagonal embedding \(X\to X\times
X\).  Analytically, this is the \(\K_G\)\nb-homology class of
the Dirac operator on~\(X\) with coefficients~\(\xi|_X\).

\subsubsection{Homogeneous correspondences}
\label{sec:trace_homogeneous}

We call a self-correspondence \(X\xleftarrow{b}
M\xrightarrow{f} X\) \emph{homogeneous} if \(X\) and~\(M\) are
homogeneous \(G\)\nb-spaces.  That is, \(X\defeq G/H\) and
\(M\defeq G/L\) for closed subgroups \(H, L\subseteq G\).  Then
there are elements \(t_b,t_f\in G\) with \(b(gL)\defeq gt_bH\),
\(f(gL) \defeq gt_fH\); we need \(L\subseteq t_bHt_b^{-1}\cap
t_fHt_f^{-1}\) for this to be well-defined.  Since \(G/L\cong
G/t_f^{-1}Lt_f\) by \(gL \mapsto gLt_f\), any homogeneous
correspondence is isomorphic to one with \(t_f=1\), so that
\(L\subseteq H\).  We assume this from now on and abbreviate
\(t=t_b\).

Since \(M\) and~\(X\) are compact, the relevant
\(\K\)\nb-theory group \(\RK^*_{G,X}(M)\) for a homogeneous
correspondence is just \(\K^*_G(M)\).  The induction
isomorphism gives \(\RK^*_{G,X}(M) = \K^*_G(G/L) \cong
\K^*_L(\pt)\).

A \(\K_G\)\nb-orientation for \(f\colon G/L\to G/H\) is
equivalent to a \(\K^H\)\nb-orientation for the projection map
\(H/L\to \pt\) because~\(f\) is obtained from this \(H\)\nb-map
by induction.  Thus we must assume an \(\K^H\)\nb-orientation
on~\(H/L\).  Equivalently, the representation of~\(L\) on
\(\Tang_{1L}(H/L)\) factors through Spin\(^\textup c\).  This
tangent space is the quotient~\(\mathfrak{h/l}\), where
\(\mathfrak{h}\) and~\(\mathfrak{l}\) denote the Lie algebras
of \(H\) and~\(L\), respectively.

Let \(L'\defeq H\cap tHt^{-1}\).  Then \(L\subseteq L'\) and
both maps \(f,b\colon G/L\to G/H\) factor through the quotient
map \(p\colon G/L\to G/L'\).  The geometric correspondence
\[
G/H \xleftarrow{b} G/L \xrightarrow{f} G/H,\qquad
\xi\in\K^*_G(G/L)
\]
is equivalent to the geometric correspondence
\[
G/H \xleftarrow{b'} G/L' \xrightarrow{f'} G/H,\qquad
\xi'\in\K^*_G(G/L')
\]
with \(\xi'\defeq p_!(\xi)\) and \(b'(gL') = gtH\) and
\(f'(gL') = gH\).  (To construct the equivalence, we first need
a normally non-singular map lifting~\(p\); then we apply
vector bundle modifications on the domain and target of~\(p\)
to replace~\(p\) by an open embedding; finally, for an open
embedding we may construct a bordism as in
\cite{Emerson-Meyer:Correspondences}*{Example 2.14}.)

Thus we may further normalise a homogeneous geometric
self-correspondence to one with \(L=H \cap tHt^{-1}\).

Now we compute the Lefschetz map for a such a normalised homogeneous
self-correspondence.

First let \(t\notin H\).  Then the image of the map
\((f,b)\colon G/L\to G/H\times G/H\) does not intersect the
diagonal.  Hence \((f,b)\) is transverse to the diagonal and
the coincidence space~\(Q_{b,f}\) is empty.  Thus the
Lefschetz map vanishes on a homogeneous correspondence with \(t\notin
H\) by Theorem~\ref{the:geo_trace_transversal}.

Now let \(t\in H\).  Then \(b=f\colon G/L\to G/H\) is the
canonical projection map.  Our normalisation condition yields
\(L=H\) and \(b=f=\id\) in this case; that is, our geometric
correspondence is the class in \(\GKK^G_*(G/H,G/H)\) of some
\(\xi\in\K^*_G(G/H)\).  Thus we have a special case of the Euler
characteristic computation in
Section~\ref{sec:Euler_characteristics}.  The Lefschetz map gives the
class of the geometric correspondence
\[
G/H \xleftarrow[=]{\id} (G/H,e(\Tang G/H)\otimes \xi) \to \pt,\qquad
\]
provided~\(G/H\) is \(\K_G\)\nb-oriented.  The Lefschetz index is the
index of the de Rham operator with coefficients in~\(\xi\).

When we identify \(\K^*_G(G/H)\cong \K^*_H(\pt)\), the Lefschetz index
becomes a map
\[
\K^*_H(\pt) \to \K^*_G(\pt).
\]
In complex \(\K\)\nb-theory, this is a map \(\Rep(H)\to\Rep(G)\).
Graeme Segal studied this map
in~\cite{Segal:Representation_ring}*{Section~2}, where it was denoted
by~\(i_!\).

For instance, assume~\(G\) to be connected and let \(H=L\) be
its maximal torus.  Let \(t\in W\defeq N_GH/H\), the Weyl group
of~\(G\).  Assume that we are working with complex
\(\K\)\nb-theory, so that \(\K^*_G(G/H) \cong \K^*_H(\pt) \cong
\Rep(H)\).  The Weyl group~\(W\) acts on~\(G/H\) by right
translations; these are \(G\)\nb-equivariant maps.  Taking the
correspondences \(X\xleftarrow{w^{-1}} X=X\), this gives a
representation \(W\to\GKK^G_0(G/H,G/H)\).  We also map
\(\Rep(H)\cong \K^0_G(G/H)\to\GKK^G_0(G/H,G/H)\) using the
correspondences \(X=(X,\xi)=X\).  These representations of
\(W\) and~\(\Rep(H)\) are a covariant pair of representations
with respect to the canonical action of~\(W\) on~\(\Rep(H)\)
induced by the automorphisms \(h\mapsto whw^{-1}\) of~\(H\) for
\(w\in W\).  Hence we map
\[
\Rep(H)\rtimes W\to\GKK^G_0(G/H,G/H).
\]
The Lefschetz index \(\Rep(H)\rtimes W\to\Rep(G)\) maps
\(a\cdot t\mapsto 0\) for \(t\in W\setminus\{1\}\) and \(a\cdot
1\mapsto \ind_G \Lambda_a\), where~\(\Lambda_a\) means the de
Rham operator on~\(G/H\) twisted by~\(a\).

\subsection{Fixed points submanifolds for torus actions}
\label{sec:Weyl_CF}

As another application of our excess intersection formula, we reprove
a result that is used in a recent article by Block and
Higson~\cite{Block-Higson:Weyl} to reformulate the Weyl
Character Formula in \(\KK\)-theory.

Block and Higson also develop a more geometric framework for
equivariant \(\KK\)-theory for a compact group.  For two locally
compact \(G\)\nb-spaces \(X\) and~\(Y\), they identify
\(\KK^G_*(X,Y)\) with the group of continuous natural transformations
\(\Phi_Z\colon \K^*_G(X\times Z) \to \K^*_G(Y\times Z)\) for all
compact \(G\)\nb-spaces~\(Z\); here continuity means that
each~\(\Phi_Z\) is a \(\K^*_G(Z)\)-module homomorphism.  The Kasparov
product then becomes the composition of natural transformations.  This
reduces Kasparov's equivariant \(\KK\)-theory to equivariant
\(\K\)\nb-theory.

The theory \(\GKK^G\) does more: it contains the knowledge that all
such natural transformations come from geometric correspondences, when
geometric correspondences give the same natural transformation, and
how to compose geometric correspondences.  Thus we get a more concrete
\(\KK\)-theory.

\begin{theorem}
  \label{the:fixed_torus_equivalent}
  Let~\(T\) be a compact torus and let~\(X\) be a smooth,
  \(\K_T\)\nb-oriented \(T\)\nb-manifold with boundary.  Let~\(e(\Tang
  X)\in\K^0_T(X)\) be the Euler class of~\(X\) for the chosen
  \(\K_T\)\nb-orientation.  Let \(F\subseteq X\) be the fixed-point
  subset of the \(T\)\nb-action on~\(X\) and let \(j\colon F\to X\) be
  the inclusion map.  Then~\(F\) is again a smooth \(\K\)\nb-oriented
  manifold with boundary, with trivial \(T\)\nb-action, so that the
  inclusion map~\(j\) is \(\K_T\)\nb-oriented.  Let~\(e(\Tang
  F)\in\K^0(F)\subseteq\K^0_T(F)\) be the Euler class of~\(F\).  The
  two geometric correspondences
  \begin{gather*}
    X \xleftarrow{\id_X} (X,e(\Tang X)) \xrightarrow{\id_X} X,\\
    X \xleftarrow{j} (F,e(\Tang F)) \xrightarrow{j} X
  \end{gather*}
  represent the same element in \(\GKK^T_0(X,X)\).
\end{theorem}

This is a generalisation of \cite{Block-Higson:Weyl}*{Lemma 3.1}.  We allow
Spin\(^\textup{c}\)-manifolds instead of complex manifolds.  For a
Spin\(^\textup{c}\)-structure coming from a complex structure, the
Euler class is \([\Lambda^* \Tang^* X]\in\K^0_T(X)\), which appears
in~\cite{Block-Higson:Weyl}.  The following proof is a translation of the
proof in~\cite{Block-Higson:Weyl} into the category~\(\GKK^G\).

\begin{proof}
  The first geometric correspondence above, involving the Euler class
  of~\(X\), is represented by the composition of geometric
  correspondences
  \[
  X\xleftarrow[=]{\id_X} X \xrightarrow{\zeta} \Tang X
  \xleftarrow{\zeta} X \xrightarrow[=]{\id_X} X
  \]
  by Example~\ref{exa:defect_explained}; here~\(\zeta\) denotes the
  zero section, which is \(\K_T\)\nb-oriented using the given
  \(\K_T\)\nb-orientation on the \(T\)\nb-vector bundle~\(\Tang X\).

  Choose a generic element~\(\xi\) in the Lie algebra of~\(T\), that
  is, the one-parameter group~\(\exp(s\xi)\), \(s\in\R\), is dense
  in~\(T\).  Let \(\alpha_t\colon X\to X\) denote the action of \(t\in
  T\) on~\(X\).  The action of~\(T\) maps~\(\xi\) to a vector field
  \(\alpha_\xi\colon X\to\Tang X\).  There is a homotopy of geometric
  correspondences
  \[
  X\xleftarrow[=]{\id_X} X \xrightarrow{\alpha_{s\xi}} \Tang X
  \xleftarrow{\zeta} X \xrightarrow[=]{\id_X} X
  \]
  for \(s\in[0,1]\).  For \(t=0\) we get the composition above,
  involving \(e(\Tang X)\).  We claim that for \(s=1\), the two
  correspondences intersect smoothly and that the intersection product
  is the second geometric correspondence in the theorem, involving
  \(F\) and its Euler class.

  First we show that the fixed-point submanifold~\(F\) is a closed
  submanifold.  Equip~\(X\) with a \(T\)\nb-invariant Riemannian
  metric.  Let \(x\in F\), that is, \(\alpha_t(x)=x\) for all \(t\in
  T\).  Split~\(\Tang_x X\) into
  \[
  V=\{v\in \Tang_x X \mid D\alpha_t(x,v)=(x,v)\text{ for all }t\in T\}
  \]
  and its orthogonal complement~\(V^\bot\).  Since the metric is
  \(T\)\nb-invariant, \(\alpha_t(\exp(x,v))=\exp(D\alpha_t(x,v))\) for
  all \(v\in\Tang_x X\).  Since the exponential mapping restricts to a
  diffeomorphism between a neighbourhood of~\(0\) in~\(\Tang_x X\) and
  a neighbourhood of~\(x\) in~\(X\), we have \(\exp(x,v)\in F\) if
  \(v\in V\), and the converse holds for~\(v\) in a suitable
  neighbourhood of~\(0\).  Thus we get a closed submanifold chart
  for~\(F\) near~\(x\) with \(\Tang_x F=V\).  Hence~\(F\) is a closed
  submanifold with
  \[
  \Tang F = \{(x,v)\in\Tang X\mid
  D\alpha_t(x,v)=(x,v)\text{ for all }t\in T\}.
  \]

  Since~\(\xi\) is generic, \(\alpha_\xi(x)=0\) in~\(\Tang_x X\) if
  and only if \(x\in F\).  Thus~\(F\) is the coincidence space of the
  pair of maps \(\zeta,\alpha_\xi\colon X\to\Tang X\).  Let \(x\in F\)
  and let \(v_1,v_2\in\Tang_x X\) satisfy
  \(D\zeta(x,v_1)=D\alpha_\xi(x,v_2)\).  Then \(v_1=v_2\) by taking
  the horizontal components; and the vertical component of
  \(D\alpha_\xi(x,v_2)\) vanishes, which means that
  \(D\alpha_{\exp(s\xi)}(x,v_2)=(x,v_2)\) for all \(s\in\R\).  Hence
  \(v_2\in\Tang_x F\).  This proves that \(\zeta\) and~\(\alpha_\xi\)
  intersect smoothly.  The excess intersection bundle is the cokernel
  of \(D\alpha_{\exp(s\xi)}-\id\); since the action of~\(T\) is by
  isometries, \(D\alpha_{\exp(s\xi)}-\id\) is normal in each fibre, so
  that its image and kernel are orthogonal complements.  Hence the
  cokernel is canonically isomorphic to the kernel of
  \(D\alpha_{\exp(s\xi)}-\id\).  Thus the excess intersection bundle
  is canonically isomorphic to~\(\Tang F\).

  Hence Theorem~\ref{theorem:composition} gives the geometric
  correspondence \(X \xleftarrow{j} (F,e(\Tang F)) \xrightarrow{j} X\)
  as the composition, as desired.
\end{proof}


\section{The homological Lefschetz index of a Kasparov morphism}
\label{sec:trace_localisation}

The example in Section~\ref{sec:self-map_transverse} shows in what
sense the geometric Lefschetz index computations in Section~\ref{sec:traces_GKK}
generalise the local fixed-point formula for the Lefschetz index of a
self-map.  Now
we turn to generalisations of the global homological formula for the
Lefschetz index.

The classical Lefschetz fixed-point formula for a self-map
\(f\colon X\to X\) contains the (super)trace of the map on the
cohomology of~\(X\) with rational coefficients induced
by~\(f\).  We take rational coefficients in order to get vector
spaces over a field, where there is a good notion of trace for
endomorphisms.  By
the Chern character, we may as well take \(\K^*(X)\otimes\Q\) instead
of rational cohomology.  It is checked in~\cite{Emerson:Lefschetz}
that the Lefschetz index of \(f\in\KK_0(A,A)\) for a dualisable
\(\Cst\)\nb-algebra~\(A\) in the bootstrap class is equal to the
supertrace of the map on \(\K_*(A)\otimes\Q\) induced by~\(f\).

We are going to generalise this result to the equivariant
situation for a compact Lie group~\(G\).  We assume that we are
working with complex \(\Cst\)\nb-algebras, so that
\(\GKK^G_*(\pt,\pt)= \KK^G_*(\C,\C)\) vanishes in odd degrees
and is the representation ring~\(\Rep(G)\) in even degrees.
Our methods do not apply to the torsion invariants in
\(\KK^G_d(\R,\R)\) for \(d\neq0\) in the real case because we
(implicitly) tensor everything
with~\(\Q\) to simplify the Lefschetz index.

Furthermore, we work in~\(\KK^G\) instead of~\(\GKK^G\) in this
section because the category~\(\KK^G\) is triangulated,
unlike~\(\GKK^G\).  We explain in
Remark~\ref{rem:bootstrap_nice_model} why~\(\GKK^G\) is not
triangulated; the triangulated structure on~\(\KK^G\) is introduced
in~\cite{Meyer-Nest:BC_Localization}.

Let \(S\subseteq\Rep(G)\) be the set of all elements that are
not zero divisors.  This is a saturated, multiplicatively closed
subset; even more, it is the largest multiplicatively closed
subset for which the canonical map \(\Rep(G)\to S^{-1}\Rep(G)\)
to the ring of fractions is injective (see
\cite{Atiyah-Macdonald:Commutative}*{Exercise 9 on p.~44}).
The localisation \(S^{-1}\Rep(G)\) is also called the
\emph{total ring of fractions} of~\(\Rep(G)\).

Since \(\KK^G\) is symmetric monoidal with unit \(\Unit =
\C\) and \(\Rep(G)=\KK^G_0(\C,\C)\), the
category~\(\KK^G\) is \(\Rep(G)\)-linear.  Hence we may
localise it at~\(S\) as
in~\cite{Inassaridze-Kandelaki-Meyer:Finite_Torsion_KK}.  The
resulting category \(\Tri \defeq S^{-1}\KK^G\) has the same
objects as \(\KK^G\) and arrows
\[
\Tri_*(A,B) \defeq S^{-1}\KK^G_*(A,B)
= S^{-1}\Rep(G) \otimes_{\Rep(G)} \KK^G_*(A,B).
\]
The category~\(\Tri\) is \(S^{-1}\Rep(G)\)-linear.  There is an
obvious functor \(\natural\colon \KK^G\to\Tri\).

If~\(A\) is a separable \(G\)\nb-\(\Cst\)-algebra, then
\[
\Tri_*(\C,A) = S^{-1}\KK^G_*(\C,A)
\cong S^{-1}\Rep(G) \otimes_{\Rep(G)} \K^G_*(A),
\]
where we use the usual \(\Rep(G)\)-module structure on
\(\K^G_*(A) \cong \KK^G_*(\C,A)\).

There is a unique symmetric monoidal structure on~\(\Tri\) for
which~\(\natural\) is a strict symmetric monoidal functor:
simply extend the exterior tensor product on \(\KK^G\)
\(S^{-1}\Rep(G)\)-linearly.  Hence if~\(A\) is dualisable
in~\(\KK^G\), then its image in~\(\Tri\) is dualisable as
well, and
\[
\natural(\tr f) = \tr (\natural f)\qquad
\text{for all \(f\in\KK^G_*(A,A)\).}
\]
The crucial point for us is that \(\natural \tr(f) =
\tr(\natural f)\) uniquely determines \(\tr f\) because the map
\[
\Rep(G) \cong \KK^G_0(\Unit,\Unit) \xrightarrow{\natural}
\Tri_0(\Unit,\Unit) \cong S^{-1}\Rep(G)
\]
is injective.  Thus it suffices to compute Lefschetz indices in~\(\Tri\).
This may be easier because~\(\Tri\) has more isomorphisms and
thus fewer isomorphism classes of objects.  Furthermore, the
endomorphism ring of the unit
\(\Tri_*(\Unit,\Unit)=S^{-1}\Rep(G)\) has a rather simple
structure:

\begin{lemma}
  \label{lem:localised_RepG}
  The ring \(S^{-1}\Rep(G)\) is a product of finitely many
  fields.
\end{lemma}

\begin{proof}
  Let \(G/\Ad G\) be the space of conjugacy classes in~\(G\)
  and let \(\Cont(G/\Ad G)\) be the algebra of continuous functions
  on \(G/\Ad G\).  Taking characters provides a ring
  homomorphism \(\chi\colon \Rep(G)\to \Cont(G/\Ad G)\), which is
  well-known to be injective.  Hence \(\Rep(G)\) is
  torsion-free as an Abelian group and has no nilpotent
  elements.  Since~\(G\) is a compact Lie group, \(\Rep(G)\) is
  a finitely generated commutative ring by
  \cite{Segal:Representation_ring}*{Corollary 3.3}.  Thus
  \(\Rep(G)\) is Noetherian and reduced.  This implies that its
  total ring of fractions is a finite product of fields (see
  \cite{Matsumura:Commutative_ring}*{Exercise 6.5}).
\end{proof}

The fields in this product decomposition correspond bijectively
to minimal prime ideals in~\(\Rep(G)\).  By
\cite{Segal:Representation_ring}*{Proposition 3.7.iii}, these
correspond bijectively to cyclic subgroups of~\(G/G^0\),
where~\(G^0\) denotes the connected component of the identity
element.  In particular, \(S^{-1}\Rep(G)\) is a field if and
only if~\(G\) is connected.

\begin{example}
  \label{exa:localise_Rep_G_semisimple}
  Let~\(G\) be a connected compact Lie group.  Let~\(T\) be a
  maximal torus in~\(G\) and let~\(W\) be the Weyl group,
  \(W\defeq N_G(T)/T\).  Highest weight theory provides an
  isomorphism \(\Rep(G) \cong\Rep(T)^W\).  Here \(\Rep(T)\) is
  a ring of integral Laurent polynomials in \(r\)~variables,
  where~\(r\) is the rank of~\(T\).  Since elements
  of~\(\N_{\ge1}\) are not zero divisors in~\(\Rep(G)\), the
  total ring of fractions of~\(\Rep(G)\) is equal to the total
  ring of fractions of \(\Rep(G)\otimes\Q\).  The latter is the
  \(\Q\)\nb-algebra of \(W\)\nb-invariant elements in
  \(\Q[x_1,\dotsc,x_r,(x_1\dotsm x_r)^{-1}]\).  This is the
  algebra of polynomial functions on the algebraic
  \(\Q\)\nb-variety~\((\Q^\times)^r\), and the
  \(W\)\nb-invariants give the algebra of polynomials on the
  quotient variety~\((\Q^\times)^r/W\).  This variety is
  connected, so that the total ring of fractions
  \(S^{-1}\Rep(G)\) in this case is the field of rational
  functions on the algebraic
  \(\Q\)\nb-variety~\((\Q^\times)^r/W\).
\end{example}

Now we can define an equivariant analogue of the trace of the
map on \(\K_*(A)\otimes\Q\) induced by \(f\in\KK_0(A,A)\):

\begin{definition}
  \label{def:trace_KGf}
  Let \(S^{-1}\Rep(G) = \prod_{i=1}^n F_i\) with
  fields~\(F_i\).  A module over \(S^{-1}\Rep(G)\) is a product
  \(\prod_{i=1}^n V_i\), where each~\(V_i\) is an
  \(F_i\)\nb-vector space.  In particular, if~\(A\) is a
  \(G\)\nb-\(\Cst\)-algebra, then \(\Tri_*(\C,A) =
  S^{-1}\K^G_*(A) = \prod_{i=1}^n \K^G_{*,i}(A)\) for certain
  \(\Z/2\)\nb-graded \(F_i\)\nb-vector
  spaces~\(\K^G_{*,i}(A)\).  An endomorphism
  \(f\in\Tri_0(A,A)\) induces grading-preserving endomorphisms
  \(\K^G_{*,i}(f)\colon \K^G_{*,i}(A)\to \K^G_{*,i}(A)\).

  If the vector spaces \(\K^G_{*,i}(A)\) are all finite-dimensional,
  then the (super)trace of \(\K^G_{*,i}(f)\) is defined to be \(\tr
  \K^G_{0,i}(f) - \tr \K^G_{1,i}(f)\in F_i\), and
  \[
  \tr S^{-1}\K^G_*(f) \defeq (\tr \K^G_{*,i}(f))_{i=1}^n
  \in \prod_{i=1}^n F_i = S^{-1}\Rep(G).
  \]
\end{definition}

We will see below that dualisability for objects in appropriate
bootstrap classes already implies that \(\K^G_*(A)\) is a
finitely generated \(\Rep(G)\)-module, and then each
\(\K^G_{*,i}(A)\) must be a finite-dimensional \(F_i\)\nb-vector
space.

\begin{theorem}
  \label{the:trace_on_genUnit}
  Let~\(A\) belong to the thick subcategory of~\(\KK^G\)
  generated by~\(\C\) and let \(f\in\KK^G_0(A,A)\).
  Then~\(A\) is dualisable in \(\KK^G\), so that \(\tr f\) is
  defined, and
  \[
  \natural(\tr f) = \tr S^{-1}\K^G_*(f) \in S^{-1}\Rep(G).
  \]
\end{theorem}

Thick subcategories are defined in
\cite{Neeman:Triangulated}*{Definition 2.1.6}.  The thick
subcategory generated by~\(\C\) is, of course, the smallest
thick subcategory that contains the object~\(\C\).  We denote
the thick subcategory generated by a set~\(A\) of objects or a
single object by~\(\thick{A}\).

As we remarked above, \(\natural(\tr f)\) uniquely determines
\(\tr f\in\Rep(G)\) because the canonical embedding
\(\natural\colon \Rep(G)\to S^{-1}\Rep(G)\) is injective.

We will prove Theorem~\ref{the:trace_on_genUnit} in
Section~\ref{sec:trace_computation}.

How restrictive is the assumption that~\(X\) should belong to
the thick subcategory of~\(\KK^G\) generated by~\(\C\)?  The
answer depends on the group~\(G\).

We consider the two extreme cases: \emph{Hodgkin Lie groups}
and finite groups.

A Hodgkin Lie group is, by definition, a connected Lie group
with simply connected fundamental group; they are the groups to
which the Universal Coefficient Theorem and the K\"unneth
Theorem in~\cite{Rosenberg-Schochet:Kunneth} apply.

\begin{theorem}
  \label{the:Hodgkin_Lie_bootstrap}
  Let~\(G\) be a compact Lie group with torsion-free
  fundamental group.  Then a \(G\)\nb-\(\Cst\)-algebra~\(A\)
  belongs to the thick subcategory generated by~\(\C\) if and
  only if
  \begin{itemize}
  \item \(A\), without the \(G\)\nb-action, belongs to the
    bootstrap category in \(\KK\), and
  \item \(A\) is dualisable.
  \end{itemize}
\end{theorem}

We postpone the proof of this theorem until after the proof of
Proposition~\ref{pro:dualisable_compact}, which generalises part of
this theorem to arbitrary compact Lie groups.

The first condition in Theorem~\ref{the:Hodgkin_Lie_bootstrap}
is automatic for commutative \(\Cst\)\nb-algebras because the
non-equivariant bootstrap category is the class of all
separable \(\Cst\)\nb-algebras that are \(\KK\)-equivalent to a
commutative separable \(\Cst\)\nb-algebra.  Hence
Theorem~\ref{the:Hodgkin_Lie_bootstrap} verifies the
assumptions needed for Theorem~\ref{the:trace_on_genUnit} if
\(A=\Cont_0(X)\) and~\(\Cont_0(X)\) is dualisable in \(\KK^G\);
the latter is necessary for the Lefschetz index to be defined, anyway.

In particular, let~\(X\) be a compact smooth \(G\)\nb-manifold
with boundary, for a Hodgkin Lie group~\(G\).  Then~\(X\) is
dualisable in \(\GKK^G\) by
Theorem~\ref{the:compact_smooth_dualisable}, and hence
\(\Cont(X)\) is dualisable in \(\KK^G\) because the functor
\(\GKK^G\to\KK^G\) is symmetric monoidal.  Furthermore,
\(\GKK^G_*(X,X) \cong \KK^G_*(\Cont(X),\Cont(X))\) in this
case, so that any endomorphism
\(f\in\KK^G_0(\Cont(X),\Cont(X))\) comes from some
self-correspondence in \(\GKK^G_0(X,X)\).  We get the following
generalisation of the Lefschetz fixed-point formula:

\begin{corollary}
  \label{cor:Lefschetz_Hodgkin_Lie_group}
  Let~\(G\) be a Hodgkin Lie group, \(X\) a smooth
  compact \(G\)\nb-manifold, possibly with boundary, and
  \(f\in\GKK^G_0(X,X)\).  Then \(\tr(f)\in \Rep(G)\subseteq
  S^{-1}\Rep(G)\) is equal to the supertrace
  of~\(S^{-1}\K^*_G(f)\), acting on the
  \(S^{-1}\Rep(G)\)-vector space \(S^{-1}\K^*_G(X)\).
\end{corollary}

Notice that \(S^{-1}\Rep(G)\) for a Hodgkin Lie group is a
field, not just a product of fields.

In particular, Corollary~\ref{cor:Lefschetz_Hodgkin_Lie_group}
for the trivial group gives the Lefschetz index formula
in~\cite{Emerson:Lefschetz}.

Whereas Theorem~\ref{the:trace_on_genUnit} yields quite
satisfactory results for Hodgkin Lie groups, its scope for a
finite group~\(G\) is quite limited:

\begin{example}
  \label{exa:pt_does_not_generate_for_Ztwo}
  For \(G=\Z/2\) there is a locally compact
  \(G\)\nb-space~\(X\) with \(\K^*_G(X)=0\) but
  \(\K^*(X)\neq0\).  Equivalently, \(\KK^G_*(\C,\Cont_0(X))=0\)
  and \(\KK^G_*(\Cont(G),\Cont_0(X))\neq 0\).  This shows
  that~\(\Cont(G)\) does not belong to~\(\thick{\C}\).

  Worse, the Lefschetz index formula in
  Theorem~\ref{the:trace_on_genUnit} is false for endomorphisms
  of~\(\Cont(G)\).  We have \(\GKK^G_*(G,G) \cong \Z[G]\), spanned by
  the classes of the translation maps \(G\to G\), \(x\mapsto
  x\cdot g\) for \(g\in G\), and these are homogeneous
  correspondences as in Section~\ref{sec:trace_homogeneous}.

  Translation by \(g=1\) is the identity map, and its Lefschetz index is
  the class of the regular representation of~\(G\) in
  \(\Rep(G)\).  For \(g\neq1\), the Lefschetz index is zero because the
  fixed point subset is empty.  However,
  \(\K^*_G(G)=\K^*(\pt)=\Z[0]\) and all translation maps induce
  the identity map on \(\K^*_G(G)\).  Thus the induced map
  on~\(\K^*_G(G)\) is not enough information to compute the
  Lefschetz index of an endomorphism of~\(G\) in \(\GKK^G\).
\end{example}

\subsection{The equivariant bootstrap category}
\label{sec:equiv_bootstrap}

A reasonable Lefschetz index formula should apply at least to
\(\KK^G\)-endomorphisms of~\(\Cont(X)\) for all smooth compact
\(G\)\nb-manifolds and thus, in particular, for finite
\(G\)\nb-sets~\(X\).  Example~\ref{exa:pt_does_not_generate_for_Ztwo}
shows that Theorem~\ref{the:trace_on_genUnit} fails on such a larger
category.  This leads us to improve the Lefschetz index
formula.  First we discuss the class of
\(G\)\nb-\(\Cst\)-algebras where we expect it to hold.

We are going to describe an equivariant analogue of the bootstrap
class in~\(\KK^G\).  Our class is larger than the class of
\(\Cst\)\nb-algebras that are \(\KK^G\)-equivalent to a commutative
\(\Cst\)\nb-algebra.  The latter subcategory is too small because it
is not thick.  The thick (or localising) subcategory of~\(\KK^G\)
generated by commutative \(\Cst\)\nb-algebras is a better choice, but
such a definition is not very intrinsic.  We will choose an even
larger subcategory of~\(\KK^G\) because it is not more difficult to
treat and has a nicer characterisation.

The category~\(\KK^G\) only has countable coproducts because we need
\(\Cst\)\nb-algebras to be separable.  Hence the standard notions of
compact objects and localising subcategories have to be modified so
that they only involve countable coproducts.  As in
\cite{dellAmbrogio:Tensor_triangular}*{Definition 2.1}, we speak of
\emph{compact\(_{\aleph_1}\)} objects,\emph{
  localising\(_{\aleph_1}\)} subcategories, and
\emph{compactly\(_{\aleph_1}\)} generated subcategories.

\begin{definition}
  \label{def:elementary}
  Call a \(G\)\nb-\(\Cst\)-algebra~\(A\) \emph{elementary} if
  it is of the form \(\Ind_H^G \Mat_n\C = \Cont(G,\Mat_n\C)^H\)
  for some closed subgroup \(H\subseteq G\) and some action
  of~\(H\) on~\(\Mat_n\C\) by automorphisms; the
  superscript~\(H\) means the fixed points for the diagonal
  action of~\(H\).
\end{definition}

\begin{definition}
  \label{def:bootstrap_general}
  Let \(\Boot^G\subseteq\KK^G\) be the localising\(_{\aleph_1}\)
  subcategory generated by all elementary
  \(G\)\nb-\(\Cst\)-algebras.  We call~\(\Boot^G\) the
  \emph{\(G\)\nb-equivariant bootstrap category}.
\end{definition}

An action of~\(H\) on~\(\Mat_n\C\) comes from a projective
representation of~\(H\) on~\(\C^n\).  Such a projective
representation is a representation of an extension of~\(H\) by
the circle group.  The extension is classified by a cohomology
class in \(\textup{H}^2(H,\Un(1))\).  Two actions
on~\(\Mat_n\C\) are \(H\)\nb-equivariantly Morita equivalent if
and only if they belong to the same class in
\(\textup{H}^2(H,\Un(1))\).  The \(G\)\nb-\(\Cst\)\nb-algebras
\(\Ind_H^G \Mat_n\C\) for actions of~\(H\) on~\(\Mat_n\C\) with
different cohomology classes need not be \(\KK^G\)-equivalent.

\begin{theorem}
  \label{the:type_I_in_bootstrap}
  A \(G\)\nb-\(\Cst\)-algebra belongs to the localising\(_{\aleph_1}\)
  subcategory generated by the elementary \(G\)\nb-\(\Cst\)-algebras
  if and only if it is \(\KK^G\)-equivalent to a \(G\)\nb-action on a
  type~I \(\Cst\)\nb-algebra.
\end{theorem}

\begin{proof}
  It is already shown in
  \cite{Rosenberg-Schochet:Kunneth}*{Theorem~2.8} that all
  \(G\)\nb-actions on type~I \(\Cst\)\nb-algebras belong to the
  localising\(_{\aleph_1}\) subcategory generated by the elementary
  \(G\)\nb-\(\Cst\)-algebras.  By definition,
  localising\(_{\aleph_1}\) subcategories are closed under
  \(\KK^G\)-equivalence.  Elementary \(G\)\nb-\(\Cst\)\nb-algebras are
  type~I \(\Cst\)\nb-algebras, even continuous trace
  \(\Cst\)\nb-algebras.  To finish the proof we must show that the
  \(G\)\nb-\(\Cst\)-algebras that are \(\KK^G\)-equivalent to type~I
  \(G\)\nb-\(\Cst\)-algebras form a localising\(_{\aleph_1}\)
  subcategory of \(\KK^G\).

  Let \(\Type\subseteq\KK^G\) be the full subcategory of type~I, separable
  \(G\)\nb-\(\Cst\)-algebras.  If \(A\in\Type\), then
  \(\Cont_0(\R,A)\in\Type\), so that~\(\Type\) is closed under
  suspension and desuspension.  Let \(A,B\in\Type\) and
  \(f\in\KK^G_0(A,B)\).  We have \(\KK^G_0(A,B) \cong
  \KK^G_1(A,\Cont_0(\R,B))\), and cycles for the latter group
  correspond to (equivariantly) semisplit extensions of
  \(G\)\nb-\(\Cst\)-algebras
  \[
  \Cont_0(\R,B)\otimes\Comp \mono D \epi A
  \]
  with \(\Comp \defeq \Comp(L^2 (G\times\N))\).  Since \(B\)
  and~\(A\) are type~I, so are \(\Cont_0(\R,B)\otimes\Comp\)
  and~\(D\) because the property of being type~I is inherited
  by extensions.  The semisplit
  extension above provides an exact triangle isomorphic to
  \[
  B[-1] \to D\to A\xrightarrow{f} B.
  \]
  Thus there is an
  exact triangle containing~\(f\) with all three entries
  in~\(\Type\).  Furthermore, countable direct sums of type~I
  \(\Cst\)\nb-algebras are again type~I.  This implies that the
  \(G\)\nb-\(\Cst\)-algebras \(\KK^G\)-equivalent to one
  in~\(\Type\) form a localising\(_{\aleph_1}\) subcategory of~\(\KK^G\).
\end{proof}

\begin{remark}
  \label{rem:bootstrap_nice_model}
  In the non-equivariant case, any \(\Cst\)\nb-algebra in the
  bootstrap class is \(\KK\)-equivalent to a commutative one.  This
  criterion fails already for \(G=\Un(1)\), as shown by a
  counterexample in~\cite{Emerson:Localization_circle}.  Since the
  bootstrap class is the smallest localising subcategory
  containing~\(\C\), it follows that the commutative
  \(\Cst\)\nb-algebras do not form a localising subcatgory.  Thus
  \(\GKK^G\) is not triangulated: it lacks cones for some maps.

  In this case, the equivariant bootstrap class is already
  generated by~\(\C\) and contains all \(\Un(1)\)-actions on
  \(\Cst\)\nb-algebras in the non-equivariant bootstrap
  category.  It is shown in~\cite{Emerson:Localization_circle}
  that the \(\Un(1)\)-equivariant \(\K\)\nb-theory of a
  suitable Cuntz--Krieger algebra with its natural gauge action
  cannot arise from any \(\Un(1)\)-action on a locally compact
  space.
\end{remark}

\begin{corollary}
  \label{cor:res_ind_bootstrap}
  The restriction and induction functors \(\KK^G\to\KK^H\) and
  \(\KK^H\to\KK^G\) for a closed subgroup~\(H\) in a compact Lie
  group~\(G\) restrict to functors between the bootstrap classes in
  \(\KK^G\) and \(\KK^H\).
\end{corollary}

\begin{proof}
  Restriction does not change the underlying
  \(\Cst\)\nb-algebra and thus preserves the property of being
  type~I.  Induction maps elementary \(H\)\nb-\(\Cst\)-algebras
  to elementary \(G\)\nb-\(\Cst\)-algebras, is triangulated, and
  commutes with direct sums.  Hence it maps~\(\Boot^H\)
  to~\(\Boot^G\).
\end{proof}

\begin{proposition}
  \label{pro:dualisable_compact}
  An object of~\(\Boot^G\) is compact\(_{\aleph_1}\) if and only
  if it is dualisable, if and only if it belongs to the thick
  subcategory of~\(\Boot^G\) \textup(or of \(\KK^G\)\textup)
  generated by the elementary \(G\)\nb-\(\Cst\)\nb-algebras.
\end{proposition}

\begin{proof}
  The tensor unit~\(\C\) is compact\(_{\aleph_1}\) because
  \(\KK^G_*(\C,A) \cong \K^G_*(A) \cong \K_*(G\ltimes A)\) is
  countable for all \(G\)\nb-\(\Cst\)\nb-algebras~\(A\), and the
  functors \(A\mapsto G\ltimes A\) and~\(\K_*\) are well-known to
  commute with coproducts.  Furthermore, the tensor product in
  \(\KK^G\) commutes with coproducts in both variables.

  Using this, we show that dualisable objects of~\(\Boot^G\) are
  compact\(_{\aleph_1}\).  If~\(A\) is dualisable with
  dual~\(A^*\), then \(\KK^G(A,B) \cong \KK^G(\C,A^*\otimes
  B)\), and since~\(\C\) is compact\(_{\aleph_1}\)
  and~\(\otimes\) commutes with countable direct sums, it
  follows that~\(A\) is compact\(_{\aleph_1}\).

  It follows from \cite{Echterhoff-Emerson-Kim:Duality}*{Corollary
    2.2} that elementary \(G\)\nb-\(\Cst\)\nb-algebras are dualisable
  and hence compact\(_{\aleph_1}\).  A compact group has only at most
  countably many compact subgroups by
  Lemma~\ref{lem:countably_many_subgroups} below; and any of them has
  at most finitely many projective representations.  Hence the set of
  elementary \(G\)\nb-\(\Cst\)\nb-algebras is at most countable.
  Therefore, \(\Boot^G\) is compactly\(_{\aleph_1}\) generated in the
  sense of \cite{dellAmbrogio:Tensor_triangular}*{Definition 2.1}.  By
  \cite{dellAmbrogio:Tensor_triangular}*{Corollary~2.4} an object
  of~\(\Boot^G\) is compact\(_{\aleph_1}\) if and only if it belongs
  to the thick subcategory generated by the elementary
  \(G\)\nb-\(\Cst\)-algebras.

  The Brown Representability Theorem
  \cite{dellAmbrogio:Tensor_triangular}*{Corollary~2.2} shows
  that for every compact\(_{\aleph_1}\) object~\(A\)
  of~\(\Boot^G\) there is a functor \(\Hom(A,\blank)\) from~$\Boot^G$
  to~$\Boot^G$ such that
  \[
  \KK^G(A\otimes B,D) \cong \KK^G(B,\Hom(A,D))
  \]
  for all \(B,D\) in~\(\Boot^G\).  Using exactness properties of
  the internal Hom functor in the first variable, we then show
  that the class of dualisable objects in~\(\Boot^G\) is thick
  (see~\cite{dellAmbrogio:Tensor_triangular}*{Section 2.3}).  Thus
  all objects of the thick subcategory generated by the
  elementary \(G\)\nb-\(\Cst\)\nb-algebras are dualisable.
\end{proof}

The following lemma is well-known, see for example 
\cite{Montgomery-Zippen:Subgroups_of_a_compact_group}. 

\begin{lemma}
  \label{lem:countably_many_subgroups}
  A compact Lie group has at most countably many conjugacy classes of
  closed subgroups.
\end{lemma}

\begin{proof}
  Let~\(H\) be a closed subgroup of a compact Lie group~\(G\).  By the
  Mostow Embedding Theorem, \(G/H\) embeds into a linear
  representation of~\(G\), that is, \(H\) is a stabiliser of a point
  in some linear representation of~\(G\).  Up to isomorphism, there
  are only countably many linear representations of~\(G\).  Each
  linear representation has finite orbit type, that is, it admits only
  finitely many different conjugacy classes of stabilisers.  Hence
  there are altogether at most countably many conjugacy classes of
  closed subgroups in~\(G\).
\end{proof}

\begin{proof}[Proof of Theorem~\ref{the:Hodgkin_Lie_bootstrap}]
  Let~\(G\) be a Hodgkin Lie group.  The main result
  of~\cite{Meyer-Nest:BC_Coactions} says that~\(A\) belongs to the
  localising subcategory of \(\KK^G\) generated by~\(\C\) if and only
  if \(A\rtimes G\) belongs to the non-equivariant bootstrap category (this is
  special for Hodgkin Lie groups).  Since this covers all elementary
  \(G\)\nb-\(\Cst\)-algebras, we conclude that the localising
  subcategory generated by~\(\C\) contains~\(\Boot^G\) and is, therefore,
  equal to~\(\Boot^G\).

  The same argument as in the proof of
  Proposition~\ref{pro:dualisable_compact} shows that the
  following are equivalent for an object~\(A\) of~\(\Boot^G\):
  \begin{itemize}
  \item \(A\) is dualisable;
  \item \(A\) is compact\(_{\aleph_1}\);
  \item \(A\) belongs to the thick subcategory generated by~\(\C\).
  \end{itemize}
  This finishes the proof of Theorem~\ref{the:Hodgkin_Lie_bootstrap}.
\end{proof}

So far we always used the bootstrap class, which is the domain where a
Universal Coefficient Theorem holds.  The next proposition is a side
remark showing that we may also use the domain where a K\"unneth
formula holds.

\begin{definition}
  \label{def:Kunneth_UCT}
  An object \(A\in\KK^G\) satisfies the K\"unneth formula if
  \(\K^G_*(A\otimes B)=0\) for all~\(B\) that satisfy
  \(\K^G_*(C\otimes B)=0\) for all elementary
  \(G\)\nb-\(\Cst\)\nb-algebras~\(C\).
\end{definition}

By results of \cites{Meyer:Homology_in_KK_II,
  Meyer-Nest:Homology_in_KK}, the assumption in
Definition~\ref{def:Kunneth_UCT} is necessary and sufficient for a
certain natural spectral sequence that computes \(\K^G_*(A\otimes B)\)
from \(\KK^G_*(C,A)\) and \(\KK^G_*(C,B)\) for elementary~\(C\) to
converge for all~\(B\); we have no need to describe this spectral
sequence.

\begin{proposition}
  \label{pro:Kunneth_versus_UCT}
  Let \(A\in\KK^G\) be dualisable with dual~\(A^*\).  If \(A\)
  or~\(A^*\) satisfies a K\"unneth formula, then both
  \(A\) and~\(A^*\) belong to~\(\Boot^G\), and vice versa.
\end{proposition}

\begin{proof}
  Since~\(\Boot^G\) is generated by the elementary
  \(G\)\nb-\(\Cst\)-algebras, \(\KK^G_*(C,B)=0\) for all elementary
  \(G\)\nb-\(\Cst\)\nb-algebras~\(C\) if and only if
  \(\KK^G_*(C,B)=0\) for all \(C\in\Boot^G\).  Any elementary
  \(G\)\nb-\(\Cst\)\nb-algebra~\(C\) is dualisable with a dual
  in~\(\Boot^G\).  Hence \(\K^G_*(C\otimes B)\cong \KK^G_*(C^*,B)=0\)
  for elementary~\(C\) if \(\KK^G_*(C',B)=0\) for all elementary
  \(G\)\nb-\(\Cst\)\nb-algebras~\(C'\); conversely \(\KK^G_*(C,B)
  \cong \K^G_*(C^*\otimes B)=0\) for elementary~\(C\) if
  \(\K^G_*(C'\otimes B)=0\) for all elementary
  \(G\)\nb-\(\Cst\)\nb-algebras~\(C'\).  Let us denote the class of
  \(G\)\nb-\(\Cst\)-algebras with these equivalent
  properties~\(\Boot^{G,\perp}\).

  It follows from \cite{Meyer:Homology_in_KK_II}*{Theorem 3.16}
  that \((\Boot^G,\Boot^{G,\perp})\) is a complementary pair of
  localising subcategories.  In particular, if
  \(\KK^G_*(A,B)=0\) for all \(B\in\Boot^{G,\perp}\), then
  \(A\in\Boot^G\).

  Now assume, say, that~\(A\) satisfies a K\"unneth formula.
  Then \(\KK^G_*(A^*,B)\cong \K^G_*(A\otimes B)=0\) for all
  \(B\in\Boot^{G,\perp}\).  Thus \(A^*\in\Boot^G\).  Then
  \(\K^G_*(A^*\otimes B)=0\) for all \(B\in\Boot^{G,\perp}\)
  because the class of~\(C\) with \(\K^G_*(C\otimes B)=0\) is
  localising and contains all elementary~\(C\) if
  \(B\in\Boot^{G,\perp}\).  As above, this implies
  \((A^*)^*=A\in\Boot^G\).
\end{proof}

The proof of Theorem~\ref{the:Hodgkin_Lie_bootstrap} above used that,
for a Hodgkin Lie group, \(\Boot^G\) is already generated by~\(\C\).  For
more general groups, we also expect that fewer generators suffice to
generate~\(\Boot^G\).  But we only need and only prove a result about
topologically cyclic groups here.

A locally compact group~\(G\) is called \emph{topologically cyclic} if
there is an element \(g\in G\) that generates a dense subgroup
of~\(G\).  A topologically cyclic group is necessarily Abelian.  We
are interested in topologically cyclic, compact Lie groups here.  A
compact Lie group is topologically cyclic if and only if it is
isomorphic to \(\T^r\times F\) for some \(r\ge0\) and some finite
cyclic group~\(F\) (possibly the trivial group), where \(\T=\R/\Z
\cong \Un(1)\).  Here we use that any extension \(\T^r \mono E\epi F\)
for a finite cyclic group~\(F\) splits.  This also implies that any
projective representation of a finite cyclic groups is a
representation.

\begin{theorem}
  \label{the:bootstrap_top_cyclic}
  Let~\(G\) be a topologically cyclic, compact Lie group.  Then the
  bootstrap class \(\Boot^G\subseteq\KK^G\) is already generated by the
  finitely many \(G\)-\(\Cst\)\nb-algebras \(\Cont(G/H)\) for all open
  subgroups \(H\subseteq G\).

  Furthermore, an object of~\(\Boot^G\) is compact\(_{\aleph_1}\) if and
  only if it is dualisable if and only if it belongs to the thick
  subcategory generated by \(\Cont(G/H)\) for open subgroups \(H\subseteq
  G\).
\end{theorem}

\begin{proof}
  The second statement about compact\(_{\aleph_1}\) objects
  in~\(\Boot^G\) follows from the first one and
  \cite{dellAmbrogio:Tensor_triangular}*{Corollary~2.4},
  compare the proof of
  Proposition~\ref{pro:dualisable_compact}.  Thus it suffices
  to prove that the objects \(\Cont(G/H)\) for open subgroups
  already generate~\(\Boot^G\).  For this, we use an isomorphism
  \(G\cong \T^r\times F\) for some \(r\ge 0\) and some finite
  cyclic subgroup~\(F\).

  Let us first consider the special case \(r=0\), that is, \(G\) is a
  finite cyclic group.  In this case, any subgroup of~\(G\) is open
  and again cyclic.  We observed above that cyclic groups have
  no non-trivial
  projective representations.  Thus any elementary
  \(G\)\nb-\(\Cst\)-algebra is Morita equivalent to \(\Cont(G/H)\) for
  some open subgroup~\(H\) in~\(G\).  Hence the assertion of the
  theorem is just the definition of~\(\Boot^G\) in this case.

  If~\(F\) is trivial, then the assertion follows from
  Theorem~\ref{the:Hodgkin_Lie_bootstrap}.  Now we consider the
  general
  case where both \(F\) and~\(\T^r\) are non-trivial.

  The Pontryagin dual~\(\hat{G}\) of~\(G\) is isomorphic to the
  discrete group \(\Z^r\times F\).  If~\(A\) is a
  \(G\)\nb-\(\Cst\)-algebra, then \(G\ltimes A\) carries a canonical
  action of~\(\hat{G}\) called the dual action.  Similarly,
  \(\hat{G}\ltimes A\) for a \(\hat{G}\)\nb-\(\Cst\)-algebra~\(A\)
  carries a canonical dual action of~\(G\).  This provides functors
  \(\KK^G\to \KK^{\hat{G}}\) and \(\KK^{\hat{G}}\to \KK^G\).
  Baaj--Skandalis duality says that they are inverse to each
  other up to natural equivalence (see
  \cite{Baaj-Skandalis:Hopf_KK}*{Section~6}).  Since both functors
  are triangulated, this is an equivalence of triangulated
  categories.

  If \(A\)~is type~I, then so is \(G\ltimes A\).  Hence all objects in
  \(\Boot^{\hat{G}}\subseteq \KK^{\hat{G}}\) are \(\KK^{\hat{G}}\)-equivalent to
  a \(\hat{G}\)\nb-action on a type~I \(\Cst\)-algebra by
  Theorem~\ref{the:type_I_in_bootstrap}.

  The group~\(\hat{G}\) is Abelian and hence satisfies a very
  strong form of the Baum--Connes conjecture: it has a dual
  Dirac morphism and \(\gamma=1\) in the sense of
  \cite{Meyer-Nest:BC}*{Definition 8.1}.  From this it follows
  that any \(\hat{G}\)\nb-\(\Cst\)-algebra~\(A\) belongs to the
  localising subcategory of \(\KK^{\hat{G}}\) that is generated
  by \(\Ind_{\hat{H}}^{\hat{G}} A\) for finite subgroups
  \(\hat{H}\subseteq\hat{G}\) (this is shown as in the proof of
  \cite{Meyer-Nest:BC}*{Theorem 9.3}).

  The finite subgroups in \(\Z^r\times\hat{F}\) are exactly the
  subgroups of~\(\hat{F}\), of course.  Since we have induction
  in stages, we may assume \(\hat{H}=\hat{F}\).  Thus the
  subcategory of type~I \(\hat{G}\)\nb-\(\Cst\)-algebras is
  already generated by \(\Ind_{\hat{F}}^{\hat{G}} A\) for
  type~I \(\hat{F}\)\nb-\(\Cst\)-algebras~\(A\).
  Since~\(\hat{F}\) is a finite cyclic group, the discussion
  above shows that the category of type~I
  \(\hat{F}\)\nb-\(\Cst\)-algebras~\(A\) is already generated
  by \(\Cont_0(\hat{F}/\hat{H})\) for subgroups
  \(\hat{H}\subseteq\hat{F}\).  Thus~\(\Boot^{\hat{G}}\) is
  generated by
  the \(\hat{G}\)\nb-\(\Cst\)-algebras
  \(\Ind_{\hat{F}}^{\hat{G}}\Cont_0(\hat{F}/\hat{H}) \cong
  \Cont_0(\hat{G}/\hat{H})\).  The finite subgroups
  \(\hat{H}\subseteq\hat{G}\) are exactly the orthogonal
  complements of (finite-index) open subgroups \(H\subseteq G\).

  Now \(G\ltimes \Cont_0(G/H)\) is Morita equivalent to
  \(\Cst(H) \cong \Cont_0(\hat{G}/\hat{H})\) for any open
  subgroup \(H\subseteq
  G\), where \(\hat{H}\subseteq\hat{G}\) denotes the orthogonal
  complement of~\(H\) in~\(\hat{G}\).  The dual action
  on~\(\Cont_0(\hat{G}/\hat{H})\) comes from the translation
  action of~\(\hat{G}\).  Thus the \(G\)- and
  \(\hat{G}\)\nb-\(\Cst\)-algebras \(\Cont_0(G/H)\) and
  \(\Cont_0(\hat{G}/\hat{H})\) correspond to each other via
  Baaj--Skandalis duality.  We conclude that the
  \(G\)\nb-\(\Cst\)-algebras \(\Cont_0(G/H)\) for open
  subgroups \(H\subseteq G\) generate~\(\Boot^G\).
\end{proof}

Let~\(G\) be topologically cyclic, say, \(G\cong \T^r\times
\Z/k\) for some \(r\ge0\), \(k\ge1\).  Then open subgroups
of~\(G\) correspond to subgroups of~\(\Z/k\) and thus to
divisors~\(d\) of~\(k\).  The representation ring of~\(G\) is
\begin{equation}
  \label{eq:RepG_top_cyclic}
  \Rep(G)
  \cong \Rep(\T^r) \otimes \Rep(\Z/k)
  \cong \Z[x_1,\dotsc,x_r,(x_1\dotsm x_r)^{-1}]
  \otimes \Z[t]/(t^k-1).
\end{equation}
Let
\[
t^k-1 = \prod_{d\mid k} \Phi_d(t)
\]
be the decomposition into cyclotomic polynomials.  Each
factor~\(\Phi_d\) generates a minimal prime ideal
of~\(\Rep(G)\), and these are all minimal prime ideals
of~\(\Rep(G)\).  The localisation at this prime ideal gives the
field \(\Q(\theta_d)(x_1,\dotsc,x_r)\) of rational functions
in~\(r\) variables over the cyclotomic field~\(\Q(\theta_d)\),
and the product of these localisations is the total ring of fractions
of~\(\Rep(G)\),
\[
S^{-1}\Rep(G) = \prod_{d\mid k} \Q(\theta_d)(x_1,\dotsc,x_r).
\]
(Compare Lemma~\ref{lem:localised_RepG}.)

\begin{lemma}
  \label{lem:GH_zero_in_localisation}
  Let \(H\subsetneq G\) be a proper open subgroup.  The
  canonical map
  \[
  \Rep(G) \to \KK^G_0(\Cont(G/H),\Cont(G/H))
  \]
  from the exterior product in \(\KK^G\) factors through the
  restriction map \(\Rep(G)\to\Rep(H)\).  The image of
  \(\Cont(G/H)\) in the localisation of~\(\KK^G\) at the prime
  ideal~\((\Phi_k)\) vanishes.
\end{lemma}

\begin{proof}
  The exterior product of the identity map on \(\Cont(G/H)\)
  and \(\xi\in\Rep(G) \cong \KK^G_0(\C,\C)\) is given by the
  geometric correspondence \(G/H=G/H=G/H\) with the class
  \(p^*(\xi) \in \K^0_G(G/H)\), where \(p\colon G/H\to\pt\) is
  the constant map.  Now identify \(\K^0_G(G/H) \cong
  \K^0_H(\pt) \cong \Rep(H)\) and~\(p^*\) with the restricton
  map \(\Rep(G)\to\Rep(H)\) to get the first statement.

  We have \(H\cong \T^r \times \Z/d\) embedded via
  \((x,j)\mapsto (x,jk/d)\) into \(G\cong \T^r\times \Z/k\).
  If \(H\neq G\), then \(d\neq k\).  The restriction map
  \(\Rep(G)\to\Rep(H)\) annihilates the polynomial
  \((t^k-1)/\Phi_k = \prod_{d\mid k, d\neq k} \Phi_d\).  This
  polynomial does not belong to the prime ideal~\((\Phi_k)\)
  and hence becomes invertible in the localisation
  of~\(\Rep(G)\) at~\((\Phi_k)\).  Since an invertible
  endomorphism can only be zero on the zero object,
  \(\Cont(G/H)\) becomes zero in the localisation of~\(\KK^G\)
  at~\((\Phi_k)\).
\end{proof}

\subsection{Localisation of the bootstrap class}
\label{sec:localise_bootstrap}

\begin{proposition}
  \label{pro:KKG_localisation_loc_cyc}
  Let \(G\cong \T^r\times\Z/k\) be topologically cyclic.
  Let~\(\Boot^G_d\) be the thick subcategory of dualisable objects
  in the bootstrap class \(\Boot^G\subseteq \KK^G\).  Any object
  in the localisation of~\(\Boot^G_d\) at the prime
  ideal~\((\Phi_k)\) in~\(\Rep(G)\) is isomorphic to a finite
  direct sum of suspensions of~\(\C\).
\end{proposition}

\begin{proof}
  By Theorem~\ref{the:bootstrap_top_cyclic} an object
  of~\(\Boot^G\) is dualisable if and only if it belongs to the
  thick subcategory generated by \(\Cont(G/H)\) for open
  subgroups \(H\subseteq G\).
  Lemma~\ref{lem:GH_zero_in_localisation} shows that all of
  them except \(\C=\Cont(G/G)\) become zero when we localise
  at~\((\Phi_k)\).  Hence the image of~\(\Boot^G_d\) in the
  localisation is contained in the thick subcategory generated
  by~\(\C\).  We must show that the objects isomorphic to a
  direct sum of suspensions of~\(\C\) already form a thick
  subcategory in the localisation of~\(\KK^G\) at~\((\Phi_k)\).

  The graded endomorphism ring of~\(\C\) in this localisation is
  \[
  \KK^G_*(\C,\C) \otimes_{\Rep(G)} \Rep(G)_{(\Phi_k)}
  \cong \Q(\theta_k)(x_1,\dotsc,x_r)[\beta,\beta^{-1}]
  \]
  with~\(\beta\) of degree two generating Bott periodicity.  It
  is crucial that \(\KK^G_*(\C,\C) \cong F[\beta,\beta^{-1}]\)
  for a field \(F \defeq \Q(\theta_k)(x_1,\dotsc,x_r)\).  The
  following argument only uses this fact.

  We map a finite direct sum \(A = \bigoplus_{i\in I}
  \C[\varepsilon_i]\) of suspensions of~\(\C\) to the
  \(\Z/2\)\nb-graded \(F\)\nb-vector spaces~\(V(A)\) with
  basis~\(I\) and generators of degree~\(\varepsilon_i\).  For two
  such direct sums, \(\KK^G_0(A,B)\) is isomorphic to the space
  of grading-preserving \(F\)\nb-linear maps \(V(A) \to V(B)\)
  because this clearly holds for a single summand.

  Now let \(f\in\KK^G_0(A,B)\) and consider the associated
  linear map \(V(f)\colon V(A)\to V(B)\).  Choose a basis for
  the kernel of~\(V(f)\) of homogeneous elements and extend it
  to a homogeneous basis for~\(V(A)\), and extend the resulting
  basis for the image of~\(V(f)\) to a homogeneous basis
  of~\(V(B)\).  This provides isomorphisms \(V(A)\cong
  V_0\oplus V_1\), \(V(B)\cong W_1\oplus W_2\) such that
  \(f|_{V_0}=0\), \(f(V_1) = W_1\) and \(f|_{V_1}\colon V_1\to
  W_1\) is an isomorphism.  The chosen bases describe how to
  lift the \(\Z/2\)\nb-graded vector spaces \(V_i\) and~\(W_i\)
  to direct sums of suspensions of~\(\C\).  Thus the map~\(f\)
  is equivalent to a direct sum of three maps \(f_0\oplus
  f_1\oplus f_2\) with \(f_0\colon A_0\to 0\) mapping to the
  zero object, \(f_1\) invertible, and \(f_2\colon 0\to B_2\)
  with domain the zero object.  The mapping cone of~\(f_0\) is
  the suspension of~\(A_0\), the cone of~\(f_2\) is~\(B_2\),
  and the cone of~\(f_1\) is zero.  Hence the
  cone is again a direct sum of suspensions of~\(\C\).
  Furthermore, any idempotent endomorphism has a range object.

  Thus the direct sums of suspensions of~\(\C\) already form an
  idempotent complete triangulated category.  As a consequence,
  any object in the thick subcategory generated by~\(\C\) is
  isomorphic to a direct sum of copies of~\(\C\).
\end{proof}

\begin{proposition}
  \label{pro:KKG_localisation_Hodgkin}
  Let~\(G\) be a Hodgkin Lie group.  Let~\(\Boot^G_d\) be the
  thick subcategory of dualisable objects in the bootstrap
  class \(\Boot^G\subseteq \KK^G\).  Any object in the
  localisation of~\(\Boot^G_d\) at~$S$ is isomorphic
  to a finite direct sum of suspensions of~\(\C\).
\end{proposition}

\begin{proof}
  Theorem~\ref{the:Hodgkin_Lie_bootstrap} shows that~\(\Boot^G_d\)
  is the thick subcategory of~\(\KK^G\) generated by~\(\C\).
  The localisation \(F\defeq S^{-1}\Rep(G)\) is a field
  because~\(G\) is connected, and the graded endomorphism ring
  of~\(\C\) in the localisation of~\(\KK^G\)
  at~$S$ is \(F[\beta,\beta^{-1}]\)
  with~\(\beta\) the generator of Bott periodicity.  Now the
  argument is finished as in the proof of
  Proposition~\ref{pro:KKG_localisation_loc_cyc}.
\end{proof}

\begin{remark}
  \label{rem:localisation_sums}
  The localisations above use the groups
  \(\KK^G(A,B) \otimes_{\Rep(G)} S^{-1}\Rep(G)\) for some
  multiplicatively closed subset $S\subseteq \End(\Unit)=R(G)$,
  following~\cite{Inassaridze-Kandelaki-Meyer:Finite_Torsion_KK}.
  A drawback of this localisation is that the canonical functor
  \(\KK^G \to S^{-1}\KK^G\) does not commute with (countable)
  coproducts.  This is why Propositions
  \ref{pro:KKG_localisation_loc_cyc}
  and~\ref{pro:KKG_localisation_Hodgkin} are formulated only
  for~\(\Boot^G_d\) and not for all of~\(\Boot^G\).

  Another way to localise~\(\Boot^G\) at~\(S\) is described in
  \cite{dellAmbrogio:Tensor_triangular}*{Theorem 2.33}.  Both
  localisations agree on~\(\Boot^G_d\) by
  \cite{dellAmbrogio:Tensor_triangular}*{Theorem 2.33.h}.  The
  construction in \cite{dellAmbrogio:Tensor_triangular} has the
  advantage that the canonical functor from \(\KK^G\) to this
  localisation commutes with small\(_{\aleph_1}\) (that is,
  countable) coproducts.
  Hence analogues of Propositions
  \ref{pro:KKG_localisation_loc_cyc}
  and~\ref{pro:KKG_localisation_Hodgkin} hold for the whole
  bootstrap category~\(\Boot^G\), with small\(_{\aleph_1}\)
  coproducts of suspensions of~\(\C\) instead of finite direct
  sums of suspensions of~\(\C\).
\end{remark}

\subsection{The Lefschetz index computation using localisation}
\label{sec:trace_computation}

Now we have all the tools available to formulate and prove a
Lefschetz index formula for general compact Lie groups.  We
first prove Theorem~\ref{the:trace_on_genUnit}, which deals
with endomorphisms of objects in the thick subcategory
generated by~\(\C\).  Then we formulate and prove the general
Lefschetz index formula.

\begin{proof}[Proof of Theorem~\ref{the:trace_on_genUnit}]
  Since~\(A\) belongs to the thick subcategory generated
  by~\(\C\), it is dualisable in~\(\KK^G\) by
  Proposition~\ref{pro:dualisable_compact}.  Hence
  \(\tr(f)\in\Rep(G)\) is defined for \(f\in\KK^G_0(A,A)\).

  The image of \(\tr(f)\) in~\(S^{-1}\Rep(G)\) is the Lefschetz index
  of the image of~\(f\) in the localisation of~\(\KK^G\) at~\(S\).
  The localisation~\(S^{-1}\Rep(G)\) is a product of fields.  It is
  more convenient to compute each component separately.  This means
  that we localise at larger multiplicatively closed
  subsets~\(\bar{S}\) such that \(\bar{S}^{-1}\Rep(G)\) is one of the
  factors of~\(S^{-1}\Rep(G)\).  In this localisation, the
  endomorphisms of~\(\C\) form a field again, not a product of fields.
  If our trace formula holds for all these localisations, it also
  holds for~\(S^{-1}\Rep(G)\).

  Since the endomorphisms of~\(\C\) form a field, the same argument as
  in the proof of Proposition~\ref{pro:KKG_localisation_loc_cyc} show
  that, in this localisation, \(A\) is isomorphic to a finite sum of
  copies of suspensions of~\(\C\).  Write \(A\cong \bigoplus_{i=1}^n
  A_i\) with \(A_i\cong\C[\varepsilon_i]\) in \(S^{-1}\KK^G\) for some
  \(\varepsilon_i\in\Z/2\).  Then~\(f\) becomes a matrix~\((f_{ij})\)
  with \(f_{ij} \in S^{-1}\KK^G_0(A_j,A_i)\).

  The dual of \(A_i\cong\C[\varepsilon_i]\) is
  \(A_i^*\cong\C[\varepsilon_i] \cong A_i\), and the unit and
  counit of adjunction \(\C \leftrightarrows\C[\varepsilon_i]
  \otimes \C[\varepsilon_i]\) are the canonical isomorphism and
  its inverse with sign \((-1)^{\varepsilon_i}\), respectively;
  the sign is necessary because the exterior product is
  \emph{graded} commutative.  Hence the dual of~\(A\) is
  isomorphic to~\(A\), with unit and counit
  \[
  \C\leftrightarrows A\otimes A
  \cong \bigoplus_{i,j=1}^n A_i\otimes A_j
  \]
  the sum of the canonical isomorphisms \(\C\leftrightarrows
  A_i\oplus A_i\), up to signs, and the zero maps
  \(\C\leftrightarrows A_i\oplus A_j\) for \(i\neq j\).  Thus
  the Lefschetz index of~\(f\) is the sum \(\sum_{i=1}^n
  (-1)^{\varepsilon_i} f_{ii}[\varepsilon_i]\) as an element in
  \(S^{-1}\KK^G_0(\C,\C)\).  This is exactly the supertrace
  of~\(f\) acting on \(S^{-1}\K^G_*(A) \cong \bigoplus_{i=1}^n
  S^{-1}\Rep(G)[\varepsilon_i]\).
\end{proof}

Let~\(G\) be a general compact Lie group.  Let~\(C_G\) denote the set
of conjugacy classes of Cartan subgroups of~\(G\) in the sense of
\cite{Segal:Representation_ring}*{Definition 1.1}.  Such subgroups
correspond bijectively to conjugacy classes of cyclic subgroups in the
finite group~\(G/G^0\), where~\(G^0\) denotes the connected component
of the identity element in~\(G\).  Thus~\(C_G\) is a non-empty, finite
set, and it has a single element if and only if~\(G\) is connected.

The support of a prime ideal~\(\prid\) in~\(\Rep(G)\) is defind
in~\cite{Segal:Representation_ring} as the smallest subgroup~\(H\)
such that~\(\prid\) comes from a prime ideal in~\(\Rep(H)\) via the
restriction map \(\Rep(G)\to\Rep(H)\).  Given any Cartan
subgroup~\(H\), there is a unique minimal prime ideal with
support~\(H\), and this gives a bijection between~\(C_G\) and the set
of minimal prime ideals in~\(\Rep(G)\) (see
\cite{Segal:Representation_ring}*{Proposition 3.7}).

More precisely, if \(H\subseteq G\) is a Cartan subgroup, then~\(H\)
is topologically cyclic and hence \(H\cong \T^r\times \Z/k\) for some
\(r\ge0\), \(k\ge1\).  We described a prime ideal~\((\Phi_k)\)
in~\(\Rep(H)\) before Lemma~\ref{lem:GH_zero_in_localisation}, and its
preimage in~\(\Rep(G)\) is a minimal prime ideal~\(\prid_H\)
in~\(\Rep(G)\).

The total ring of fractions \(S^{-1}\Rep(G)\) is a product of fields
by Lemma~\ref{lem:localised_RepG}.  We can make this more explicit:
\[
S^{-1}\Rep(G) \cong \prod_{H\in C_G} F(\Rep(G)/\prid_H),
\]
where~\(F(\blank)\) denotes the field of fractions for an integral
domain.

\begin{definition}
  \label{def:trace_at_H}
  Let~\(A\) be dualisable in \(\Boot^G\subseteq\KK^G\), let
  \(\varphi\in\KK^G_0(A,A)\), and let \(H\in C_G\).  Let \(F\defeq
  F(\Rep(G)/\prid_H)\) and let \(\K_H(A) \defeq \K^H_*(A)
  \otimes_{\Rep(H)} F\), considered as a \(\Z/2\)-graded \(F\)\nb-vector
  space.  Let \(\K_H(\varphi)\) be the grading-preserving \(F\)\nb-linear
  endomorphism of~\(\K_H(A)\) induced by~\(\varphi\).
\end{definition}

\begin{theorem}
  \label{the:trace_compact_Lie_group}
  Let~\(A\) be dualisable in \(\Boot^G\subseteq\KK^G\), let
  \(\varphi\in\KK^G_0(A,A)\), and let \(H\in C_G\).  Then the image of
  \(\tr(\varphi)\) in \(F(\Rep(G)/\prid_H)\) is the supertrace of
  \(\K_H(\varphi)\).
\end{theorem}

\begin{proof}
  The map \(\Rep(G)\to F(\Rep(G)/\prid_H)\) factors through the
  restriction homomorphism \(\Rep(G)\to\Rep(H)\) because~\(\prid_H\)
  is supported in~\(H\).  Restricting the group action to~\(H\) maps
  the bootstrap category in~\(\KK^G\) into the bootstrap category
  in~\(\KK^H\) by Corollary~\ref{cor:res_ind_bootstrap}, and
  commutes with taking Lefschetz indices because restriction is a tensor
  functor.  Hence we may
  replace~\(G\) by~\(H\) and take \(\varphi\in\KK^H_0(A,A)\) throughout.

  Since~\(H\) is topologically cyclic,
  Proposition~\ref{pro:KKG_localisation_loc_cyc} applies.  It shows
  that in the localisation of~\(\KK^H\) at~\(\prid_H\), any dualisable
  object in~\(\Boot^G\) becomes isomorphic to a finite direct sum of
  suspensions of~\(\C\).  Now the argument continues as in the proof
  of Theorem~\ref{the:trace_on_genUnit} above.
\end{proof}

\section{Hattori--Stallings traces}
\label{sec:Hattori-Stallings}

Before we found the above approach through localisation, we
developed a different trace formula where, in the case of a Hodgkin
Lie group, the trace is identified with the Hattori--Stallings trace
of the \(\Rep(G)\)-module map~\(\K^G_*(f)\) on~\(\K^G_*(A)\).  We
briefly sketch this alternative formula here, although the
localisation approach above seems much more useful for computations.
The Hattori--Stallings trace has the advantage that it obviously
belongs to~\(\Rep(G)\).

We work in the general setting of a tensor triangulated category
\((\Tri,\otimes,\Unit)\).  We assume that~\(\Tri\) satisfies
additivity of traces, that is:

\begin{assumption}
  \label{ass:additivity_trace}
  Let \(A\to B\to C\to A[1]\) be an exact triangle in~\(\Tri\) and
  assume that \(A\) and~\(B\) are dualisable.  Assume also that the
  left square in the following diagram
  \[
  \begin{tikzpicture}[baseline=(current bounding box.west)]
    \matrix (m) [cd]{
      A&B&C&A[1]\\
      A&B&C&A[1]\\
    };
    \begin{scope}[cdar]
      \draw (m-1-1) -- (m-1-2);
      \draw (m-1-2) -- (m-1-3);
      \draw (m-1-3) -- (m-1-4);
      \draw (m-2-1) -- (m-2-2);
      \draw (m-2-2) -- (m-2-3);
      \draw (m-2-3) -- (m-2-4);
      \draw (m-1-1) -- node {$f_A$} (m-2-1);
      \draw (m-1-2) -- node {$f_B$} (m-2-2);
      \draw[dotted] (m-1-3) -- node {$f_C$} (m-2-3);
      \draw (m-1-4) -- node {$f_A[1]$} (m-2-4);
    \end{scope}
  \end{tikzpicture}
  \]
  commutes.  Then~\(C\) is dualisable and there is an arrow
  \(f_C\colon C\to C\) such that the
  whole diagram commutes and \(\tr(f_C) - \tr(f_B) + \tr(f_A) = 0\).
\end{assumption}

Additivity of traces holds in the bootstrap category
\(\Boot^G\subseteq\KK^G\).  The quickest way to check this is the
localisation formula for the trace in
Theorem~\ref{the:trace_compact_Lie_group}.  It shows that~\(\Boot^G\)
satisfies even more: \(\tr(f_C) - \tr(f_B) + \tr(f_A) = 0\) holds for
\emph{any} arrow~\(f_C\) that makes the diagram commute.

There are several more direct ways to verify additivity of traces, but
all require significant work which we do not want to get into here.
The axioms worked out by J.~Peter May in~\cite{May:Additivity} are
lengthy and therefore rather unpleasant to check by hand.  In a
previous manuscript we embedded the localising subcategory
of~\(\KK^G\) generated by~\(\C\) into a category of module spectra.
Since additivity is known for categories of module spectra, this
implies the required additivity result at least for this smaller
subcategory.  Another way would be to show that additivity of traces
follows from the derivator axioms and to embed~\(\KK^G\) into a
triangulated derivator.

In the following, we will just assume additivity of traces and use it
to compute the trace.  Let
\[
R\defeq \Tri_*(\Unit,\Unit) = \bigoplus_{n\in\Z} \Tri_n(\Unit,\Unit)
\]
be the graded endomorphism ring of the tensor unit.  It is
graded-commutative provided~\(\Tri\) satisfies some very basic
compatibility axioms; see~\cite{Suarez-Alvarez:Comm} for
details.

If~\(A\) is any object of~\(\Tri\), then \(M(A)\defeq \Tri_*(\Unit,A)
= \bigoplus_{n\in\Z} \Tri_n(\Unit,A)\) is an \(R\)\nb-module in a
canonical way, and an endomorphism \(f\in\Tri_n(A,A)\) yields a
degree-\(n\) endomorphism \(M(f)\) of \(M(A)\).  We will prove in
Theorem~\ref{the:HS_trace} below that, under some assumptions, the
trace of~\(f\) equals the Hattori--Stallings trace of \(M(f)\) and, in
particular, depends only on \(M(f)\).

Before we can state our theorem, we must define the Hattori--Stallings
trace for endomorphisms of graded modules over graded rings.  This is
well-known for ungraded rings (see~\cite{Bass:Euler_discrete}).  The
grading causes some notational overhead.  Let~\(R\) be a (unital)
graded-commutative graded ring.  A finitely generated free
\(R\)\nb-module is a direct sum of copies of \(R[n]\), where~\(R[n]\)
denotes~\(R\) with degree shifted by~\(n\), that is
\(R[n]_i=R_{n+i}\).  Let \(F\colon P\to P\) be
a module endomorphism of such a free module, let us assume that~\(F\)
is homogeneous of degree~\(d\).  We use an isomorphism
\begin{equation}
  \label{eq:iso_free}
  P\cong\bigoplus_{i=1}^r R[n_i]
\end{equation}
to rewrite~\(F\) as a matrix \((f_{ij})_{1\le i,j\le r}\) with
\(R\)\nb-module homomorphisms \(f_{ij}\colon R[n_j]\to R[n_i]\) of
degree~\(d\).  The entry~\(f_{ij}\) is given by right multiplication
by some element of~\(R\) of degree \(n_i-n_j+d\).  The (super)trace
\(\tr F\) is defined as
\[
\tr F \defeq \sum_{i=1}^r (-1)^{n_i} \tr f_{ii} ;
\]
this is an element of~\(R\) of degree~\(d\).

It is straightforward to check that~\(\tr F\) is well-defined, that
is, independent of the choice of the isomorphism
in~\eqref{eq:iso_free}.  Here we use that the degree-zero part
of~\(R\) is central in~\(R\) (otherwise, we still get a well-defined
element in the commutator quotient \(R_d/[R_d,R_0]\)).  Furthermore,
if we shift the grading on~\(P\) by~\(n\), then the trace is
multiplied by the sign \((-1)^n\) -- it is a supertrace.

If~\(P\) is a finitely generated projective graded \(R\)\nb-module,
then \(P\oplus Q\) is finitely generated and free for some~\(Q\), and
for an endomorphism~\(F\) of~\(P\) we let
\[
\tr F \defeq \tr (F\oplus 0\colon P\oplus Q\to P\oplus Q).
\]
This does not depend on the choice of~\(Q\).

A \emph{finite projective resolution} of a graded
\(R\)\nb-module~\(M\) is a resolution
\begin{equation}
  \label{eq:finite_resolution}
  \dotsb \to P_\ell \xrightarrow{d_\ell}
  P_{\ell-1} \xrightarrow{d_{\ell-1}}
  \dotsb \xrightarrow{d_1} P_0 \xrightarrow{d_0} M
\end{equation}
of finite length by finitely generated projective graded
\(R\)\nb-modules~\(P_j\).  \emph{We assume that the maps~\(d_j\) have
  degree one} (or at least odd degree).  Assume that~\(M\) has such a
resolution and let \(f\colon M\to M\) be a module homomorphism.
Lift~\(f\) to a chain map \(f_j\colon P_j\to P_j\),
\(j=0,\dotsc,\ell\).  We define the \emph{Hattori--Stallings trace}
of~\(f\) as
\[
\tr(f) = \sum_{j=0}^\ell \tr(f_j).
\]
It may be shown that this trace does not depend on the choice of
resolution.  It is important for this that we choose~\(d_j\) of degree
one.  Since shifting the degree by one alters the sign of the trace of
an endomorphism, the sum in the definition of the trace becomes an
\emph{alternating} sum when we change conventions to have even-degree
boundary maps~\(d_j\).  Still the trace changes sign when we shift the
degree of~\(M\).

\begin{theorem}
  \label{the:HS_trace}
  Let \(F\in\Tri(A,A)\) be an endomorphism of some object~\(A\)
  of~\(\Tri\).  Assume that~\(A\) belongs to the localising
  subcategory of~\(\Tri\) generated by~\(\Unit\).  If the graded
  \(R\)\nb-module \(M(A)\defeq \Tri_*(\Unit,A)\) has a finite
  projective resolution, then~\(A\) is dualisable in~\(\Tri\) and the
  trace of~\(F\) is equal to the Hattori--Stallings trace of the
  induced module endomorphism \(\Tri_*(\Unit,f)\) of~\(M(A)\).
\end{theorem}

\begin{proof}
  Our main tool is the phantom tower over~\(A\), which is constructed
  in~\cite{Meyer:Homology_in_KK_II}.  We recall some details of this
  construction.

  Let~\(M^\bot\) be the functor from finitely generated projective
  \(R\)\nb-modules to~\(\Tri\) defined by the adjointness property
  \(\Tri(M^\bot(P),B) \cong \Tri(P,M(B))\) for all \(B\in\Tri\).
  The functor~\(M^\bot\) maps the free rank-one module~\(R\)
  to~\(\Unit\), is additive, and commutes with suspensions; this
  determines~\(M^\bot\) on objects.  Since \(R=\Tri_*(\Unit,\Unit)\),
  \(\Tri_*(M^\bot(P_1),M^\bot(P_2))\) is isomorphic (as a graded
  Abelian group) to the space of \(R\)\nb-module homomorphisms
  \(P_1\to P_2\).  Furthermore, we have canonical isomorphisms
  \(M\bigl(M^\bot(P)\bigr) \cong P\) for all finitely generated
  projective \(R\)\nb-modules~\(P\).

  By assumption, \(M(A)\) has a finite projective resolution as
  in~\eqref{eq:finite_resolution}.  Using~\(M^\bot\), we lift
  it to a chain complex in~\(\Tri\), with entries \(\hat{P}_j
  \defeq M^\bot(P_j)\) and boundary maps \(\hat{d}_j\defeq
  M^\bot(d_j)\) for \(j\ge1\).  The map \(\hat{d}_0\colon
  \hat{P_0}\to A\) is the pre-image of~\(d_0\) under the
  adjointness isomorphism \(\Tri(M^\bot(P),B) \cong
  \Tri(P,M(B))\).  We get back the resolution of modules by
  applying~\(M\) to the chain complex
  \((\hat{P}_j,\hat{d}_j)\).

  Next, it is shown in~\cite{Meyer:Homology_in_KK_II} that we may
  embed this chain complex into a diagram
  \begin{equation}
    \label{eq:phantom_tower}
    \begin{tikzpicture}[baseline=(current bounding box.west)]
      \matrix (m) [cd,column sep=1.5em]{
        A=N_0&&N_1&&N_2&&N_3&&\cdots\\
        &\hat{P}_0&&\hat{P}_1&&\hat{P}_2&&\hat{P}_3&&\cdots\\
      };
      \begin{scope}[cdar]
        \draw (m-1-1) -- node {$\iota_0^1$} (m-1-3);
        \draw (m-1-3) -- node {$\iota_1^2$} (m-1-5);
        \draw (m-1-5) -- node {$\iota_2^3$} (m-1-7);
        \draw (m-1-7) -- (m-1-9);

        \draw[odd] (m-1-3) -- node[swap] {$\varepsilon_0$} (m-2-2);
        \draw[odd] (m-1-5) -- node[swap] {$\varepsilon_1$} (m-2-4);
        \draw[odd] (m-1-7) -- node[swap] {$\varepsilon_2$} (m-2-6);
        \draw[odd] (m-1-9) -- (m-2-8);

        \draw (m-2-2) -- node {$\hat{d}_0=  \pi_0$} (m-1-1);
        \draw (m-2-4) -- node {$\pi_1$} (m-1-3);
        \draw (m-2-6) -- node {$\pi_2$} (m-1-5);
        \draw (m-2-8) -- node {$\pi_3$} (m-1-7);

        \draw[odd] (m-2-4) -- node {$\hat{d}_1$} (m-2-2);
        \draw[odd] (m-2-6) -- node {$\hat{d}_2$} (m-2-4);
        \draw[odd] (m-2-8) -- node {$\hat{d}_3$} (m-2-6);
        \draw[odd] (m-2-10) -- (m-2-8);
      \end{scope}
    \end{tikzpicture}
  \end{equation}
  where the wriggly lines are maps of degree one; the triangles
  involving~\(\hat{d}_j\) commute; and the other triangles are
  exact.  This diagram is called the \emph{phantom tower}
  in~\cite{Meyer:Homology_in_KK_II}.

  Since \(\hat{P}_j=0\) for \(j>\ell\), the maps~\(\iota_j^{j+1}\) are
  invertible for \(j>\ell\).  Furthermore, a crucial property of the
  phantom tower is that these maps~\(\iota_j^{j+1}\) are \emph{phantom
    maps}, that is, they induce the zero map on
  \(\Tri_*(\Unit,\blank)\).  Together, these facts imply that
  \(M(N_j)=0\) for \(j>\ell\).  Since we assumed~\(\Unit\) to be a
  generator of~\(\Tri\), this further implies \(N_j=0\) for
  \(j>\ell\).
  Therefore, $A\in \langle \Unit\rangle$, so that~$A$ is
  dualisable as claimed.

  Next we recursively extend the endomorphism~\(F\) of \(A=N_0\) to an
  endomorphism of the phantom tower.  We start with \(F_0=F\colon
  N_0\to N_0\).  Assume \(F_j\colon N_j\to N_j\) has been constructed.
  As in~\cite{Meyer:Homology_in_KK_II}, we may then lift~\(F_j\) to a
  map \(\hat{F}_j\colon \hat{P}_j\to \hat{P}_j\) such that the square
  \[
  \begin{tikzpicture}[baseline=(current bounding box.west)]
      \matrix (m) [cd]{
      \hat{P}_j&N_j\\
      \hat{P}_j&N_j\\
    };
    \begin{scope}[cdar]
      \draw (m-1-1) -- node {$\pi_j$} (m-1-2);
      \draw (m-1-1) -- node {$\hat{F}_j$} (m-2-1);
      \draw (m-2-1) -- node {$\pi_j$} (m-2-2);
      \draw (m-1-2) -- node {$F_j$} (m-2-2);
    \end{scope}
  \end{tikzpicture}
  \]
  commutes.  Now we apply additivity of traces
  (Assumption~\ref{ass:additivity_trace}) to construct an
  endomorphism
  \(F_{j+1}\colon N_{j+1}\to N_{j+1}\) such that
  \((\hat{F}_j,F_j,F_{j+1})\) is a triangle morphism and \(\tr (F_j) =
  \tr (\hat{F}_j) + \tr (F_{j+1})\).  Then we repeat the recursion
  step with~\(F_{j+1}\) and thus construct a sequence of
  maps~\(F_j\).  We get
  \[
  \tr(F)
  = \tr(F_0)
  = \tr(\hat{F}_0) + \tr(F_1)
  = \dotsb
  = \tr(\hat{F}_0) + \dotsb + \tr(\hat{F}_\ell) +  \tr(F_{\ell+1}).
  \]
  Since \(N_{\ell+1}=0\), we may leave out the last term.

  Finally, it remains to observe that the trace
  of~\(\hat{F}_j\) as an endomorphism of~\(\hat{P}_j\) agrees
  with the trace of the induced map on the projective
  module~\(P_j\).  Since both traces are additive with respect
  to direct sums of maps, the case of general finitely
  generated projective modules reduces first to free modules
  and then to free modules of rank one.  Both traces change by
  a sign if we suspend or desuspend once, hence we reduce to
  the case of endomorphisms of~\(\Unit\), which is trivial.
  Hence the computation above does indeed yield the
  Hattori--Stallings trace of~\(M(A)\) as asserted.
\end{proof}

\begin{remark}
  Note that if a module has a finite projective resolution, then it
  must be finitely generated.  Conversely, if the graded ring~\(R\) is
  \emph{coherent} and \emph{regular}, then any finitely generated
  module has a finite projective resolution.  (Regular means that every
  finitely generated module has a finite \emph{length} projective resolution;
  coherent means that every finitely generated homogeneous ideal is
  finitely presented~-- for instance, this holds if~\(R\) is (graded)
  Noetherian; coherence implies that any finitely generated graded
  module has a resolution by finitely generated projectives.)

  Moreover, if~\(R\) is coherent then the finitely presented
  \(R\)-modules form an abelian category, and this implies (by
  an easy induction on the triangular length of~$A$) that for
  every \(A\in \langle \Unit \rangle =(\langle \Unit \rangle_\loc)_d\)
  the module~\(M(A)\) is finitely presented and thus \emph{a
    fortiori} finitely generated.  If~\(R\) is also
  regular, each such~\(M(A)\) has a finite projective resolution.

  In conclusion: if~\(R\) is regular and coherent, an object \(A\in
  \langle \Unit \rangle_\loc\) is dualisable if and only if the
  graded \(R\)\nb-module~\(M(A)\)
  has a finite projective resolution.
\end{remark}

\end{document}